\documentclass[10pt]{article}

\usepackage[a4paper, total={6in, 8in}]{geometry}
\usepackage[T1]{fontenc}
\usepackage[utf8]{inputenc}
\usepackage{comment}
\usepackage{amsfonts}
\usepackage{amsmath}
\usepackage{amsthm}
\usepackage{pgfplots}
\usepackage{tikz}
\usepackage{mathabx}
\usepackage{comment}
\usepackage{graphicx,subcaption}
\usepackage{wrapfig}

\usetikzlibrary{positioning} 
\usetikzlibrary{calc}
\usetikzlibrary{arrows,shapes,backgrounds}

\newtheorem{theorem}{Theorem}[section]
\newtheorem{definition}[theorem]{Definition}
\newtheorem{remark}[theorem]{Remark}
\newtheorem{lemma}[theorem]{Lemma}

\newtheorem{proposition}[theorem]{Proposition}

\begin{document}

\title{On the stability of the ill-posedness of a quasi-linear wave equation in two dimensions with initial data in $H^{7/4} (\ln H^{-\beta})$}

\date{}

\author{Gaspard Ohlmann \\ University of Basel, \emph{gaspard.ohlmann@unibas.ch}}
\maketitle

\begin{center}
	\textbf{Abstract}
\end{center}

\emph{This article is the continuation of \cite{ohlmann2021illposedness} where we exhibited the ill-posedness of a quasi-linear wave equation in dimension $2$ for initial data in $H^{7/4}(\ln H^{7/4})$. Here, we look at modifications of the equation and show that the blow-up phenomenon still occurs. First, we study another equation with the same characteristics but a different underlying ODE. Later, we study an equation where a $x_2$ dependency is introduced. The latter case constitutes the main contribution of this paper, as we are able to show that the solution still behaves in a pathological way without having an explicit formula for neither the characteristics nor the values of the solution. Finally, we study the case where a perturbation of the initial data is introduced. }

\begin{center}
	\textbf{Keywords}
\end{center}

\emph{Partial Differential Equations, Well-posedness, geometric blow-up, Harmonic Analysis, Fractional calculus}

\tableofcontents
	
	\section{Introduction}

We study the well-posedness of quasi-linear wave equations. We will consider the following equation

\begin{equation}\label{QLW}
	\sum_{i,j=0}^n g^{ij}(u,u') \partial_{x^i} \partial_{x^j} u = F(u,u'), \hspace{0.2cm} (t,x) \in S_T = [0,T[\times \mathbb{R}^n,
\end{equation}
where $\partial_{x^0} = \partial_t$ and $G=(g^{ij})$ and $F$ are smooth functions. Also, we assume that $g$ is close to the Minkowski metric $m$ ; i.e.,

\begin{equation}\label{elli}
	\sum_{i,j=0}^n |g^{ij} - m^{ij}| \leq 1/2.
\end{equation}

We will also define the corresponding Cauchy problem, 

\begin{equation}\label{QLW1.5}
	\left\{
	\begin{aligned}
		&\sum_{i,j=0}^n g^{ij}(u,u') \partial_{x^i} \partial_{x^j} u = F(u,u'), \hspace{0.2cm} (t,x) \in S_T = [0,T[\times \mathbb{R}^n,\\
		&(u,\partial_t u)_{|t=0} = (f,h),
	\end{aligned}
	\right.
\end{equation}
where $\partial_{x^0} = \partial_t$ and $G=(g^{ij})$ and $F$ are smooth functions.

Throughout this article, we will consider $H^{s} (\ln H)^{-\beta}$ to be the following function space.

\begin{definition}
	We define the logarithmic Sobolev norm $||f||_{H^s (\ln H)^{-\beta}}$ as 
	
	\begin{equation}
		\begin{aligned}
		||f||_{\dot H^s (\ln H)^{-s}}^2 &= \left| \left| \mathcal{F}(f) \cdot |\xi|^s (1+|\ln(|\xi|)|)^{-\beta} \right| \right|_{L^2}^2\\
		&=\int_{\xi \in \mathbb{R}^d} \left[ \frac{|\xi|^s}{(1 + |\ln(|\xi|)|^\beta)} \int_{x \in \mathbb{R}^d} e^{-2 i \pi \langle x,\xi \rangle} f(x) dx \right]^2 d\xi.
		\end{aligned}
	\end{equation}

	We also define the space $\dot H^s (\ln H)^{-\beta}$ as 
	
	\begin{equation}
		\dot H^s (\ln H)^{-\beta}  = \{ f,~ ||f||_{\dot H^s (\ln H)^{-\beta}} < \infty \}.
	\end{equation}
	
\end{definition}

The intuition is that this space is located in between $\dot H^{s}$ and $\dot H^{s - \lambda}$ for any $\lambda>0$. In fact, we prove in \cite{ohlmann2021illposedness} the following properties.

\begin{proposition}\label{dotHlog}
	The space $\dot H^s (\ln H)^{-\beta}$ satisfies the following properties.
	\begin{enumerate}
		\item[(1)] 
			Let $f$ be a function in $L^1(\mathbb{R}^2)$. For any non-negative $s$ and $\beta$, for any small enough and positive $\lambda$, we have the following properties.
			\begin{itemize}
				\item[(i)] $||f||_{\dot H^{s}} < \infty \Rightarrow ||f||_{\dot H^{s}(\ln H)^{-\beta}} <\infty$,
				\item[(ii)] $||f||_{\dot  H^{s}(\ln H)^{-\beta}} <\infty  \Rightarrow ||f||_{\dot H^{s-\lambda}} < \infty$.
			\end{itemize}
		\item[(2)] Let $f \in L^2$.  
		Let $\beta$ and $s$ be two non-negative real numbers, $\lambda$ a small enough real number and $K$ a compact subset of $\mathbb{R}^2$. If $f$ is supported in $K$, then we have the two following:
		\begin{itemize}
			\item[(i)] $||f||_{\dot H^{s}} < \infty \Rightarrow ||f||_{\dot H^{s}(\ln H)^{-\beta}} <\infty$,
			\item[(ii)] $||f||_{\dot  H^{s}(\ln H)^{-\beta}} <\infty  \Rightarrow ||f||_{\dot H^{s-\lambda}} < \infty$.
		\end{itemize}
		\item[(3)] 	Let $s$ be a non-negative real number, and $f\in L^1_{loc}$. If $f$ belongs to $ \dot H^{s} (\ln H)^{-\beta}$ and is compactly supported, then $f$ belongs to $H^s (\ln H)^{-\beta}$.
	\end{enumerate}

\end{proposition}

We now consider the model equation in dimension $2+1$ and the corresponding Cauchy problem

\begin{equation}\label{modelequation}
	\left\{
	\begin{aligned}
		&\Box u = (D u) D^{2} u,  \\
		&(u,\partial_t u)_{|t=0} = (f,g),
	\end{aligned}
	\right.
\end{equation}
where $D= (\partial_{x_1} - \partial_{t})$.

Note that (\ref{modelequation}) is of the form (\ref{QLW}) with

\begin{equation}\label{metricg}
	g=
	\begin{bmatrix}
		1-v & v & 0  \\
		v & -1-v & 0  \\
		0 & 0 & -1  \\
	\end{bmatrix}, v= Du.
\end{equation}

In \cite{ohlmann2021illposedness}, we show the following theorem, which states that the problem (\ref{QLW1.5}) is ill-posed in $H^{7/4} (\ln H)^{-\beta}$.

\begin{theorem}\label{maingeneral} 
	We consider $\beta>1/2$. There exists initial data $(f,g) \in \dot H^{11/4}(\ln H)^{-\beta} \times \dot H^{7/4}(\ln H)^{-\beta}$ supported on a compact set, with $||f||_{\dot H^{11/4}(\ln H)^{-\beta}} + ||g||_{\dot H^{7/4} (\ln H)^{-\beta}}$ arbitrarily small, such that (\ref{modelequation}) does not have any proper solution $u$ such that 
	
	$(u,\partial_t u) \in C\left( [0,T[; \dot H^{11/4}(\ln H)^{-\beta} (\mathbb{R}^2) \times \dot H^{7/4} (\ln H)^{-\beta}(\mathbb{R}^2) \right)$ for any $T>0$.
	
	In fact, for all $\lambda>0$ small enough, (\ref{modelequation}) does not have any proper solution $u$ such that \\ $(u,\partial_t u) \in C\left( [0,T[; \dot H^{11/4-\lambda} (\mathbb{R}^2) \times \dot H^{7/4-\lambda}(\mathbb{R}^2) \right)$ for any $T>0$.

\end{theorem} 

The blow-up is a geometric blow-up obtained by concentration of the characteristics. 

The goal of the present article, is to consider perturbed versions of (\ref{modelequation}) and show that the problem is still ill-posed. We consider two perturbations, the main contribution being the second one. The method we use allows us to show a pathological behavior of the solution without having an explicit formula, which is why it could be adapted to other equations.

\section{Context and preliminaries}

In this section, we introduce the definitions that we will use. In particular, we introduce the notion of domain and dependence and Huygens principle. We will also introduce the initial condition and its different regularized versions as well as the relevant properties on it, so we can use those in the next sections.

\begin{definition}\label{dependence}
	Let $\Omega \subset \mathbb{R}_+ \times \mathbb{R}^2$
	be an open set equipped with a Lorentzian metric $g_{j,k}$ satisfying (\ref{elli}). It is a domain of dependence for $g$ if the closure of the causal past $\Lambda_{t',x'}$ of each $(t',x') \in \Omega$ is contained in $\Omega$, with $z\in \Lambda_{t',x'}$ if and only if it can be joined to $(t',x')$ by a Lipschitz continuous curve $(t,x(t))$ satisfying 
	\begin{equation}\label{condlight}
		\sum_{i,j=0}^2 g_{i,j} (x) \frac{dx_i}{dt} \frac{dx_j}{dt} \geq 0,
	\end{equation}
	almost everywhere.
\end{definition}

For a domain $\Omega \subseteq \mathbb{R}^{1+n}$, we denote by $\Omega_t$ the set
\begin{equation}
	\Omega_t = \{ (\tau,x) \in \Omega, \hspace{0.2cm} t=\tau  \}.
\end{equation}

Also, for a set $\Omega \subseteq \mathbb{R}^{1+2}$, we denote (the dependence in $\Omega$ is not explicitly written) 

\begin{equation}
	a_t(x_1) = \left| \{ y \in \mathbb{R}, \hspace{0.2cm} (t,x_1,y)\in \Omega  \} \right|.
\end{equation}

We will make free use of the Huygens principle, meaning that for a solution $u$ defined on a domain of dependence $\Omega$ corresponding to the Lorentzian metric involved in the equation, the values of $u$ on the set $\Omega_t$ only depend of the value of $u$ on the set $\Omega_0$. 

We also define the fractional derivative. For $\alpha$ a multi-index (and $|\alpha|=\sum_i \alpha_i$), 

\begin{equation}\label{coming2}
	\frac{\partial^\alpha f}{\partial x^\alpha} (x) = \mathcal{F}^{-1} \left( (2 \pi i)^{|\alpha|} \xi_x^\alpha \cdot \mathcal{F}( f)(\xi) \right)(x).
\end{equation}

We now introduce the initial data and the properties it satisfies that we will use later. We will choose

We choose 

\begin{equation}\label{condcondomega}
	\Omega_0 = \left\{ (x_1,x_2) | x_1 \geq 0, |x_2| \leq \frac{\sqrt{x_1}}{|\ln(x_1)|^\delta} \right\} \bigcap \left[0;\frac{1}{2}\right] \times \left[0;\frac{1}{2}\right],
\end{equation}

and we take $\Omega$ to be the largest domain of dependence (defined in Definition \ref{dependence}) for the metric whose inverse is given by
\begin{equation}\label{invmet}
	\sum_{i,j=0}^2 g^{ij} (t,x) \partial_{x_i} \partial_{x_j} = \partial_t^2 - \sum_{i=1}^2 \partial_{x_i}^2 - v(t,x_1)(\partial_t - \partial_{x_1})^2,
\end{equation}
and such that $\Omega\cap \{ t=0 \}= \Omega_0$. Now, we define the initial data.

\begin{definition}
	Let $\varepsilon >0$, and
	
	\begin{equation}
		\left.
		\begin{aligned}
			&\psi_\varepsilon : \mathbb{R} \to \mathbb{R},
		\end{aligned}
		\right.
	\end{equation}
	
	satisfying the following conditions for some constant $C(k)$,
	
	\begin{equation}
	\left\{
	\begin{aligned} 
		&\psi_\varepsilon(x) = 0, \hspace{0.2cm} x\in [0,\varepsilon/2], \\
		&\psi_\varepsilon(x) = 1, \hspace{0.2cm} x\in [\varepsilon,\infty], \\
		&\forall x, ~ \psi_\varepsilon(x) \in [0,1], \\
		&|\partial^k \psi_\varepsilon(x)| \leq \frac{C(k)}{\varepsilon^k}.
	\end{aligned}
	\right.
\end{equation}
	
We consider three indices, $\alpha,\beta$ and $\delta$ such that $2 \alpha - 2\beta - \delta < -1$, with $\beta > 1/2$ and $\alpha,\delta>0$. Introducing 

\begin{equation}
		\left\{
	\begin{aligned} 
		&\chi_\varepsilon : \mathbb{R}^2 \to \mathbb{R}, \\
		&x \mapsto -\int_{y = 0}^{x_1} \psi_\varepsilon(s) |\ln(y)|^\alpha dy, ~ x\geq 0,\\
		&x \mapsto 0,~ x<0.
	\end{aligned}
	\right.
\end{equation}	
	
	 there exists (shown in \cite{ohlmann2021illposedness}) $h_\varepsilon$, an extension of $\xi_\varepsilon$ from $\Omega_0$ (as defined in (\ref{condcondomega})) to $\mathbb{R}^2$ that is still in $H^{7/4} (\ln H)^{-\beta}$ (uniformly in $\varepsilon$). More precisely, we have
	 
	 \begin{equation}
	 	\left\{
	 	\begin{aligned}
	 	&h_\varepsilon: \mathbb{R}^2 \to \mathbb{R},\\
	 	&h_\varepsilon(x) = \chi_\varepsilon(x_1),~ \forall \in \Omega_0,\\
	 	&||h_\varepsilon||_{H^{7/4} (\ln H)^{-\beta}} \leq C < \infty.
	 	\end{aligned}
	 	\right.
	 \end{equation}
 
 To show the existence of $h_\varepsilon$, we multiply an extension of $\chi_\varepsilon$ by a cutoff that follows the geometry of $\Omega_0$. 
	 
\end{definition}

Then, taking $h_\varepsilon$, we can show, with $ t_\varepsilon < C/|\ln(\varepsilon)|^\alpha$,

\begin{theorem}\label{conj}
	Let $u_\varepsilon$ be a solution of (\ref{modelequation}) with initial data $(0,-h_\varepsilon)$.

	There exists $\delta_\varepsilon>0$ such that,
	with $\psi_{\varepsilon}^1: \mathbb{R} \to \mathbb{R}$ be a $C^\infty$ function satisfying
	
	\begin{equation}
		\left\{
		\begin{aligned}
			&\psi_\varepsilon^1(x) = 1 \text{ for } \phi(t_\varepsilon,\nu_\varepsilon) - \delta_\varepsilon <x<\phi(t_\varepsilon,\nu_\varepsilon) + \delta_\varepsilon \\
			&\psi_\varepsilon^1(x) = 0 \text{ for } \phi(t_\varepsilon,\nu_\varepsilon) + 2 \delta_\varepsilon <x \text{ or } x<\phi(t_\varepsilon,\nu_\varepsilon) - 2 \delta_\varepsilon \\
			&0 < \psi_\varepsilon^1(x) < 1 \text{ elsewhere},
		\end{aligned}
		\right.
	\end{equation}
	
	and $\psi_{\varepsilon}^2: \mathbb{R} \to \mathbb{R}$ be a $C^\infty$ function satisfying
	
	\begin{equation}
		\left\{
		\begin{aligned}
			&\psi_\varepsilon^2(x) = 1 \text{ for } - \delta_\varepsilon <x< \delta_\varepsilon \\
			&\psi_\varepsilon^2(x) = 0 \text{ for }  2 \delta_\varepsilon <x \text{ or } x< - 2 \delta_\varepsilon \\
			&0 < \psi_\varepsilon^2(x) < 1 \text{ elsewhere},
		\end{aligned}
		\right.
	\end{equation}
	
	so that $h_\varepsilon:(t,x_1,x_2) \mapsto v_\varepsilon(t,x_1,x_2) \psi_\varepsilon^1(x_1) \psi_\varepsilon^2(x_2)$ is localized in a square of width $4 \delta_\varepsilon$, cut in half by $x_1 = \phi_\varepsilon(t_\varepsilon,\nu_\varepsilon)$; and such that $h_\varepsilon=v_\varepsilon$ in a square of width $2 \delta_\varepsilon$, cut in half by $x_1 = \nu_{t,\varepsilon}$.
	
	We have, for any $\lambda >0$ small enough,
	
	\begin{equation}
		||h_\varepsilon(t)||_{H^{7/4-\lambda}_{x_1}} \rightarrow \infty \hspace{0.2cm} \text{ as } t\rightarrow t_\varepsilon.
	\end{equation}
	
\end{theorem}

More specifically, we define the initial condition

\begin{proof}
	We describe the main ingredients of the proof here, as we will use some of them later on. One of the main arguments, is that when $x$ belongs to $\Omega$, the smallest domain of dependence generated by $\Omega_0$, the solution only depends on one variable and satisfies 
	
	\begin{equation}\label{dimone}
		\begin{split}
			\left( \left( \partial_t + \partial_{x_1} \right) + v \left( \partial_{x_1} - \partial_{t}\right) \right) \left( \partial_{x_1} -\partial_t \right) u =0 &\\
			u(0,x_1) =0, \hspace{0.4cm} \partial_t u(0,x_1) = -\chi(x_1)& 
		\end{split}
	\end{equation}
	where $v = (\partial_{x_1} - \partial_{t}) u$. This allows us to compute an explicit expression for $v$, obtaining 
	
	\begin{equation}
		\frac{\partial}{\partial t}\left(v(t,\phi(t,y)) \right) = 0\hspace{0.15cm} \Rightarrow \hspace{0.15cm} v(t,\phi(t,y)) = \chi(y) \hspace{0.2cm} \forall t.
	\end{equation}

	where 
	
	\begin{equation}
		\phi(t,y) = y + t \frac{1+\chi(y)}{1-\chi(y)}.
	\end{equation}

	Another important part was to show that the domain of dependence $\Omega$ in not empty when $t \rightarrow \varepsilon$. This is done by comparing the evolution of the size of $\Omega$ and the time $t_\varepsilon$ it takes for the solution to blow up on $\Omega$. Lastly, the estimation of the Sobolev norm uses fractional derivatives and the explicit formula we have for $v$.
	
\end{proof}

The index $s=7/4$ was chosen in response to the following result that has been established by D. Tataru and P. Smith in 2005 in \cite{smith2005sharp}, that provided the best positive result for the well-posedness of quasi-linear wave equations.

\begin{theorem}\cite{smith2005sharp}\label{smithsharp}
	
	We consider the following Cauchy problem. 
	
	\begin{equation}\label{QLW2}
		\left\{
		\begin{aligned}
			&\sum_{i,j} g^{ij}(u) \partial_{x^i} \partial_{x^j} u = \sum_{i,j} q^{ij} (u) \partial_{x^i} u \partial_{x^j} u, \hspace{0.2cm} (t,x) \in S_T = [0,T[\times \mathbb{R}^n,\\
			&(u,\partial_t u)_{|t=0} = (f,g),
		\end{aligned}
		\right.
	\end{equation}
	where $\partial_{x^0} = \partial_t$ and $G=(g^{ij})$ and $Q=(q^{ij})$ are smooth functions. Also, we assume that $g$ is close to the Minkowski metric $m$. 
	
	The Cauchy problem (\ref{QLW2}) is locally well-posed in $H^s \times H^{s-1}$ provided that 
	\begin{equation}
		\begin{aligned}
			s &> \frac{n}{2} + \frac{3}{4} =\frac{7}{4} \hspace{0.3cm} \text{ for } n=2, \\
			s &> \frac{n+1}{2} \hspace{0.35cm} \text{ for } n=3,4,5. 
		\end{aligned}
	\end{equation}
\end{theorem}

The difference between the two indexes, $7/4$ in \cite{smith2005sharp} and $11/4$ in \cite{ohlmann2021illposedness} or here, is due to the fact that here, we allow the metric $g$ to depend not only on $u$, but also on the first derivatives of $u$. The fact that a counter-example in our situation provides a consider example for $s=7/4$ in the class of equations studied by Tataru and Smith is justified in \cite{ohlmann2021illposedness}.

In section \ref{modif1}, we introduce a modification of the equation (\ref{modelequation}) that modifies the underlying ODE when we follow the characteristics. Crossing of the characteristics still occur, and the modification of the ODE does not compensate this, so for any $t_0>0$, we can find an initial condition with uniformly bounded $H^{7/4} (\ln H)^{-\beta}$ norm, such that the solution blows up at $t<t_0$.

In section \ref{modif2}, we introduce  a new term with an unknown $x_2$ dependency. This dependency changes everything, as we can no longer consider a domain where the function only depends on the $x_1$ variable. Hence, the characteristics method can no longer be used. To solve this issue, we introduce $\phi$, a function that is almost a characteristic, and show that the main component of the solution along this line blows up for similar reasons. In this case, we will make the assumption that the solution has a real dependency with respect to $x_2$, as the proof is otherwise similar to the one done in \cite{ohlmann2021illposedness}.

\section{A first modification of the equation that leads to an instantaneous blow-up} \label{modif1}

In this chapter, we consider a modification of our model equation. We will show that the corresponding Cauchy problem (with the same initial data) also leads to a blow-up at time $t=0^+$. The common point between this example and the previous one is the behavior of the characteristic curves, but with a different underlying ODE. Indeed, this time the characteristics seen as a function of $t$ for a fixed initial starting point $(0,x_1,x_2)$ will not be an affine function. We will mainly focus on the parts that make this proof different from the one that has been done in \cite{ohlmann2021illposedness}, and skip the identical parts.

\subsection{Definition of the modified problem and preliminary results}

We consider the Cauchy problem 

\begin{equation}\label{Cau}
\left\{
\begin{aligned}
&\left( \partial_t + \frac{1+v}{1-v} \partial_{x_1} \right) \left( \partial_t - \partial_{x_1} \right) u = c \left( \partial_t - \partial_{x_1} \right) u,\\
&\frac{\partial u_{|t=0}}{\partial t} = -\chi, u_{|t=0} = 0.
\end{aligned}
\right.
\end{equation}

With $v = \left( \partial_{x_1} - \partial_t \right) u \neq 1$. 

We define $\chi_\varepsilon$ as  

\begin{equation}\label{encoreunchie}
\chi_\varepsilon(y) = - \int_{s=0}^y \psi_\varepsilon |\ln(s)|^\alpha ds.
\end{equation}

Where $\psi_\varepsilon: \mathbb{R^+}\to \mathbb{R^+}$ satisfies the following conditions.

\begin{equation}
\left\{
\begin{aligned}
&\psi_\varepsilon(s) = 0,\hspace{0.2cm} \forall s<\frac{\varepsilon}{2},\\
&\psi_\varepsilon(s) = 1,\hspace{0.2cm} \forall s>\varepsilon, \\
&\psi_\varepsilon(s) \in [0,1], \hspace{0.2cm} \forall s\in \mathbb{R}^{+},\\
&|\psi_\varepsilon'|(s) \leq \frac{2}{\varepsilon}, \hspace{0.2cm} \forall s \in \mathbb{R}^+.
\end{aligned}
\right.
\end{equation}

From now on, we will not always write explicitly the dependency of all functions with respect to $\varepsilon$. We will consider, and name, without explicit relabelling, the function $\chi$ constructed in \cite{ohlmann2021illposedness} that is a global extension of the previously defined function. Let us recall here its important properties. The function $\chi$ satisfies

\begin{equation}\label{condcondiniini}
\begin{aligned}
&\chi(x_1,x_2) = \chi_\varepsilon(x_1,x_2) \text{ as in (\ref{encoreunchie}) } ~ \forall (x_1,x_2)\in \Omega_0. \\
& \chi(x_1,x_2) \text{ has a compact support K}\\
& \chi(x_1,x_2) \in H^{7/4} (\ln H)^{-\beta},
\end{aligned}
\end{equation}

with 

\begin{equation}\label{condcondomega2}
	\Omega_0 = \left\{ (x_1,x_2) | x_1 \geq 0, |x_2| \leq \frac{\sqrt{x_1}}{|\ln(x_1)|^\delta} \right\} \bigcap \left[0;\frac{1}{2}\right] \times \left[0;\frac{1}{2}\right].
\end{equation}

Also, we consider $\Omega$ to be the largest domain of dependence such that $\Omega \cap \{t=0\} = \Omega_0$. The properties on $\Omega$ established in \cite{ohlmann2021illposedness} still hold. By definition of $\Omega$, the computations done using the formula (\ref{encoreunchie}) will hold inside $\Omega$.

Now, if we set the condition 

\begin{equation}\label{conditionimposee}
\phi_t = \frac{1+v}{1-v} \text{ and } \phi(0,x)=x,
\end{equation}

Using the chain rule and (\ref{Cau}), we obtain

\begin{equation}
\partial_t \left[ v(t,\phi(t,x)) \right] = c v(t,\phi(t,x)).
\end{equation}

And hence,

\begin{equation}
v(t,\phi(t,x)) = e^{ct} v(0,x) = e^{ct} \chi(x).
\end{equation}

This result combined with (\ref{conditionimposee}) leads to the following equation for $\phi$, the characteristic function.

\begin{equation}\label{pasintegre}
\partial_t \phi(t,x) = \frac{1+e^{ct} \chi(x)}{1-e^{ct} \chi(x)}.
\end{equation}

Integrating (\ref{pasintegre}) with respect to $t$ leads to the following explicit formula for $\phi$.

\begin{multline}\label{philol}
\phi(t,x) = x + \int_{s=0}^t \frac{1+e^{cs} \chi(x)}{1-e^{cs} \chi(x)} ds = x + \int_{u=1}^{e^{ct}} \frac{1+u \chi(x)}{1-u \chi(x)} \frac{du}{cu}\\
= x + \int_{u=1}^{e^{ct}} \left[ \frac{1}{cu} + \frac{2\chi(x)}{c(1-u\chi(x))} \right] du = x + t + \frac{2}{c} \ln\left( \frac{1-\chi(x)}{1-e^{ct} \chi(x)} \right).
\end{multline}

\begin{remark}
For comparison sake, note that in \cite{ohlmann2021illposedness}, $\phi$ can be written in the following form
\begin{equation}\label{previous}
\phi(t,x) = x + t + 2t \frac{\chi(x)}{1-\chi(x)}.
\end{equation}

Also, taking the limit $c\rightarrow 0$, we have that  
\begin{multline}
x+t+\frac{2}{c} \ln\left(\frac{1-\chi(x)}{1-e^{ct} \chi(x) }\right) = x + t + \frac{2}{c} \ln\left(\frac{1}{1-\frac{ct\chi(x)}{1-\chi(x)} + o(c)} \right) \\
= x + t + \frac{2}{c} \ln\left(1 + \frac{ct\chi(x)}{1-\chi(x)} +o(c) \right) = x + t + 2t \frac{\chi(x)}{ 1 - \chi(x)} + o(c).
\end{multline}

Hence, we obtain the previous expression of $\phi$, given by (\ref{previous}).
\end{remark}

Now, we will state the main theorem of this chapter.

\begin{theorem}\label{deuxiemecas}
There exists an initial condition, such that the solution $u$ of the corresponding Cauchy problem (\ref{Cau}) satisfies
\begin{equation}
\left\{
\begin{aligned}
&u_{|t=0} \in H^{11/4} (\ln H)^{-\beta} (\mathbb{R}^2),\\
&u(t,\cdot) \notin H^{11/4}(\ln H)^{-\beta}(\mathbb{R}^2), \hspace{0.2cm} \forall t>0.\\
\end{aligned}
\right.
\end{equation}
\end{theorem}

For this, we consider the following Cauchy problem

\begin{equation}\label{Caueps}
\left\{
\begin{aligned}
&\left( \partial_t + \frac{1+v}{1-v} \partial_{x_1} \right) \left( \partial_t - \partial_{x_1} \right) u = c \left( \partial_t - \partial_{x_1} \right) u,\\
&\frac{\partial u_{|t=0}}{\partial t} = - \chi_\varepsilon, u_{|t=0} = 0,
\end{aligned}
\right.
\end{equation}

where $\chi$ is defined in (\ref{condcondiniini}).

First, we prove the following lemma.

\begin{lemma}\label{ttttt} Let $\phi$ defined as in (\ref{philol}).
There exists $t_\varepsilon >0$ such that for any $t<t_\varepsilon$, and for any $y\in \Omega_t$, $\phi_y(t,y) \neq 0$, and there exists $\nu_\varepsilon$ such that $\phi_y(t_\varepsilon,\nu_\varepsilon)= 0$. Furthermore, $t_\varepsilon \rightarrow 0$ as $\varepsilon \rightarrow 0$.
\end{lemma}

And we will also show the following preliminary results. 

\begin{lemma}

\begin{equation}\label{lesphi}
\begin{aligned}
&\phi_{ty}(t,y) = \frac{2\chi'(y)e^{ct}}{(1-e^{ct}\chi(y))^2}, \hspace{0.3cm} \phi_{tyy}(t,y) = \frac{2e^{ct}\left[ \chi''(y) (1-e^{ct}\chi(y)) + 2\chi'(y)^2 e^{ct} \right]}{(1-e^{ct}\chi(y))^3} \\
&\phi_y(t,y) = 1 + \frac{2}{c} \frac{(e^{ct}-1)\chi'(y)}{(1-\chi(y))(1-e^{ct}\chi(y))}, \hspace{0.3cm}\\ 
&\phi_{yy}(t,y) = \frac{2(e^{ct}-1)}{c} \frac{\left[ 1+ e^{ct} - 2 e^{ct} \chi(y) \right]\chi'^2(y) + (1-\chi(y))(1-e^{ct}\chi(y))\chi''(y)}{(1-\chi(y))^2(1-e^{ct}\chi)}
\end{aligned}
\end{equation}

\begin{equation}\label{lessignes}
\begin{aligned}
&\forall y,\forall t<t_\varepsilon, \hspace{0.2cm} \chi(y) \leq 0,\hspace{0.2cm} \chi'(y) \leq 0,\hspace{0.2cm} \phi_{y}(t,y)>0 \\ 
& \forall y\geq \varepsilon, \forall t<t_\varepsilon, \chi'' (y) >0,\hspace{0.2cm} \phi_{yy}(t,y) > 0, \hspace{0.2cm} \phi_{tyy}(t,y) > 0.
\end{aligned}
\end{equation}

\begin{equation}\label{ennu}
\phi_y(t_\varepsilon,\nu_\varepsilon) = 0, \hspace{0.2cm} \phi_{yy}(t_\varepsilon,\nu_\varepsilon) = 0, \hspace{0.2cm} \phi_{tyy} (t_\varepsilon,\nu_\varepsilon) \neq 0,\hspace{0.2cm} \phi_{yyy}(t_\varepsilon,\nu_\varepsilon) \neq 0.
\end{equation}

\end{lemma}

\begin{proof}
The expressions of $\phi_{ty}$, $\phi_{yy}$, $\phi_{tyy}$ are obtained by differentiation of (\ref{conditionimposee}) with respect to $t$, $y$ and $ty$. The expression of $\phi_{y}$ is obtained by differentiation of (\ref{philol}).

Now, using this expression for $\phi_{ty}$, we obtain

\begin{equation}\label{give}
\phi_{ty}(t,\varepsilon) = \frac{- 2 |\ln(\varepsilon)|^\alpha e^{ct}}{(1-\chi^2(\varepsilon))^2} < 1/5|\ln(\varepsilon)|^\alpha,
\end{equation}

because 

\begin{equation}
|\chi(\varepsilon)| \leq \int_{s=\varepsilon/2}^\varepsilon |\ln(\varepsilon/2)|^\alpha \leq 1/10,  
\end{equation}

for $\varepsilon$ small enough, which leads to (\ref{give}). Also, this means that for $t<1$, and for $\varepsilon$ small enough,
 
\begin{equation}\label{uttt}
(1-\chi(y)) > 4/5, \hspace{0.2cm} (1-e^{ct}\chi(y)) > 4/5.
\end{equation}

Now, we show the existence of $t_\varepsilon$ as in Lemma \ref{ttttt}.

\begin{equation}
\phi_{y}(t,\varepsilon)_{|y=0} = 1, \hspace{0.2cm} \phi_{ty}(t,\varepsilon) < -\frac{1}{5} |\ln(\varepsilon)|^\alpha.
\end{equation}

Now that means that there exists $\tilde{t}_\varepsilon \leq 5 \frac{1}{|\ln(\varepsilon)|^\alpha}$ such that $\phi_{y}(\tilde{t}_\varepsilon,\varepsilon) = 0$.

Hence, the set $\{ t<1 |\exists y\hspace{0.2cm} \phi_y(t,y) \neq 0 \}$ is not empty. By continuity, consider $t_\varepsilon$ its minimum, and call $\nu_\varepsilon$ the corresponding $y$. Then we have 
\begin{equation}\label{teprop}
t_\varepsilon \leq 5 \frac{1}{|\ln(\varepsilon)|^\alpha}, \hspace{0.2cm} \phi_y(t_\varepsilon,\nu_\varepsilon) = 0, \hspace{0.2cm} \forall t<t_\varepsilon, \phi_y(t,y) >0.
\end{equation}

Now, we show (\ref{lessignes}). 

\begin{equation}
\chi(y) = \int_{s=0}^y - \psi_\varepsilon(s) |\ln(s)|^\alpha ds \leq 0. \hspace{0.2cm} \chi'(y) = - \psi_\varepsilon(y) |\ln(y)|^\alpha \leq 0.
\end{equation}

\begin{equation}
\phi_y(t,y)_{|t=0} = 1 >0, \hspace{0.2cm} \phi_y(t,y) \neq 0\hspace{0.2cm} \text{when } t<t_\varepsilon \Rightarrow \phi_y(t,y) >0 \hspace{0.2cm} \text{when } t<t_\varepsilon.
\end{equation}

\begin{equation}\label{primeprime}
\forall y\geq\varepsilon, \hspace{0.2cm} \chi''(y) = - \left[ \frac{-\alpha}{y} (-\ln(y))^{\alpha-1} \right] >0.
\end{equation}

\begin{equation}
\forall y\geq \varepsilon,\hspace{0.2cm} \phi_{yy}(t,y) >0,
\end{equation}

Indeed, we plug in (\ref{primeprime}) in the expression of $\phi_{yy}$ given by (\ref{lesphi}), as well as (\ref{uttt}), and we obtain the result. We do the same for $\phi_{tyy}$ and get $\phi_{tyy}>0$.

Lastly, we prove (\ref{ennu}).

By definition of $(t_\varepsilon,\nu_\varepsilon)$, it is clear that $\phi_y(t_\varepsilon,\nu_\varepsilon) = 0$. Now, if there exists $y_\varepsilon$ such that $\phi_{y}(t_\varepsilon,y_\varepsilon) <0$, by continuity, we would have the existence of $t_\varepsilon'$ and $y_\varepsilon'$ such that $\phi_y(t_\varepsilon',y_\varepsilon')=0$ which would contradict the definition of $t_\varepsilon$. Hence, $\phi_y(t_\varepsilon,\nu_\varepsilon)$ is a maximum of $\phi_y$ (for a fixed $t$) and we have $\phi_{yy}(t_\varepsilon,\nu_\varepsilon)=0$.

Now, we consider a root $\tilde{\nu}_\varepsilon$ of $\phi_{tyy}(t_\varepsilon,\cdot)$. we have that
\begin{equation}\label{intint}
\phi_{tyy} (t_\varepsilon,\tilde{\nu}_\varepsilon) = 0 \Rightarrow \chi''(\tilde \nu_\varepsilon) = \frac{-2\chi'(\tilde \nu_\varepsilon)^2 e^{ct_\varepsilon}}{(1-e^{ct_\varepsilon}\chi(\tilde \nu_\varepsilon))}.
\end{equation}

Plugging (\ref{intint}) in the expression of $\phi_{yy}$ given by (\ref{lesphi}), we obtain

\begin{equation}
\phi_{yy}(t_\varepsilon,\tilde{\nu}_\varepsilon) = -\frac{2(e^{ct_\varepsilon}-1)^2}{c} \frac{\chi'^2(\tilde \nu_\varepsilon)}{(1-\chi(\tilde \nu_\varepsilon))^2 (1- e^{ct_\varepsilon}\chi(\tilde \nu_\varepsilon))}
\end{equation}

Using (\ref{uttt}) again, we obtain $\phi_{yy} (t_\varepsilon, \tilde{\nu}_\varepsilon) \neq 0$. Then, using $\phi_{tyy} = 0 \Rightarrow \phi_{yy} \neq 0$, we obtain by contraposition $\phi_{yy} = 0 \Rightarrow \phi_{tyy} \neq 0$. 

Now, because for $y>\varepsilon$, $\phi_{yy}>0$, it means that $\nu_\varepsilon<\varepsilon$. We can choose $\psi_\varepsilon$ such that $\phi_{yyy}(t_\varepsilon,\nu_\varepsilon) \neq 0$.

\end{proof}

Lastly, our last preliminary work will be to prove the following estimates, which will hold when $y-\nu_\varepsilon$ is small enough, and $t$ is close enough to $t_\varepsilon$.

\begin{equation}\label{lesestimatesdesphis}
\begin{aligned}
&\exists C_\varepsilon^1, C_\varepsilon^2>0, \hspace{0.2cm} -C_\varepsilon^1(y-\nu_\varepsilon)^2 \leq \phi_y(t_\varepsilon,y) \leq C_\varepsilon^2(y-\nu_\varepsilon)^2\\
&\exists C_\varepsilon^1, C_\varepsilon^2>0, \hspace{0.2cm} C_\varepsilon^1(y-\nu_\varepsilon)^2 + C_\varepsilon^1 (t_\varepsilon-t) \leq \phi_y(t,y) \leq C_\varepsilon^2(y-\nu_\varepsilon)^2 + C_\varepsilon^2 (t_\varepsilon-t)\\
&\exists C_\varepsilon^1, C_\varepsilon^2>0, \hspace{0.2cm} C_\varepsilon^1(\nu_\varepsilon-y) \leq \phi_{yy}(t_\varepsilon,y) \leq C_\varepsilon^2 (\nu_\varepsilon-y)
\end{aligned}
\end{equation}

These results are quickly obtained using Taylor expansions and (\ref{lessignes}). Now, we do not have a lower bound for $\phi_{yy}$. 

On the right interval, we have that

\begin{equation}\label{phiyy9}
C_\varepsilon^1 (\nu_\varepsilon - y) - C_\varepsilon^1 (t_\varepsilon-t) \leq \phi_{yy}(t,y) \leq C_\varepsilon^2 (\nu_\varepsilon - y) - C_\varepsilon^2 (t_\varepsilon-t),
\end{equation}

this will provide an upper bound for the norm but no lower bound, since the sign of the two expressions are different.

\subsection{Proof of the blow-up as $t \rightarrow t_\varepsilon$}

Our next goal is to prove theorem \ref{deuxiemecas}. 

We will proceed by applying a cutoff to the solution around the $x$ corresponding to $\nu_\varepsilon$, i.e. $x=\phi(t_\varepsilon,\nu_\varepsilon)$.

First, we start by defining $h_\varepsilon$, as well as the cutoff functions.

\begin{definition}
Let $\psi_{\varepsilon}^1: \mathbb{R} \to \mathbb{R}$ be a $C^\infty$ function satisfying

\begin{equation}\label{proppsi}
\left\{
\begin{aligned}
&\psi_\varepsilon^1(x) = 1 \text{ for } \phi(t_\varepsilon,\nu_\varepsilon) - \delta_\varepsilon <x<\phi(t_\varepsilon,\nu_\varepsilon) + \delta_\varepsilon \\
&\psi_\varepsilon^1(x) = 0 \text{ for } \phi(t_\varepsilon,\nu_\varepsilon) + 2 \delta_\varepsilon <x \text{ or } x<\phi(t_\varepsilon,\nu_\varepsilon) - 2 \delta_\varepsilon \\
&0 < \psi_\varepsilon^1(x) < 1 \text{ elsewhere},
\end{aligned}
\right.
\end{equation}

and $\psi_{\varepsilon}^2: \mathbb{R} \to \mathbb{R}$ be a $C^\infty$ function satisfying

\begin{equation}\label{proppsi2}
\left\{
\begin{aligned}
&\psi_\varepsilon^2(x) = 1 \text{ for } - \delta_\varepsilon <x< \delta_\varepsilon \\
&\psi_\varepsilon^2(x) = 0 \text{ for }  2 \delta_\varepsilon <x \text{ or } x< - 2 \delta_\varepsilon \\
&0 < \psi_\varepsilon^2(x) < 1 \text{ elsewhere},
\end{aligned}
\right.
\end{equation}

so that $h_\varepsilon:(x_1,x_2) \mapsto v_\varepsilon(t,x_1,x_2) \psi_\varepsilon^1(x_1) \psi_\varepsilon^2(x_2)$ is localized in a square of width $4 \delta_\varepsilon$, cut in half by $x_1 = \phi(t_\varepsilon,\nu_\varepsilon)$; and such that $h_\varepsilon=v_\varepsilon$ in a square of width $2 \delta_\varepsilon$, cut in half by $x_1 = \nu_{t,\varepsilon}$.

\end{definition}

Also, we now define $I_\varepsilon(t)$, the integral that will diverge at $t=t_\varepsilon$, thus proving the blow-up of $||u_\varepsilon(t,\cdot)||_{H^{11/4}(\mathbb{R}^2)}$.

\begin{multline}
I_\varepsilon(t) = ||h_\varepsilon(t,\cdot)||_{\dot H_{x_1}^{7/4}} =  \int_{x_1=\phi(t_\varepsilon,\nu_\varepsilon) - 2 \delta_\varepsilon }^{\phi(t_\varepsilon,\nu_\varepsilon) + 2 \delta_\varepsilon } 
 \int_{x_2=-2\delta_\varepsilon}^{2 \delta_\varepsilon} \left( \frac{\partial^2 (v \psi_{\varepsilon}^1 \psi_\varepsilon^2)}{\partial x_1^2} \right)(t,x_1,x_2) \cdot \\
\left[
\int_{y = \phi(t_\varepsilon,\nu_\varepsilon) - 2 \delta_\varepsilon }^{\phi(t_\varepsilon,\nu_\varepsilon) + 2 \delta_\varepsilon } |x_1-y|^{-1/2+2\lambda} \left( \frac{\partial^2 (v \psi_{\varepsilon}^1 \psi_\varepsilon^2)}{\partial x_1^2} \right)(t,y,x_2) dy \right] dx_2 dx_1 
.
\end{multline}

We compute the involved derivatives of $v$.

\begin{equation}\label{vr2}
\begin{aligned}
&v(t,\phi(t,y)) = \chi(y) e^{ct} \\
&v_{x} (t,\phi(t,y)) = \frac{e^{ct} \chi'(y)}{ \phi_y(t,y)}\\
&v_{xx} (t,\phi(t,y)) = \frac{e^{ct} \left( \chi''(y) \phi_y(t,y) - \chi'(y) \phi_{yy}(t,y) \right)}{(\phi_{y}(t,y))^3}
\end{aligned}
\end{equation}

We have the following estimation 

\begin{equation}\label{remai2}
\begin{aligned}
&h_\varepsilon(t,x_1,x_2) = \psi^1_{\varepsilon}(x_1) \psi_\varepsilon^2 (x_2) v_\varepsilon(t,x), \\
&\frac{\partial h_\varepsilon}{\partial_{x_1}}(t,x_1,x_2) = \psi_\varepsilon^2(x_2) \big[ \psi_\varepsilon^{1,'} (x_1) v_\varepsilon(t,x_1) + \psi_\varepsilon^1(x_1) v_{\varepsilon,x}(t,x_1), \big]\\
&\frac{\partial^2 h_\varepsilon}{\partial_{x_1}^2}(t,x_1,x_2) \psi_{\varepsilon}^2(x_2) \big[ \psi_\varepsilon^{1,''}(x_1) v_\varepsilon(t,x_1) + 2 \psi_\varepsilon^{1,'}(x_1) v_{\varepsilon,x}(t,x_1) + \psi_\varepsilon^1(x_1) v_{\varepsilon,xx}(t,x_1) \big],\\
&\Rightarrow \exists C_1,C_2>0, \hspace{0.2cm} C_1 \psi_\varepsilon^2(x_2) v_{\varepsilon,xx}(t,x_1) \leq \frac{\partial^2 h_\varepsilon}{\partial_{x_1}^2}(t,x_1,x_2) \leq C_2 \psi_\varepsilon^2(x_2) v_{\varepsilon,xx}(t,x_1) \hspace{0.1cm} (i) \\
&\Rightarrow \exists C_1>0, \hspace{0.2cm} \left|\frac{\partial^2 h_\varepsilon}{\partial_{x_1}^2}(t,x_1,x_2)\right| \leq C_1 \left|\psi_\varepsilon^2(x_2) v_{\varepsilon,xx}(t,x_1)\right|, \hspace{0.1cm} (ii) \\
\end{aligned}
\end{equation}

Now, with $K_\varepsilon = \int_{x_2=-2\delta_\varepsilon}^{2\delta_\varepsilon} \left( \psi_\varepsilon(x_2) \right)^2  dx_2 $,

\begin{multline}
I_\varepsilon(t) =   K_\varepsilon \int_{x_1 = \phi(t_\varepsilon,\nu_\varepsilon) - 2 \delta_\varepsilon}^{\phi(t_\varepsilon,\nu_\varepsilon) + 2 \delta_\varepsilon}  \left( \frac{\partial^2 (v \psi_{\varepsilon}^1)}{\partial x_1^2} \right)(t,x_1,x_2) \cdot \\
\left[
\int_{y = \phi(t_\varepsilon,\nu_\varepsilon) - 2 \delta_\varepsilon }^{\phi(t_\varepsilon,\nu_\varepsilon) + 2 \delta_\varepsilon } |x_1-y|^{-1/2+2\lambda} \left( \frac{\partial^2 (v \psi_{\varepsilon}^1)}{\partial x_1^2} \right)(t,y,x_2) dy \right] dx_2 dx_1 .
\end{multline}

We now split the first integral's domain into three, corresponding to the three following integration domains. 

\begin{equation}
\begin{aligned}
I_\varepsilon(t)
=& K_\varepsilon \int_{x_1=\phi(t_\varepsilon,\nu_\varepsilon) - 2 \delta_\varepsilon }^{\phi(t_\varepsilon,\nu_\varepsilon) - \delta_\varepsilon }  \left( \frac{\partial^2 (v \psi_{\varepsilon}^1)}{\partial x_1^2} \right)(t,x_1) \\
& \cdot  \int_{y = \phi(t_\varepsilon,\nu_\varepsilon) - 2 \delta_\varepsilon }^{\phi(t_\varepsilon,\nu_\varepsilon) + 2 \delta_\varepsilon } |x_1-y|^{-1/2+2\lambda} \left( \frac{\partial^2 (v \psi_{\varepsilon}^1)}{\partial x_1^2} \right)(t,y) dy dx_1 \\
&+ K_\varepsilon \int_{x_1=\phi(t_\varepsilon,\nu_\varepsilon) - \delta_\varepsilon }^{\phi(t_\varepsilon,\nu_\varepsilon) + \delta_\varepsilon } \left( \frac{\partial^2 v}{\partial x_1^2} \right)(t,x_1) \\
& \cdot  \int_{y = \phi(t_\varepsilon,\nu_\varepsilon) - 2 \delta_\varepsilon }^{\phi(t_\varepsilon,\nu_\varepsilon) + 2 \delta_\varepsilon } |x_1-y|^{-1/2+2\lambda} \left( \frac{\partial^2 v}{\partial x_1^2} \right)(t,y) dy dx_1 \\
&+ K_\varepsilon \int_{x_1=\phi(t_\varepsilon,\nu_\varepsilon) + \delta_\varepsilon }^{\phi(t_\varepsilon,\nu_\varepsilon) + 2 \delta_\varepsilon }  \left( \frac{\partial^2 (v \psi_{\varepsilon}^1)}{\partial x_1^2} \right)(t,x_1) \\
& \cdot  \int_{y = \phi(t_\varepsilon,\nu_\varepsilon) - 2 \delta_\varepsilon }^{\phi(t_\varepsilon,\nu_\varepsilon) + 2 \delta_\varepsilon }|x_1-y|^{-1/2+2\lambda} \left( \frac{\partial^2 (v \psi_{\varepsilon}^1)}{\partial x_1^2} \right)(t,y) dy dx_1. \\
\end{aligned}
\end{equation}

The first difference with the computation made in \cite{ohlmann2021illposedness} is that the expressions of $\phi$, $\phi_y$ and $\phi_{yy}$ are not the same since they depend on the value of $v$, which itself depends on the underlying ODE.
The second difference is that the value of $v$ is not the same because of the underlying ODE, which will have an impact on the computations when we will show the blow-up. Fortunately, it does not make a big difference in the computations.

In the following computations, we use the function $\zeta_\varepsilon^i(t)$ that will be defined in the coming lemma.

\begin{equation}
\begin{aligned}
I_\varepsilon(t)
= K_\varepsilon \int_{y_1=\zeta_\varepsilon^1(t)}^{\zeta_\varepsilon^2(t)} & \phi_{\varepsilon,y}(t,y_1) \left( \frac{\partial^2 (v \psi_{\varepsilon}^1)}{\partial x_1^2} \right)(t,\phi_\varepsilon(t,y_1))\\
 \cdot  \int_{y_2 = \zeta_\varepsilon^1(t)}^{\zeta_\varepsilon^4(t)}& |\phi_\varepsilon(t,y_1)- \phi_\varepsilon(t,y_2)|^{-1/2+2\lambda} \phi_{\varepsilon,y}(t,y_2) \left( \frac{\partial^2 (v \psi_{\varepsilon}^1)}{\partial x_1^2} \right)(t,\phi_\varepsilon(t,y_2)) dy dx_1 \\
+ K_\varepsilon \int_{y_1=\zeta_\varepsilon^2(t)}^{\zeta_\varepsilon^3(t)} & \phi_{\varepsilon,y}(t,y_1) \left( \frac{\partial^2 (v)}{\partial x_1^2} \right)(t,\phi_\varepsilon(t,y_1))\\
 \cdot  \int_{y_2 = \zeta_\varepsilon^1(t)}^{\zeta_\varepsilon^4(t)}& |\phi_\varepsilon(t,y_1)- \phi_\varepsilon(t,y_2)|^{-1/2+2\lambda} \phi_{\varepsilon,y}(t,y_2) \left( \frac{\partial^2 (v)}{\partial x_1^2} \right)(t,\phi_\varepsilon(t,y_2)) dy dx_1 \\
+ K_\varepsilon \int_{y_1=\zeta_\varepsilon^3(t)}^{\zeta_\varepsilon^4(t)} & \phi_{\varepsilon,y}(t,y_1) \left( \frac{\partial^2 (v \psi_{\varepsilon}^1)}{\partial x_1^2} \right)(t,\phi_\varepsilon(t,y_1))\\
 \cdot  \int_{y_2 = \zeta_\varepsilon^1(t)}^{\zeta_\varepsilon^4(t)}& |\phi_\varepsilon(t,y_1)- \phi_\varepsilon(t,y_2)|^{-1/2+2\lambda} \phi_{\varepsilon,y}(t,y_2) \left( \frac{\partial^2 (v \psi_{\varepsilon}^1)}{\partial x_1^2} \right)(t,\phi_\varepsilon(t,y_2)) dy dx_1 \\ \\
= K_\varepsilon (I_\varepsilon^1(t) + I_\varepsilon^2(t) +& I_\varepsilon^3(t)).
\end{aligned}
\end{equation}

Here, the strategy is to show that $I_2(t) >> |I_1(t)| + |I_3(t)| $, and that $I_2(t) \rightarrow \infty$ as $t\rightarrow t_\varepsilon$. At this point, the proof is extremely similar to the proof that we have already done in \cite{ohlmann2021illposedness}, as the only difference is the term $e^{ct}$ that is introduced in the values of $v$ and its derivatives.

Since this coefficient is bounded and regular close to $t_\varepsilon$ and $t_\varepsilon$ satisfies (\ref{teprop}), we are now in the same situation as in chapter 6 of \cite{ohlmann2021illposedness}. We will hence skip those computations since they are identical, as we have shown how we deal with modifications.

\section{Introducing a source term with a $x_2$ dependency}\label{modif2}

This chapter constitutes the main contribution of this paper.
Here, we study the stability of the blow-up with respect to the addition of a source term in the equation. This case is a study of the stability with respect to the equation, but can also be seen as a next step toward the stability of the phenomenon by perturbations depending on $x_2$ as well. Indeed, using spectral cutoffs of $u$ with respect to its second space variable, it may be possible that we can reduce the general stability to this case, for a fixed range of frequencies. This case is more tedious as we no longer have an explicit resolution, but it also deeper as we show that the function has to behave pathologically. We will also see that we can, in some sense, show that the solution will locally share similarities with the function that we previously computed for the problem without the $x_2$ dependency. To do so, we will use an "approximate" characteristics method.

\subsection{Definition of the new problem and preliminary results}

We now consider the following modified Cauchy problem.

\begin{equation} \label{Cau5}
\left\{
\begin{aligned}
&\Box_{x_1} u(t,x_1,x_2) = DuD^2u + f(t,x_1,x_2,Du)\\
&\frac{\partial u}{\partial t}_{|t=0} = - \chi, \hspace{0.2cm} u_{|t=0} = 0.
\end{aligned}
\right.
\end{equation}

\begin{remark}
Here, we study the case where $f$ depends on $Du$ but not on $u$ because it is the most difficult case. The case where $f$ depends on $u$ and $Du$ is really similar and does not bring more difficulties.
\end{remark}

We also define 

\begin{equation}\label{chilol}
\chi_\varepsilon(x_1,x_2) = - \int_{s=0}^{x_1} \psi_\varepsilon(s) |\ln(s)|^\alpha ds.
\end{equation}

Where $\psi_\varepsilon: \mathbb{R^+}\to \mathbb{R^+}$ satisfies the following conditions.

\begin{equation}\label{psilol}
\left\{
\begin{aligned}
&\psi_\varepsilon(s) = 0,\hspace{0.2cm} \forall s<\frac{\varepsilon}{2},\\
&\psi_\varepsilon(s) = 1,\hspace{0.2cm} \forall s>\varepsilon, \\
&\psi_\varepsilon(s) \in [0,1], \hspace{0.2cm} \forall s\in \mathbb{R}^{+},\\
&|\psi_\varepsilon'|(s) \leq \frac{2}{\varepsilon}, \hspace{0.2cm} \forall s \in \mathbb{R}^+.
\end{aligned}
\right.
\end{equation}

We will consider a cutoff of (\ref{chilol}) defined on $\mathbb{R}^2$, $\chi$, without explicit relabeling, as presented in (\ref{condcondiniini}).  Lastly, we also assume

\begin{equation}\label{uniff}
\forall \alpha, \exists C, \hspace{0.2cm} \left| \frac{\partial^\alpha}{\partial x^\alpha} f \right| \leq C.
\end{equation}

Now, rewriting (\ref{Cau5}), we obtain with $v = Du \neq 1$,

\begin{equation} \label{Cau6}
\left\{
\begin{aligned}
&\left( \partial_t + \frac{1+v}{1-v} \partial_{x_1} \right) v  = \frac{f(t,x_1,x_2,v)}{1-v} = g(t,x_1,x_2,v). \\
&\frac{\partial u}{\partial t}_{|t=0} = - \chi, \hspace{0.2cm} u_{|t=0} = 0,
\end{aligned}
\right.
\end{equation}

where $g$ also satisfies (\ref{uniff}) as $v<0$. We call $C_g$ the involved constant. Note that (\ref{Cau6}) is a perturbed version of (\ref{dimone}).

We now state the main theorem of the chapter.

\begin{theorem}\label{troisiemecas}

Let $u_\varepsilon$ be the solution to problem (\ref{Cau5}).
There exists a time $t_\varepsilon$ such that 
\begin{equation}
\left\{
\begin{aligned}
&||u_\varepsilon(0,\cdot)||_{H^{11/4}(\ln H)^{-\beta}} < \infty \\
&||u_\varepsilon(t,\cdot)||_{H^{11/4}(\ln H)^{-\beta})} \rightarrow \infty \hspace{0.2cm} \text{ as } t\rightarrow t_\varepsilon.
\end{aligned}
\right.
\end{equation}

Also, we have that 

\begin{equation}
t_\varepsilon \rightarrow 0 \hspace{0.2cm} \text{ as } \varepsilon \rightarrow 0.
\end{equation}

\end{theorem}

We now do some preliminary work, later we will prove a technical lemma, and then we will be able to make the proof of the theorem. Here, we will make free use of Gronwall type inequalities that can for instance be found in \cite{dragomir2003some}.

\textbf{Preliminary work}

First, we define $\phi$ as

\begin{equation}\label{defphiproblem3}
\left\{
\begin{aligned}
&\phi(0,x_1,x_2) = x_1 \\
& \partial_t \phi(t,x_1,x_2) = \frac{1+v(t,\phi(t,x_1,x_2),x_2)}{1-v(t,\phi(t,x_1,x_2),x_2)}.
\end{aligned}
\right.
\end{equation}

Now, using the chain rule, we obtain from (\ref{defphiproblem3}) and (\ref{Cau6}), 

\begin{equation}\label{diffvvv}
\frac{\partial}{\partial_t} \left(v(t,\phi(t,x_1,x_2),x_2) \right) = g(t,\phi(t,x_1,x_2),x_2,v(t,\phi(t,x_1,x_2),x_2)).
\end{equation}

From now on, for clarity's sake, we will not always specify the variables when there is no ambiguity. Most of the time, we only specify if the first space variable is $x_1$ or $\phi(t,x_1,x_2)$. Now, we obtain from (\ref{diffvvv}). 

\begin{equation}\label{intvvv}
v(\phi) = \chi(x_1) + \int_{\tau=0}^t g(\tau,\phi) d\tau.
\end{equation}

Differentiating (\ref{intvvv}) leads to the following expressions for the derivatives of $v$.

\begin{multline}\label{vvvx}
\partial_{x_1} \phi(x_1) \partial_{x_1} v(\phi) = \chi'(x_1) + \int_\tau \partial_{x_1} \phi(\tau,x_1) \partial_1 g(\tau,x_1) + \int_\tau \partial_{x_1} \phi(\tau,x_1) \partial_{x_1} v(\tau,x_1) \partial_3 g(\tau,x_1) \\
\Rightarrow \partial_{x_1} v(\phi)= \frac{\chi'(x_1) + \int_\tau \partial_{x_1} \phi(\tau,x_1) \partial_1 g(\tau,\phi) + \int_\tau \partial_{x_1} \phi(x_1) \partial_{x_1} v(\tau,x_1) \partial_3 g(\tau,x_1) }{\partial_{x_1} \phi(\tau,x_1)}.
\end{multline}

\begin{multline}\label{vvvxx}
\partial_{x_1}^2 v(\phi) = \frac{1}{\left( \partial_{x_1} \phi(x_1) \right)^2} \Big[ \chi''(x_1) + \int_{\tau} \left( \partial_{x_1} \phi(\tau,x_1) \right)^2\partial_{1}^2 g(\tau,\phi) + \int_\tau \partial_{x_1}^2 \phi(\tau,x_1) \partial_{1} g(\tau,\phi) \\
+ \int_\tau \left( \partial_{x_1} \phi(\tau,x_1) \right)^2 \partial_{x_1} v(\tau,\phi) \partial_3 \partial_1 g(\tau,\phi) + \int_\tau \partial_{x_1}^2 \phi(\tau,x_1) \partial_3 g(\tau,\phi) \partial_{x_1} v(\tau,\phi) \\
+ \int_\tau \left( \partial_{x_1} \phi(\tau,x_1) \right)^2 \partial_1 \partial_3 g(\tau,\phi) \partial_{x_1} v(\tau,\phi)  
+ \int_\tau \left( \partial_{x_1} \phi(\tau,x_1) \right)^2 \left( \partial_{x_1} v(\tau,\phi) \right)^2 \partial_3^2 g(\tau,\phi)  \\+ \int_\tau \left( \partial_{x_1} \phi(\tau,x_1) \right)^2  \partial_{x_1}^2 v(\tau,\phi) \partial_3 g(\tau,\phi)  \Big] \\
- \frac{ \partial_{x_1}^2 \phi(x_1)}{\left( \partial_{x_1} \phi(x_1) \right)^3 } \Big[ \chi'(x_1) + \int_\tau \partial_{x_1} \phi(\tau,x_1) \partial_1 g(\tau,\phi) + \int_\tau \partial_{x_1} \phi(\tau,x_1) \partial_3 g(\tau,\phi) \partial_{x_1} v(\tau,\phi) \Big]\\
= A - B.
\end{multline}

We also compute the derivatives of $\phi$ that we will need. From (\ref{defphiproblem3}), we get

\begin{equation}\label{phiphiphitxlol}
\phi_{tx_1}(x_1) = \frac{2 \phi_{x_1} v_{x_1}(\phi)}{(1-v(\phi))^2}.
\end{equation}

\begin{equation}\label{phiphiphitxxlol}
\phi_{tx_1x_1}(x_1) = \frac{2 (v(\phi))_{x_1x_1} (1-v(\phi)) + 2 (v(\phi))_{x_1}^2}{(1-v(\phi))^3}.
\end{equation}

At this stage, we will assume that $\phi_{x_2x_1x_1}(t_\varepsilon,\nu_\varepsilon)$ is not $0$. This is an important assumption, as assuming $\phi_{x_2x_1x_1}(t_\varepsilon,\nu_\varepsilon)=0$ leads to different computations. However, we consider only this case because the other case is much more similar that the situations that we have previously studied.

We start with a bootstrap argument to control the function $\phi_x$ and show estimates. We choose $\psi_\varepsilon$ to be non-decreasing with respect to $x$. We call $c_0$ the number (implicitly depending on $\varepsilon$) such that $\psi_\varepsilon \in [0,\frac{1}{10}]$ for $x\in [\varepsilon/2,c_0]$. We consider only times such that $\phi_x \neq 0$ and $t<1/|\ln|\ln(\varepsilon)||$.

We can now state the technical lemma.

\begin{lemma}\label{laconditionimportante}

We consider $x\in [c_0,\varepsilon]$.
There exists a time $t_\varepsilon$ such that the following properties are verified.

\begin{equation}\label{teps1}
t_\varepsilon \rightarrow 0 \hspace{0.2cm} \text{ as } \varepsilon \rightarrow 0. 
\end{equation}

\begin{equation}\label{teps2}
\begin{aligned}
&\phi_{x_1} > 0 \hspace{0.2cm} \forall t<t_\varepsilon\\
&\exists (\nu_\varepsilon^1, \nu_\varepsilon^2) \hspace{0.2cm} s.t. \hspace{0.2cm} \phi_{x_1}(t_\varepsilon,\nu_\varepsilon^1,\nu_\varepsilon^2) = 0
\end{aligned}
\end{equation}

If $|\phi_{x_1}(t,x_1,x_2)| \leq 2$ for $t\in [0,t_1] \subseteq [0,t_\varepsilon]$ and $x\in [c_0,\varepsilon]$, then $\phi_{x_1}(t,x_1,x_2) \leq 1$ for $t\in [0,t_1]$ and $x\in [c_0,\varepsilon]$.

\begin{equation}
\begin{aligned}
&0 < \phi_{x_1} < 1, \hspace{0.1cm} \phi_{t,x_1} <0 \hspace{0.2cm}, \forall t<t_\varepsilon, \\
&\exists C_\varepsilon, \hspace{0.2cm} |\phi_{t,x_1,x_1}| \leq C_\varepsilon |\ln(t_\varepsilon-t)|. \\
&v_{x_1} <0, \hspace{0.2cm} v<0\\
&\exists C_\varepsilon, \hspace{0.2cm} |v_{x_1}(\phi)| \leq \frac{C_\varepsilon(\chi'(x_1) + C_\varepsilon)}{|\phi_{x_1}|}.
\end{aligned}
\end{equation}

If $x_1$, $x_2$ and $t$ are sufficiently close to $(\nu_\varepsilon^1,\nu_\varepsilon^2,t_\varepsilon)$, we have the following estimates. 

\begin{multline}\label{estestestphix}
\exists C_\varepsilon^1,C_\varepsilon^2 >0, \hspace{0.2cm} \frac{C_\varepsilon^1}{(x_1-\nu_\varepsilon^1)^2 + (x_2-\nu_\varepsilon^2)^2 + (t_\varepsilon-t)} \leq \frac{1}{\phi_{x_1}(x_1)} \\ \leq 
\frac{C_\varepsilon^2}{(x_1-\nu_\varepsilon^1)^2 + (x_2-\nu_\varepsilon^2)^2 + (t_\varepsilon-t)}\\
\exists C_\varepsilon^1, C_\varepsilon^2, C_\varepsilon^3, \hspace{0.2cm} \phi_{x_1x_1}(x_1) = C_\varepsilon^1 (x_1-\nu_\varepsilon^1) + C_\varepsilon^2 (x_2-\nu_\varepsilon^2) + C_\varepsilon^3 (t_\varepsilon-t) \\
+ \sum_{i+j+k=2} (x_1-\nu_\varepsilon^1)^i (x_2-\nu_\varepsilon^2)^j (t_\varepsilon-t)^k f_{i,j,k}(t,x_1,x_2),
\end{multline}

where all the involved $f$ functions are bounded near $(t,\nu_\varepsilon^1,\nu_\varepsilon^2)$. We also have

\begin{equation}\label{superestA}
|A| \leq \frac{C_\varepsilon |\ln(t_\varepsilon-t)|}{|\phi_{x_1}|^2},
\end{equation}

where $A$ is the $A$ involved in (\ref{vvvxx}).

Lastly, for $\varepsilon$ small enough, we have 

\begin{equation}\label{superestB}
\begin{aligned}
&\frac{C_\varepsilon \phi_{x_1x_1}(x_1)}{\left( \phi_{x_1}(x_1) \right)^3} \frac{\chi'}{2} \leq B \leq \frac{C_\varepsilon \phi_{x_1x_1}(x_1)}{\left( \phi_{x_1}(x_1) \right)^3} 2\chi', \\
&\frac{C_\varepsilon \phi_{x_1x_1}(x_1)}{\left( \phi_{x_1}(x_1) \right)^3} 2\chi' \leq B \leq \frac{C_\varepsilon \phi_{x_1x_1}(x_1)}{\left( \phi_{x_1}(x_1) \right)^3} \frac{\chi'}{2}.
\end{aligned}
\end{equation}

The first inequality being verified when $\phi_{x_1x_1} \leq 0$, and the second being verified when $\phi_{x_1x_1} \geq 0$.

\end{lemma}

\begin{proof}

We hence assume $|\phi_{x_1}| \leq 2$.

We have from (\ref{intvvv}) and (\ref{uniff}) that 
\begin{equation}\label{vbornee}
|v(\phi(t,x_1))| \leq |\chi(x_1)| + Ct \leq C_\varepsilon.
\end{equation}

Also, we have from (\ref{intvvv}) that for $t<t_1$, where $t_1$ only depends on $f$, that

\begin{equation}\label{signedev}
v(t,x_1,x_2) \leq 1/2.
\end{equation}

Hence we have from (\ref{defphiproblem3}), (\ref{signedev}) and (\ref{vbornee}) that

\begin{equation}
|\phi_t| = \frac{\left| 1 + v(t,\phi) \right|}{\left| 1-v(t,\phi) \right|} \leq C_\varepsilon.
\end{equation}

We will make use of Gr\"onwall's inequality in (\ref{vvvx}), we have 

\begin{equation}\label{ki}
\left| \phi_{x_1} v_{x_1} \right| \leq |\psi_\varepsilon(x_1)| |\ln(x_1)|^\alpha + 2 \frac{1}{|\ln(\ln(\varepsilon))|} C_g + \int_\tau C_g |\phi_{x_1} v_{x_1}|
\end{equation} 

and

\begin{equation}\label{grownki}
|v_{x_1} \phi_{x_1}| \leq \left( |\ln(\varepsilon)|^\alpha + C \right)  e^{\int_\tau |\partial_3 g|} \leq C \left| \ln(\varepsilon) \right|^\alpha e^{\frac{1}{|\ln(|\ln(\varepsilon)||}}.
\end{equation}

Now, looking at (\ref{vvvx}) again, we obtain

\begin{multline}\label{ceciestunlabel}
-\psi_\varepsilon(x_1) |\ln(x_1)|^\alpha + \int_\tau \phi_{x_1}(\tau,x_1,x_2) \partial_1 g + \int_\tau \phi_{x_1} v_{x_1} \partial_3 g \\ \leq -\frac{|\ln(\varepsilon)|^\alpha}{10} + 2 C_g \frac{1}{|\ln(|\ln(\varepsilon)|)|} + C_g \frac{1}{|\ln(|\ln(\varepsilon)|)|} C \left| \ln(\varepsilon) \right|^\alpha e^{\frac{1}{|\ln(|\ln(\varepsilon)|)|}} \leq 0,
\end{multline}

for $\varepsilon$ small enough. From (\ref{phiphiphitxlol}), we hence get that $\phi_{t,x_1} \leq 0$. 
Because we only consider times such that $\phi_{x_1} \neq 0$, (\ref{defphiproblem3}) gives by continuity

\begin{equation}\label{ki2}
\phi_{x_1} >0.
\end{equation} 

Hence $\phi_{x_1} \leq 1$ on the considered interval. This is the estimate we wanted for $\phi_{x_1}$ to make our bootstrap argument on $\phi_{x_1}$ work. We go on with the other estimates.

Using the same reasoning as for (\ref{ceciestunlabel}), we obtain

\begin{equation}\label{signedevx}
2 \frac{\chi'(x_1)}{\phi_{x_1}} \leq v_{x_1} \leq \frac{\chi'(x_1)}{2} \frac{1}{\phi_{x_1}} \leq 0.
\end{equation}

Now, from (\ref{vvvx}), (\ref{phiphiphitxlol}) and (\ref{signedevx})

\begin{equation}
\phi_{tx_1} \leq \frac{\chi'(x_1)}{2 (1-v)^2} < 0.
\end{equation}

This leads us to the following two conclusions. From (\ref{chilol}) and (\ref{psilol}), we have that $\chi'(\varepsilon) = -|\ln(\varepsilon)|^\alpha$.

Hence, for $\varepsilon$ small enough, there exists a time $t_\varepsilon$ satisfying (\ref{teps1}) and (\ref{teps2}). We also have

\begin{equation}\label{estblowup}
t_\varepsilon \leq \frac{4}{|\ln(\varepsilon)|^\alpha}.
\end{equation}

Since we have that $t_\varepsilon << \frac{1}{|\ln(|\ln(\varepsilon)|)|}$, our time $t_\varepsilon$ belongs to the set in which we made our estimate.

Now, because, by a continuity argument, we have that $\phi_{x_1} (t_\varepsilon,\nu_\varepsilon^1,\nu_\varepsilon^2)$ is a minimum in $(x_1,x_2)$, we have that $\phi_{x_1x_1} = \phi_{x_1x_2}=0$. A Taylor expansion leads to

\begin{equation}\label{taylortaylor}
\begin{aligned}
\phi_{x_1}(t_\varepsilon,x_1,x_2) &= C_{\varepsilon,1} (x_1-\nu_\varepsilon^1)^2 + C_{\varepsilon,2} (x_2 - \nu_\varepsilon^2)^2 + C_{\varepsilon,3}(x_1-\nu_\varepsilon^1)(x_2-\nu_\varepsilon^2) \\
&+ o(d((x_1,x_2),(\nu_\varepsilon^1,\nu_\varepsilon^2))^2) \\
&C_{\varepsilon,3}^2 \leq C_{\varepsilon,1} C_{\varepsilon,2}, \hspace{0.2cm} {C_\varepsilon^1}, {C_\varepsilon^2} > 0. 
\end{aligned}
\end{equation}

Now, from (\ref{taylortaylor}) and using $\sqrt{ab} \leq \frac{a+b}{2}$, we obtain

\begin{multline}\label{taylortaylor2}
\frac{C_{\varepsilon,3}^2}{ C_{\varepsilon,1} C_{\varepsilon,2}} \leq \delta <1 \Rightarrow \left| C_{\varepsilon,3}(x_1-\nu_\varepsilon^1)(x_2-\nu_\varepsilon^2) \right| \leq \delta \left| \sqrt{C_{\varepsilon,1}} (x_1 - \nu_\varepsilon^1) \sqrt{C_{\varepsilon,2}} (x_2 - \nu_\varepsilon^2) \right| \\
\leq \delta \left[ \frac{C_{\varepsilon,1}(x_1-\nu_\varepsilon^2)^2 + C_{\varepsilon,2}(x_2-\nu_\varepsilon^2)^2}{2} \right]  . 
\end{multline}

Now, from (\ref{taylortaylor}) and (\ref{taylortaylor2}),  we obtain

\begin{multline}
\frac{2-\delta}{2} C_{\varepsilon,1} (x_1-\nu_\varepsilon^1)^2 + \frac{2-\delta}{2} C_{\varepsilon,2} (x_2-\nu_\varepsilon^2)^2 \\
\leq C_{\varepsilon,1} (x_1-\nu_\varepsilon^1)^2 + C_{\varepsilon,2} (x_2 - \nu_\varepsilon^2)^2 + C_{\varepsilon,3}(x_1-\nu_\varepsilon^1)(x_2-\nu_\varepsilon^2) \\
\leq \frac{2+\delta}{2} C_{\varepsilon,1} (x_1-\nu_\varepsilon^1)^2 + \frac{2+\delta}{2} C_{\varepsilon,2} (x_2-\nu_\varepsilon^2)^2.
\end{multline}

Hence, using the sign of $\phi_{tx_1}$, we obtain estimate (\ref{estestestphix}).

Now, we deal with the terms $A$ and $B$ involved in (\ref{vvvxx}). We will go through each term one by one. With $0<\phi_{x_1}\leq 1$, we obtain

\begin{equation}
\left| \int_\tau \left( \phi_{x_1} (\tau,x_1) \right)^2 \partial_1^2 g(\tau,\phi) \right| \leq \int_\tau C \leq C.
\end{equation}

Using (\ref{grownki}) and $0<\phi_{x_1}\leq 1$, we get

\begin{equation}
\left| \int_\tau (\phi_{x_1})^2 v_{x_1} (\tau,\phi) \partial_3\partial_1 g \right| \leq C \int_\tau (\phi_{x_1}) \left( \phi_{x_1} v_{x_1} \right) \leq C \int_\tau C(C+|\chi'|) \leq C (C+|\chi'|)
\end{equation}

\begin{equation}
\left| \int_\tau \left( \phi_{x_1} \right)^2 \left( v_{x_1} \right)^2 \partial_3^2 g \right| \leq \int_\tau |C(C+\chi')|^2 \leq C_\varepsilon.
\end{equation}

\begin{equation}
\left| \int_\tau \phi_{x_1x_1} v_{x_1} \partial_3 g \right| \leq C_\varepsilon \int_\tau \frac{C(C+|\chi'|)}{\phi_{x_1}} \leq C_\varepsilon |\ln(t_\varepsilon-t)|.
\end{equation}

Indeed, (\ref{estestestphix}) provides in particular that $\frac{1}{\phi_{x_1}} \geq \frac{1}{(t_\varepsilon-t)}$.

We hence obtain a first estimate for $A$ :

\begin{multline}\label{Apremier}
|A| \leq \frac{1}{(\phi_{x_1})^2} \left[ C_\varepsilon +C_\varepsilon |\ln(t_\varepsilon-t)| +\int_\tau |\phi_{x_1}^2 v_{x_1x_1} \partial_3 g| \right] \\
\leq \frac{1}{(\phi_{x_1})^2} \left[ C_\varepsilon +C_\varepsilon |\ln(t_\varepsilon-t)| + C_\varepsilon \int_\tau |\phi_{x_1}|^2 | v_{x_1x_1}| \right]
\end{multline}

We look at the terms involved in $B$. Using (\ref{grownki}) and $0<\phi_{x_1}\leq 1$ again, we obtain

\begin{equation}
\left| \int_\tau \phi_{x_1} \partial_1 g\right| \leq C_\varepsilon, \hspace{0.2cm} \left| \int_\tau \phi_{x_1} v_{x_1} \partial_3 g\right| \leq C_\varepsilon.
\end{equation}

This means that 

\begin{equation}\label{estdeuxieme}
|v_{x_1x_1}| \leq \frac{1}{\phi_{x_1}^2} \left[ C_\varepsilon |\ln(t_\varepsilon-t)| + C_\varepsilon \int_\tau |v_{x_1x_1}| \phi_{x_1}^2 \right] + \frac{C_\varepsilon}{\phi_{x_1}^3}. 
\end{equation}

Using Gr\"onwall's inequality once again in (\ref{estdeuxieme}), we obtain

\begin{equation}\label{esttroisieme}
|v_{x_1x_1}| \leq \frac{C_\varepsilon}{\phi_{x_1}^3}e^{\frac{C_\varepsilon}{\phi_{x_1}^2} \int_\tau \phi_{x_1}^2(\tau)} \leq \frac{C_\varepsilon}{\phi_{x_1}^3} e^{C_\varepsilon \int_\tau 1}\leq \frac{C_\varepsilon}{\phi_{x_1}^3}. 
\end{equation}

Now, inserting (\ref{esttroisieme}) back in (\ref{Apremier}), we obtain

\begin{equation}
|A| \leq \frac{C_\varepsilon}{\phi_{x_1}^2} \left[ \ln(t_\varepsilon-t)| + \int_\tau C_\varepsilon \frac{1}{\phi_{x_1}} \right] \leq \frac{C_\varepsilon |\ln(t_\varepsilon-t)|}{\phi_{x_1}^2},
\end{equation}

and hence we have proven (\ref{superestA}).

Now, we want to be more precise about the term involved in $B$. Using (\ref{estblowup}), and $0<\phi_x\leq 1$, we obtain

\begin{equation}\label{kk1}
\left| \int_\tau \phi_{x_1} \partial_1 g \right| \leq C \int_{\tau} |\phi_{x_1}| \leq C \frac{1}{|\ln(\varepsilon)|^\alpha} \rightarrow_\varepsilon 0.
\end{equation}

\begin{equation}\label{kk2}
\left| \int_\tau \phi_{x_1} v_{x_1} \partial_3 g \right| \leq C \int C(C+|\chi'|) \leq \frac{C(C+|\chi'|)}{|\ln(\varepsilon)|^\alpha} = o_\varepsilon(|\chi'|).
\end{equation}

Hence, for $\varepsilon$ small enough (the constants involved in (\ref{kk1}) and (\ref{kk2}) being absolute constants only depending on the equation), estimates (\ref{superestB}) holds.

\end{proof}

Now, we present a lemma that removes the assumption $x_1 \in [c_0,\varepsilon]$. Essentially, this will mean that $\nu_\varepsilon \in [c_0,\varepsilon]$. In fact, we need slightly more, we need that there is an open set $O= ]\nu_\varepsilon - \delta_\varepsilon, \nu_\varepsilon + \delta_\varepsilon[$ such that $O \cap [\varepsilon/2,c_0] = \emptyset$. This will ensure that all our estimates are valid when we will do the cutoff in the next step.

\begin{lemma}
For $x_1$ such that $\psi_\varepsilon(x_1) \leq \frac{1}{9}$, then $\nu_\varepsilon \neq x_1$. 
\end{lemma}

\begin{proof}

To show this, we will use the estimate (\ref{estblowup}) and show that $\phi_{x_1} (x_1)$ can not reach $0$ during that time. We consider $x_1$ such that $\psi_\varepsilon(x_1) \leq \frac{1}{9}$. By the same Gr\"onwall estimate as in the previous proof, we obtain 

\begin{equation}
|\phi_{x_1} v_{x_1}| \leq (C+ \frac{|\ln(\varepsilon)|^\alpha}{9} )e^{\int_\tau \partial_3 g} \leq \frac{|\ln(\varepsilon)|^\alpha}{8},
\end{equation}

for $\varepsilon$ small enough. Now, because $(\phi_{x_1})_{|t=0} =1$, we have that for all $t\leq \frac{4}{|\ln(\varepsilon)|^\alpha}$, 

\begin{equation}
\phi_{x_1}(x_1) \geq 1 - t \frac{|\ln(\varepsilon)|^\alpha}{8} \geq 1/2.
\end{equation}

\end{proof}

\begin{remark}
Informally, we now have that 
\begin{equation*}
v_{x_1x_1} \simeq \frac{C_\varepsilon |\ln(t_\varepsilon-t)|}{\phi_{x_1}^2} - \frac{C_\varepsilon \chi' \phi_{x_1x_1}}{\phi_{x_1}^3}.
\end{equation*}

The $\simeq$ symbols being an upper-bound for the first term, and control of the behavior for the second term. The situation is hence closer to the one we had in the two previous cases. However, the estimates that we will now use for $\phi_{x_1x_1}$ and $\phi_{x_1}$ now depends on $x_2$, and the estimation will be harder, especially the control of the sign.

\end{remark}

\subsection{Proof of the blow-up as $t\rightarrow t_\varepsilon$}

We now consider the following cutoffs in $x_1$ and $x_2$.

Let $\psi_{\varepsilon}^1: \mathbb{R} \to \mathbb{R}$ be a $C^\infty$ function satisfying

\begin{equation}\label{psi1lol}
\left\{
\begin{aligned}
&\psi_\varepsilon^1(x) = 1 \text{ for } \phi(t_\varepsilon,\nu_\varepsilon) - \delta_\varepsilon <x<\phi(t_\varepsilon,\nu_\varepsilon) + \delta_\varepsilon \\
&\psi_\varepsilon^1(x) = 0 \text{ for } \phi(t_\varepsilon,\nu_\varepsilon) + 2 \delta_\varepsilon <x \text{ or } x<\phi(t_\varepsilon,\nu_\varepsilon) - 2 \delta_\varepsilon \\
&0 < \psi_\varepsilon^1(x) < 1 \text{ elsewhere},
\end{aligned}
\right.
\end{equation}

and $\psi_{\varepsilon}^2: \mathbb{R} \to \mathbb{R}$ be a $C^\infty$ function satisfying

\begin{equation}\label{troisiemecaspsi2}
\left\{
\begin{aligned}
&\psi_\varepsilon^2(x) = 1 \text{ for } - \delta_\varepsilon <x< \delta_\varepsilon \\
&\psi_\varepsilon^2(x) = 0 \text{ for }  2 \delta_\varepsilon <x \text{ or } x< - 2 \delta_\varepsilon \\
&0 < \psi_\varepsilon^2(x) < 1 \text{ elsewhere}.
\end{aligned}
\right.
\end{equation}

We define $h(t,x_1,x_2) = \psi_\varepsilon^1(x_1) \psi_\varepsilon^2(x_2) v(t,x_1,x_2)$.
We will show that $||h||_{\dot H_{x_1}^{7/4}(\mathbb{R}^2)}\rightarrow \infty$ as $t\rightarrow t_\varepsilon$, and use the fact that

\begin{equation}
I_\varepsilon (t) = ||\psi_\varepsilon^1 \psi_\varepsilon^2 v||_{\dot H_{x_1}^{7/4}(\ln H)^{-\lambda}} = \int_{x_2} \int_{x_1} \int_y \frac{1}{|x_1-y|^{1/2-2\lambda}} h_{x_1x_1} (t,x_1,x_2) h_{x_1x_1} (t,y,x_2). 
\end{equation}

This will show theorem \ref{troisiemecas}.

\begin{proof}

We first split the integral in $x_1$ in three domains, 

\begin{multline}
I_\varepsilon(t)
= \int_{x_2=\nu_\varepsilon^2 -2\delta_\varepsilon}^{\nu_\varepsilon^2+ 2\delta_\varepsilon} \psi_\varepsilon^2(x_2) \int_{x_1=\phi(t_\varepsilon,\nu_\varepsilon) - 2 \delta_\varepsilon }^{\phi(t_\varepsilon,\nu_\varepsilon) - \delta_\varepsilon } \left( \frac{\partial^2 (v \psi_{\varepsilon}^1)}{\partial x_1^2} \right)(t,x_1) \\
\cdot  \int_{y = \phi(t_\varepsilon,\nu_\varepsilon) - 2 \delta_\varepsilon }^{\phi(t_\varepsilon,\nu_\varepsilon) + 2 \delta_\varepsilon } |x_1-y|^{-1/2-2\lambda} \left( \frac{\partial^2 (v \psi_{\varepsilon}^1)}{\partial x_1^2} \right)(t,y) dy dx_1 \\
+ \int_{x_2=\nu_\varepsilon^2 -2\delta_\varepsilon}^{\nu_\varepsilon^2+ 2\delta_\varepsilon} \psi_\varepsilon^2(x_2) \int_{x_1=\phi(t_\varepsilon,\nu_\varepsilon) - \delta_\varepsilon }^{\phi(t_\varepsilon,\nu_\varepsilon) + \delta_\varepsilon } \left( \frac{\partial^2 v}{\partial x_1^2} \right)(t,x_1)\\ \cdot  \int_{y = \phi(t_\varepsilon,\nu_\varepsilon) - 2 \delta_\varepsilon }^{\phi(t_\varepsilon,\nu_\varepsilon) + 2 \delta_\varepsilon } |x_1-y|^{-1/2-2\lambda} \left( \frac{\partial^2 v}{\partial x_1^2} \right)(t,y) dy dx_1 \\
+ \int_{x_2=\nu_\varepsilon^2 -2\delta_\varepsilon}^{\nu_\varepsilon^2+ 2\delta_\varepsilon} \psi_\varepsilon^2(x_2) \int_{x_1=\phi(t_\varepsilon,\nu_\varepsilon) + \delta_\varepsilon }^{\phi(t_\varepsilon,\nu_\varepsilon) + 2 \delta_\varepsilon } \left( \frac{\partial^2 (v \psi_{\varepsilon}^1)}{\partial x_1^2} \right)(t,x_1) \\ \cdot  \int_{y = \phi(t_\varepsilon,\nu_\varepsilon) - 2 \delta_\varepsilon }^{\phi(t_\varepsilon,\nu_\varepsilon) + 2 \delta_\varepsilon }|x_1-y|^{-1/2-2\lambda} \left( \frac{\partial^2 (v \psi_{\varepsilon}^1)}{\partial x_1^2} \right)(t,y) dy dx_1. \\
= \int_{x_2=\nu_\varepsilon^2 -2\delta_\varepsilon}^{\nu_\varepsilon^2+ 2\delta_\varepsilon} \psi_\varepsilon^2(x_2) I_\varepsilon^1(t,x_2) + \int_{x_2=\nu_\varepsilon^2 -2\delta_\varepsilon}^{\nu_\varepsilon^2+ 2\delta_\varepsilon} \psi_\varepsilon^2(x_2) I_\varepsilon^2(t,x_2) + \int_{x_2=\nu_\varepsilon^2 -2\delta_\varepsilon}^{\nu_\varepsilon^2+ 2\delta_\varepsilon} \psi_\varepsilon^2(x_2) I_\varepsilon^3(t,x_2).
\end{multline}

We now consider a fixed $x_2$ and study the integrals $I_\varepsilon^1(t,x_2)$, $I_\varepsilon^2(t,x_2)$ and $I_\varepsilon^3(t,x_2)$. 

Let us make the change of variable $x_1 = \phi(t,x_1,x_2)$ (not explicitly relabeled).

\begin{equation}\label{troisiemecaslesI}
\begin{aligned}
I_\varepsilon(t,x_2)
=  &\int_{x_1=\zeta_\varepsilon^1(t,x_2)}^{\zeta_\varepsilon^2(t,x_2)}  \phi_{x_1}(t,x_1,x_2) \left( \frac{\partial^2 (v \psi_{\varepsilon}^1)}{\partial x_1^2} \right)(t,\phi(t,x_1,x_2))\\
 \cdot  \int_{y = \zeta_\varepsilon^1(t,x_2)}^{\zeta_\varepsilon^4(t,x_2)}& |\phi(t,x_1)- \phi(t,y,x_2)|^{-1/2+2\lambda} \phi_{x_1}(t,y,x_2) \left( \frac{\partial^2 (v \psi_{\varepsilon}^1)}{\partial x_1^2} \right)(t,\phi(t,y,x_2)) dx_1 dy \\
+  \int_{x_1=\zeta_\varepsilon^2(t,x_2)}^{\zeta_\varepsilon^3(t,x_2)} & \phi_{x_1}(t,x_1,x_2) \left( \frac{\partial^2 (v)}{\partial x_1^2} \right)(t,\phi(t,x_1,x_2))\\
 \cdot  \int_{y = \zeta_\varepsilon^1(t,x_2)}^{\zeta_\varepsilon^4(t,x_2)}& |\phi(t,y,x_2)- \phi(t,y,x_2)|^{-1/2+2\lambda} \phi_{x_1}(t,y,x_2) \left( \frac{\partial^2 (v)}{\partial x_1^2} \right)(t,\phi(t,y,x_2)) dx_1 dy \\
+  \int_{x_1=\zeta_\varepsilon^3(t,x_2)}^{\zeta_\varepsilon^4(t,x_2)} & \phi_{x_1}(t,x_1,x_2) \left( \frac{\partial^2 (v \psi_{\varepsilon}^1)}{\partial x_1^2} \right)(t,\phi(t,x_1,x_2))\\
 \cdot  \int_{y = \zeta_\varepsilon^1(t,x_2)}^{\zeta_\varepsilon^4(t,x_2)}& |\phi(t,x_1)- \phi(t,y,x_2)|^{-1/2+2\lambda} \phi_{x_1}(t,y,x_2) \left( \frac{\partial^2 (v \psi_{\varepsilon}^1)}{\partial x_1^2} \right)(t,\phi(t,y,x_2)) dx_1 dy \\ \\
&= I_\varepsilon^1(t,x_2) + I_\varepsilon^2(t,x_2) + I_\varepsilon^3(t,x_2).
\end{aligned}
\end{equation}

We will use all the following results without proving it, as it is very similar to what we obtained in the previous chapter.

\begin{lemma}

The following holds on an interval centered at $(t_\varepsilon,\nu_\varepsilon^1,\varepsilon^2)$.

If we define 
\begin{equation}
\begin{aligned}
&\zeta_\varepsilon^1(t,x_2) = \phi^{-1}_{t}(\phi(t_\varepsilon,\nu_\varepsilon^1,\nu_\varepsilon^2)-2\delta_\varepsilon,x_2),\\
&\zeta_\varepsilon^2(t,x_2) = \phi^{-1}_{t}(\phi(t_\varepsilon,\nu_\varepsilon^1,\nu_\varepsilon^2)-\delta_\varepsilon,x_2),\\
&\zeta_\varepsilon^3(t,x_2) = \phi^{-1}_{t}(\phi(t_\varepsilon,\nu_\varepsilon^1,\nu_\varepsilon^2)+\delta_\varepsilon,x_2),\\
&\zeta_\varepsilon^4(t,x_2) = \phi^{-1}_{t}(\phi(t_\varepsilon,\nu_\varepsilon^1,\nu_\varepsilon^2)+2\delta_\varepsilon,x_2),\\
\end{aligned}
\end{equation}

For any $\eta_\varepsilon$, we can choose $\delta_\varepsilon$ and $t_\varepsilon^1$ such that for $t\in ]t_\varepsilon^1, t_\varepsilon[$, we have

\begin{equation}\label{leszetaslol}
\nu_\varepsilon- \eta_\varepsilon < \zeta_\varepsilon^1(t,x_2) < \zeta_\varepsilon^2(t,x_2) < \nu_\varepsilon < \zeta_\varepsilon^3(t,x_2) < \zeta_\varepsilon^4(t,x_2) < \nu_\varepsilon + \eta_\varepsilon.
\end{equation}

There exists $\zeta_\varepsilon^{2+}$ and $\zeta_\varepsilon^{3-}$ such that for $\delta_\varepsilon$ small enough, and $t$ close enough to $t_\varepsilon$, 

\begin{equation}\label{order2lol}
\nu_\varepsilon- \eta_\varepsilon < \zeta_\varepsilon^1(t,x_2) < \zeta_\varepsilon^2(t,x_2) <\zeta_\varepsilon^{2+} < \nu_\varepsilon < \zeta_\varepsilon^{3-} < \zeta_\varepsilon^3(t,x_2) < \zeta_\varepsilon^4(t,x_2) < \nu_\varepsilon + \eta_\varepsilon.
\end{equation}

\end{lemma}

We first study $I_\varepsilon^1(t,x_2)$. The case of $I_\varepsilon^3(t,x_2)$ is similar. 
Using the expression of $I_\varepsilon^1(t,x_2)$ provided in (\ref{troisiemecaslesI}), as well as (\ref{superestA}), (\ref{estestestphix}) and (\ref{superestB}), we obtain

\begin{multline}\label{lel1}
\left| I_\varepsilon^1(t,x_2) \right| \leq \int_{x_1=\nu_\varepsilon^1 - \eta_\varepsilon}^{\zeta_\varepsilon^{2+}} \phi_{x_1}(t,x_1,x_2) \left( |A| + |B| \right)(x_1) \int_{y=\nu_\varepsilon^1 - \eta_\varepsilon}^{\nu_\varepsilon^1 + \eta_\varepsilon} \frac{\phi_{x_1}(t,y,x_2) \left( |A| + |B| \right)(y)}{|\phi(t,x_1,x_2) - \phi(t,y,x_2)|^{1/2-2\lambda}}.\\
\left| I_\varepsilon^1(t,x_2) \right| \leq \int_{x_1=\nu_\varepsilon^1 - \eta_\varepsilon}^{\zeta_\varepsilon^{2+}} \left( \frac{C_\varepsilon |\ln(t_\varepsilon-t)| }{\phi_{x_1}(x_1)} + \frac{C_\varepsilon |\phi_{x_1x_1}(t,x_1,x_2)|}{\phi_{x_1}^2(x_1)} \right) \\
\cdot  \int_{y=\nu_\varepsilon^1 - \eta_\varepsilon}^{\nu_\varepsilon^1 + \eta_\varepsilon} \left( \frac{C_\varepsilon |\ln(t_\varepsilon-t)| }{\phi_{x_1}(y)} + \frac{C_\varepsilon |\phi_{x_1x_1}(t,y,x_2)| }{\phi_{x_1}^2(y)} \right) \\
\cdot \frac{1}{|\phi(t,x_1,x_2) - \phi(t,y,x_2)|^{1/2-2\lambda}}.\\
\end{multline}

On the considered domain for $x_1$, there is a constant depending only on $\varepsilon$ such that

\begin{equation}
\frac{C_\varepsilon |\ln(t_\varepsilon-t)| }{\phi_{x_1}(x_1)} + \frac{C_\varepsilon |\phi_{x_1x_1}(t,x_1,x_2)|}{\phi_{x_1}^2(x_1)} \leq C_\varepsilon |\ln(t_\varepsilon-t)|.
\end{equation}

Now, we split the domain of the second integral involved in (\ref{lel1}) into two parts. The first domain corresponding to the unboundedness of $\frac{1}{|\phi(t,x_1,x_2) - \phi(t,y,x_2)|^{1/2-2\lambda}}$, the second to the unboundedness of $\frac{1}{\phi_{x_1}}$. 

\begin{multline}
|I_\varepsilon^1(t,x_2)| \leq M_\varepsilon \int_{x_1=\nu_\varepsilon^1 - \eta_\varepsilon}^{\zeta_\varepsilon^{2+}} \int_{y=\nu_\varepsilon^1-\eta_\varepsilon}^{\zeta_\varepsilon^{2+}} \frac{|\ln(t_\varepsilon-t)|^2}{|\phi(t,x_1,x_2) - \phi(t,y,x_2)|^{1/2-2\lambda}} \\
+ M_\varepsilon \int_{x_1=\nu_\varepsilon^1 - \eta_\varepsilon}^{\zeta_\varepsilon^{2+}} \int_{y=\zeta_\varepsilon^{2+}}^{\nu_\varepsilon^1+\eta_\varepsilon} \frac{ |\ln(t_\varepsilon-t)|^2}{(y-\nu_\varepsilon^1)^2 + (x_2-\nu_\varepsilon^2)^2 + (t_\varepsilon-t)}\\
+M_\varepsilon \int_{x_1=\nu_\varepsilon^1 - \eta_\varepsilon}^{\zeta_\varepsilon^{2+}} \int_{y=\zeta_\varepsilon^{2+}}^{\nu_\varepsilon^1+\eta_\varepsilon} \frac{|\ln(t_\varepsilon-t)|((t_\varepsilon-t) + |x_2-\nu_\varepsilon^2| + |x_1-\nu_\varepsilon^1|)}{ \left( (y-\nu_\varepsilon^1)^2 + (x_2-\nu_\varepsilon^2)^2 + (t_\varepsilon-t)\right)^2}\\
\leq (i) + (ii) + (iii).
\end{multline}

Now, using the mean value theorem and (\ref{estestestphix}), we obtain (for a $c\in[x_1,y] \subseteq [\nu_\varepsilon^1 - \eta_\varepsilon,\zeta_\varepsilon^{2+}]$) for $(i)$

\begin{equation}
(i) \leq M_\varepsilon \int_{x_1=\nu_\varepsilon^1 - \eta_\varepsilon}^{\zeta_\varepsilon^{2+}} \int_{y=\nu_\varepsilon^1-\eta_\varepsilon}^{\zeta_\varepsilon^{2+}} \frac{|\ln(t_\varepsilon-t)|^2}{|\phi_{x_1}(t,c,x_2)|^{1/2-2\lambda}|x_1 - y|^{1/2-2\lambda}} \leq M_\varepsilon \frac{|\ln(t_\varepsilon-t)|^2}{(t_\varepsilon-t)^{1/2-2\lambda}}.
\end{equation}

Now, for $(ii)$, we obtain

\begin{multline}
(ii) \leq M_\varepsilon |\ln(t_\varepsilon-t)|^2 \int_{x_1=\nu_\varepsilon^1}^{\zeta_\varepsilon^{2+}} \frac{1}{\left( (x_2-\nu_\varepsilon^2)^2 + (t_\varepsilon-t)\right)^{1/2}} \leq M_\varepsilon \frac{|\ln(t_\varepsilon-t)|^2}{\left( (x_2-\nu_\varepsilon^2)^2 + (t_\varepsilon-t)\right)^{1/2}}.
\end{multline}

For $(iii)$, we get

\begin{multline}
(iii) \leq M_\varepsilon |\ln(t_\varepsilon-t)| \int_{x_1=\nu_\varepsilon^1}^{\zeta_\varepsilon^{2+}}  \frac{1}{\left((x_2-\nu_\varepsilon^2)^2 + (t_\varepsilon-t)\right)^2} \leq M_\varepsilon \frac{|\ln(t_\varepsilon-t)|}{\left((x_2-\nu_\varepsilon^2)^2 + (t_\varepsilon-t)\right)^{3/2}}.
\end{multline}

Now, using (\ref{troisiemecaspsi2}), we get

\begin{multline}
\left| \int_{x_2=\nu_\varepsilon^2-\eta_\varepsilon}^{\nu_\varepsilon^2+\eta_\varepsilon} \psi_\varepsilon^2(x_2) I_\varepsilon^1(t,x_2) \right| \leq M_\varepsilon \int_{x_2=\nu_\varepsilon^2-\eta_\varepsilon}^{\nu_\varepsilon^2+\eta_\varepsilon} \Big[ |\ln(t_\varepsilon-t)|^2 + \frac{|\ln(t_\varepsilon-t)|^2}{\sqrt{(x_2-\nu_\varepsilon^2)^2 + (t_\varepsilon-t)}}\\
+ \frac{|\ln(t_\varepsilon-t)|}{\left( (x_2-\nu_\varepsilon^2)^2 + (t_\varepsilon-t)\right)^{3/2}} \Big]
\leq M_\varepsilon \frac{|\ln(t_\varepsilon-t)|^2}{(t_\varepsilon-t)}.
\end{multline}

By symmetry, we also have that

\begin{multline}
\left| \int_{x_2=\nu_\varepsilon^2-\eta_\varepsilon}^{\nu_\varepsilon^2+\eta_\varepsilon} \psi_\varepsilon^2(x_2) I_\varepsilon^3(t,x_2) \right| \leq M_\varepsilon \int_{x_2=\nu_\varepsilon^2-\eta_\varepsilon}^{\nu_\varepsilon^2+\eta_\varepsilon} \Big[ |\ln(t_\varepsilon-t)|^2 + \frac{|\ln(t_\varepsilon-t)|^2}{\sqrt{(x_2-\nu_\varepsilon^2)^2 + (t_\varepsilon-t)}}\\
+ \frac{|\ln(t_\varepsilon-t)|}{\left( (x_2-\nu_\varepsilon^2)^2 + (t_\varepsilon-t)\right)^{3/2}} \Big]
\leq M_\varepsilon \frac{|\ln(t_\varepsilon-t)|^2}{(t_\varepsilon-t)}.
\end{multline}

Now, we exhibit a lower bound for $I_\varepsilon^2(t,x_2)$. Using (\ref{superestA}). We consider the integral

\begin{multline}\label{lel3}
I_\varepsilon^2(t,x_2) = \int_{x_1= \zeta_\varepsilon^{2+}}^{\zeta_\varepsilon^{3-}} \int_{y=\nu_\varepsilon - \eta_\varepsilon}^{\nu_\varepsilon+\eta_\varepsilon} \frac{(A(x_1)-B(x_1))(A(y)-B(y))}{|\phi(t,x_1,x_2) - \phi(t,y,x_2)|^{1/2-2\lambda}} \\
= \int_{x_1= \zeta_\varepsilon^{2+}}^{\zeta_\varepsilon^{3-}} \int_{y=\nu_\varepsilon - \eta_\varepsilon}^{\nu_\varepsilon+\eta_\varepsilon} \frac{1}{|\phi(t,x_1,x_2) - \phi(t,y,x_2)|^{1/2-2\lambda}}\Big[ A(x_1)A(y) + A(x_1)B(y) \\
+ B(x_1)A(y) + B(x_1)B(y) \Big] = (i) + (ii) + (iii) + (iv).\\
\end{multline}

We will show that 

\begin{multline}
\int_{x_2=\nu_\varepsilon^2 - 2 \delta_\varepsilon}^{\nu_\varepsilon^2+2\delta_\varepsilon} \psi_\varepsilon(x_2) (iv) >> \int_{x_2=\nu_\varepsilon^2 - 2 \delta_\varepsilon}^{\nu_\varepsilon^2+2\delta_\varepsilon} \psi_\varepsilon(x_2)|(i)|+\int_{x_2=\nu_\varepsilon^2 - 2 \delta_\varepsilon}^{\nu_\varepsilon^2+2\delta_\varepsilon} \psi_\varepsilon(x_2)|(ii)|\\
+\int_{x_2=\nu_\varepsilon^2 - 2 \delta_\varepsilon}^{\nu_\varepsilon^2+2\delta_\varepsilon} \psi_\varepsilon(x_2)|(iii)|,
\end{multline}

as $t\rightarrow t_\varepsilon$.

We now make the additional change of variable $(x_1,y) = (x_1 - \nu_\varepsilon^1,y-\nu_\varepsilon^1)$, and thus the new domain is $\iota = [\zeta_\varepsilon^{2+} - \nu_\varepsilon^1,\zeta_\varepsilon^{3-}- \nu_\varepsilon^1] \times [-\eta_\varepsilon,\eta_\varepsilon].$
We first consider the term $(i)$ of (\ref{lel3}). We obtain from (\ref{estestestphix}) and (\ref{superestA}),

\begin{multline}\label{calcul(i)}
|(i)| \leq \int\int_\iota \frac{1}{|\phi(t,\nu_\varepsilon^1+x_1,x_2) - \phi(t,\nu_\varepsilon+y,x_2)|^{1/2-2\lambda}} \frac{C_\varepsilon |\ln(t_\varepsilon-t)|^2}{\phi_{x_1}(t,\nu_\varepsilon^1+x_1,x_2) \phi_{x_1}(t,\nu_\varepsilon^1+y,x_2)}\\
\leq C_\varepsilon \frac{|\ln(t_\varepsilon-t)|^2}{(t_\varepsilon-t)^{1/2-2\lambda}}\int\int_\iota \frac{1}{|x_1-y|^{1/2-2\lambda}} \frac{1}{x_1^2 + (x_2-\nu_\varepsilon^2)^2 + (t_\varepsilon-t)} \frac{1}{y^2 + (x_2-\nu_\varepsilon^2)^2 + (t_\varepsilon-t)}\\
\leq_{Holder}  \frac{C_\varepsilon|\ln(t_\varepsilon-t)|^2}{(t_\varepsilon-t)^{1/2-2\lambda}}\left( \int\int_\iota \frac{1}{|x_1-y|^{3/4-3\lambda}} \right)^{2/3} \left( \int_{x_1=0}^\infty \frac{1}{\left( x_1^2+(x_2-\nu_\varepsilon^2)^2 + (t_\varepsilon-t)\right)^{3}}  \right)^{1/3}\\
\cdot \left( \int_{y=0}^\infty \frac{1}{\left( y^2+(x_2-\nu_\varepsilon^2)^2 + (t_\varepsilon-t)\right)^3}  \right)^{1/3} \\
\leq \frac{C_\varepsilon |\ln(t_\varepsilon-t)|^2}{(t_\varepsilon-t)^{1/2-2\lambda}} \frac{1}{\left( (x_2-\nu_\varepsilon^2)^2 + (t_\varepsilon-t)\right)^{10/6}}.
\end{multline}

Now, using (\ref{troisiemecaspsi2}) and the inequality (\ref{calcul(i)}), we obtain

\begin{multline}
\left| \int_{x_2=\nu_\varepsilon^2-2\delta_\varepsilon}^{\nu_\varepsilon^2 + 2\delta_\varepsilon} \psi_\varepsilon^2(x_2) (i) \right| \leq \frac{C_\varepsilon |\ln(t_\varepsilon-t)|^2}{(t_\varepsilon-t)^{1/2-2\lambda}} 2 \int_{x_2=0}^\infty \frac{1}{\sqrt{x_2} (x_2+(t_\varepsilon-t))^{5/3}} \\
\leq \frac{C_\varepsilon |\ln(t_\varepsilon-t)|^2}{(t_\varepsilon-t)^{5/3-2\lambda}}.
\end{multline}

\begin{remark}
Here, we used that 
\begin{equation*}
\int_{x=0}^\infty \frac{1}{\sqrt{x}} \frac{1}{(x+a)^{5/3}} = \frac{B(\frac{1}{2},\frac{7}{6})}{a^{7/6}},
\end{equation*}

where $B$ is the Euler integral of the first kind.

\end{remark}

For $(ii)$, we obtain using the triangular inequality in (\ref{estestestphix}) as well as (\ref{superestA}) and (\ref{superestB}),

\begin{multline}\label{calcul(ii)}
|(ii)| \leq \int\int_\iota \frac{1}{|\phi(t,\nu_\varepsilon^1+x_1,x_2) - \phi(t,\nu_\varepsilon^2+y,x_2)|^{1/2-2\lambda} } \frac{C_\varepsilon |\ln(t_\varepsilon-t)| |\phi_{x_1x_1}(t,\nu_\varepsilon^1+y,x_2)|}{\phi_{x_1}(t,\nu_\varepsilon^1+x_1,x_2)\phi_{x_1}^2(t,\nu_\varepsilon^1+y,x_2)}\\
\leq \frac{C_\varepsilon|\ln(t_\varepsilon-t)|}{(t_\varepsilon-t)^{1/2-2\lambda}}\int\int_\iota \frac{1}{|x_1-y|^{1/2-2\lambda}} \frac{1}{x_1^2 + (x_2-\nu_\varepsilon^2)^2+(t_\varepsilon-t)} \frac{|y|+|x_2-\nu_\varepsilon^2| + (t_\varepsilon-t)}{\left( y^2 + (x_2-\nu_\varepsilon^2)^2+(t_\varepsilon-t)\right)^2}\\
\leq_{Holder} \frac{C_\varepsilon |\ln(t_\varepsilon-t)|}{(t_\varepsilon-t)^{1/2-2\lambda}} \left( \int\int_\iota \frac{1}{|x_1-y|^{3/4-3\lambda}} \right)^{2/3} \left( \int_{x_1=0}^\infty \frac{1}{\left(x_1^2+(x_2-\nu_\varepsilon^2)^2 + (t_\varepsilon-t)\right)^3} \right)^{1/3} \\
\cdot \left( \int_{y=0}^\infty \frac{\left( y+|\nu_\varepsilon^2-x_2| + (t_\varepsilon-t)\right)^3}{\left(y^2+(x_2-\nu_\varepsilon^2)^2 + (t_\varepsilon-t)\right)^6} \right)^{1/3}\\
\leq \frac{C_\varepsilon |\ln(t_\varepsilon-t)|}{(t_\varepsilon-t)^{1/2-2\lambda}} \frac{1}{\left( (x_2-\nu_\varepsilon^2)^2 + (t_\varepsilon-t)\right)^{5/6}} \frac{1}{\left( (x_2-\nu_\varepsilon^2)^2 + (t_\varepsilon-t) \right)^{4/3}}.
\end{multline}

Now, integrating (\ref{calcul(ii)}) with respect to $x_2$ and using (\ref{troisiemecaspsi2}) leads to

\begin{multline}
\left| \int_{x_2=\nu_\varepsilon^2 - 2 \delta_\varepsilon}^{\nu_\varepsilon^2+2\delta_\varepsilon} \psi_\varepsilon^2(x_2) (ii) \right| \leq \frac{C_\varepsilon|\ln(t_\varepsilon-t)|}{(t_\varepsilon-t)^{1/2-2\lambda}} 2 \int_{x_2=0}^\infty \frac{1}{\left( x_2^2 + (t_\varepsilon-t) \right)^{13/6}} \\
\leq \frac{C_\varepsilon|\ln(t_\varepsilon-t)|}{(t_\varepsilon-t)^{1/2-2\lambda}} \int_{x_2=0}^\infty \frac{1}{\sqrt{x_2} \left( x_2 + (t_\varepsilon-t) \right)^{13/6}} \leq \frac{C_\varepsilon |\ln(t_\varepsilon-t)|}{(t_\varepsilon-t)^{13/6-2\lambda}}.
\end{multline}

\begin{remark}
Here, we used that 
\begin{equation*}
\int_{x=0}^\infty \frac{1}{\sqrt{x}} \frac{1}{(x+a)^{13/6}} = \frac{B(\frac{1}{2},\frac{5}{3})}{a^{5/3}},
\end{equation*}

where $B$ is the Euler integral of the first kind.

\end{remark}

The case of $(iii)$ is identical, and we obtain

\begin{equation}
\left| \int_{x_2=\nu_\varepsilon^2 - 2 \delta_\varepsilon}^{\nu_\varepsilon^2+2\delta_\varepsilon} \psi_\varepsilon^2(x_2) (iii) \right| \leq \frac{C_\varepsilon |\ln(t_\varepsilon-t)|}{(t_\varepsilon-t)^{13/6}}.
\end{equation}

Now, we look at $(iv)$ and exhibit a lower bound. We split the integrand $A$ into two parts, the first part corresponds to terms of order one of $\phi_{x_1x_1}$, the second corresponds to the terms of superior orders, as expressed in (\ref{estestestphix}). We denote :

\begin{equation}\label{touslesi}
\begin{aligned}
&(iv)_1 = \int\int_\iota \frac{ C_\varepsilon^1 x_1 + C_\varepsilon^2 (x_2-\nu_\varepsilon^2) + C_\varepsilon^3 (t_\varepsilon-t) }{|\phi(t,\nu_\varepsilon^1+x_1,x_2) - \phi(t,\nu_\varepsilon^2+y,x_2)|^{1/2-2\lambda} } \\
&\hspace{7cm}\cdot\frac{ C_\varepsilon^1 y + C_\varepsilon^2 (x_2-\nu_\varepsilon^2) + C_\varepsilon^3 (t_\varepsilon-t) }{\phi_{x_1}(y)^2 \phi_{x_1}(x_1)^2}          \\  
&(iv)_2 = \int\int_\iota \frac{ \sum_{i+j+k=2} (x_1)^i (x_2-\nu_\varepsilon^2)^j (t_\varepsilon-t)^k f_{i,j,k}(t,x_1,x_2) }{|\phi(t,\nu_\varepsilon^1+x_1,x_2) - \phi(t,\nu_\varepsilon^2+y,x_2)|^{1/2-2\lambda} } \\
&\hspace{7cm}\cdot\frac{ C_\varepsilon^1 y + C_\varepsilon^2 (x_2-\nu_\varepsilon^2) + C_\varepsilon^3 (t_\varepsilon-t) }{\phi_{x_1}(y)^2 \phi_{x_1}(x_1)^2}          \\
&(iv)_3 = \int\int_\iota \frac{ C_\varepsilon^1 x_1 + C_\varepsilon^2 (x_2-\nu_\varepsilon^2) + C_\varepsilon^3 (t_\varepsilon-t) }{|\phi(t,\nu_\varepsilon^1+x_1,x_2) - \phi(t,\nu_\varepsilon^2+y,x_2)|^{1/2-2\lambda} }\\
&\hspace{5cm}\cdot \frac{ \sum_{i+j+k=2} (y)^i (x_2-\nu_\varepsilon^2)^j (t_\varepsilon-t)^k f_{i,j,k}(t,x_1,x_2) }{\phi_{x_1}(y)^2 \phi_{x_1}(x_1)^2}          \\  
&(iv)_4 = \int\int_\iota \frac{ \sum_{i+j+k=2} (x_1)^i (x_2-\nu_\varepsilon^2)^j (t_\varepsilon-t)^k f_{i,j,k}(t,x_1,x_2)  }{|\phi(t,\nu_\varepsilon^1+x_1,x_2) - \phi(t,\nu_\varepsilon^2+y,x_2)|^{1/2-2\lambda} } \\
&\hspace{5cm}\cdot\frac{ \sum_{i+j+k=2} (y)^i (x_2-\nu_\varepsilon^2)^j (t_\varepsilon-t)^k f_{i,j,k}(t,x_1,x_2) }{\phi_{x_1}(y)^2 \phi_{x_1}(x_1)^2}
\end{aligned}
\end{equation}

We will show that 

\begin{multline}\label{desix2}
\left| \int_{x_2 = \nu_\varepsilon^2 - 2\delta_\varepsilon}^{\nu_\varepsilon^2+2\delta_\varepsilon} \psi_\varepsilon^2(x_2) (iv)_1 \right| >> \left| \int_{x_2 = \nu_\varepsilon^2 - 2\delta_\varepsilon}^{\nu_\varepsilon^2+2\delta_\varepsilon} \psi_\varepsilon^2(x_2) (iv)_2 \right| \\
 + \left| \int_{x_2 = \nu_\varepsilon^2 - 2\delta_\varepsilon}^{\nu_\varepsilon^2+2\delta_\varepsilon} \psi_\varepsilon^2(x_2) (iv)_3 \right| + \left| \int_{x_2 = \nu_\varepsilon^2 - 2\delta_\varepsilon}^{\nu_\varepsilon^2+2\delta_\varepsilon} \psi_\varepsilon^2(x_2) (iv)_4 \right|, \hspace{0.2cm} \text{ as } t\rightarrow t_\varepsilon.
\end{multline}

First, we look at $(iv)_1$. 

We more precisely consider the integral

\begin{equation}
(iv)_1^1= \int\int_\iota \frac{ x_1 }{|\phi(t,\nu_\varepsilon^1+x_1,x_2) - \phi(t,\nu_\varepsilon^2+y,x_2)|^{1/2-2\lambda} } \cdot\frac{ y }{\phi_{x_1}(y)^2 \phi_{x_1}(x_1)^2}
\end{equation}

Again, because of $(y-\nu_\varepsilon)$ anti-symmetric properties but the fact that $\frac{1}{|\phi(t,x_1) - \phi(t,y)|^{1/2-2\lambda}}$ will concentrate the weight near $x_1 = y$, we have to split the domain into four parts.

\begin{equation}\label{domains3}
\begin{aligned}
&\alpha_{x_2} = \{(x_1,y) \in [\zeta_\varepsilon^{2+},\nu_\varepsilon] \times [\nu_\varepsilon-\eta_\varepsilon,\nu_\varepsilon] \}, \hspace{0.2cm} \beta_{x_2} = \{(x_1,y) \in [\nu_\varepsilon,\zeta_\varepsilon^{3-}] \times [\nu_\varepsilon-\eta_\varepsilon,\nu_\varepsilon] \}\\
&\gamma_{x_2} = \{(x_1,y) \in [\zeta_\varepsilon^{2+},\nu_\varepsilon] \times [\nu_\varepsilon,\nu_\varepsilon+\eta_\varepsilon] \}, \hspace{0.2cm} \delta_{x_2} = \{(x_1,y) \in [\nu_\varepsilon,\zeta_\varepsilon^{3-}] \times [\nu_\varepsilon,\nu_\varepsilon+\eta_\varepsilon] \}\\
&\alpha = \alpha_{x_2} \times [\nu_\varepsilon^2 - 2\delta_\varepsilon,\nu_\varepsilon^2+2\delta_\varepsilon], \hspace{0.2cm} \beta = \beta_{x_2} \times [\nu_\varepsilon^2 - 2\delta_\varepsilon,\nu_\varepsilon^2+2\delta_\varepsilon], \hspace{0.2cm} \\
&\gamma = \gamma_{x_2} \times [\nu_\varepsilon^2 - 2\delta_\varepsilon,\nu_\varepsilon^2+2\delta_\varepsilon], \hspace{0.2cm} \delta = \delta_{x_2} \times [\nu_\varepsilon^2 - 2\delta_\varepsilon,\nu_\varepsilon^2+2\delta_\varepsilon]
\end{aligned}
\end{equation}

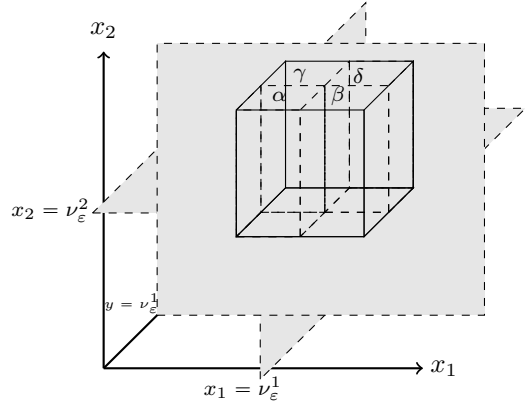
\begin{figure}[h]\label{alphastar3d}\centering 
\begin{tikzpicture}[scale=0.7]
	\draw[thick,->] (0,0,0) -- (6,0,0);
	\draw (0,0,-5.4) node {$y$};
	\draw[thick,->] (0,0,0) -- (0,6,0);
	\draw (6.4,0,0) node {$x_1$};
	\draw[thick,->] (0,0,0) -- (0,0,-5);
	\draw (0,6.4,0) node {$x_2$};

	\draw[fill=black!10,dashed,opacity=0.1] (-0.15,3,0.15) -- (6,3,0.15) -- (6,3,-5) -- (-0.15,3,-5) -- cycle;
	\draw (-1,3,0) node {\footnotesize $x_2 = \nu_\varepsilon^2$};
	
	\draw[fill=black!10,opacity=0.1,dashed] (3,-0.15,0.15) -- (3,5,0.15) -- (3,5,-5) -- (3,-0.15,-5) -- cycle;
	\draw (3,0,1) node {\footnotesize  $x_1 = \nu_\varepsilon^1$};

	\draw[dashed,fill=black!10,opacity=0.1] (-0.15,-0.15,-3) -- (-0.15,5,-3) -- (6,5,-3) -- (6,-0.15,-3) -- cycle;
	\draw (-0.7,0,-3.2) node {\tiny $y = \nu_\varepsilon^1$};	

	\draw (1.8,1.8,-1.8) -- (4.2,1.8,-1.8) -- (4.2,1.8,-4.2) -- (1.8,1.8,-4.2) -- cycle;
	\draw (1.8,1.8,-1.8) -- (1.8,4.2,-1.8) -- (1.8,4.2,-4.2) -- (1.8,1.8,-4.2) -- cycle;
	\draw (1.8,1.8,-1.8) -- (4.2,1.8,-1.8) -- (4.2,4.2,-1.8) -- (1.8,4.2,-1.8) -- cycle;
	\draw (1.8,1.8,-4.2) -- (4.2,1.8,-4.2) -- (4.2,4.2,-4.2) -- (1.8,4.2,-4.2) -- cycle;
	\draw (4.2,1.8,-1.8) -- (4.2,4.2,-1.8) -- (4.2,4.2,-4.2) -- (4.2,1.8,-4.2) -- cycle;
	
\pgfmathsetmacro{\cubex}{1.8}
\pgfmathsetmacro{\cubey}{1.8}
\pgfmathsetmacro{\cubez}{-1.8}
\pgfmathsetmacro{\cubx}{3}
\pgfmathsetmacro{\cuby}{4.2}
\pgfmathsetmacro{\cubz}{-3}

	\draw[dashed] (\cubex,\cubey,\cubez) -- (\cubx,\cubey,\cubez) -- (\cubx,\cubey,\cubz) -- (\cubex,\cubey,\cubz) -- cycle;
	\draw[dashed] (\cubex,\cubey,\cubez) -- (\cubex,\cuby,\cubez) -- (\cubex,\cuby,\cubz) -- (\cubex,\cubey,\cubz) -- cycle;
	\draw[dashed] (\cubex,\cubey,\cubez) -- (\cubx,\cubey,\cubez) -- (\cubx,\cuby,\cubez) -- (\cubex,\cuby,\cubez) -- cycle;
	\draw[dashed] (\cubex,\cubey,\cubz) -- (\cubx,\cubey,\cubz) -- (\cubx,\cuby,\cubz) -- (\cubex,\cuby,\cubz) -- cycle;
	\draw[dashed] (\cubx,\cubey,\cubez) -- (\cubx,\cuby,\cubez) -- (\cubx,\cuby,\cubz) -- (\cubx,\cubey,\cubz) -- cycle;

\pgfmathsetmacro{\cubex}{3}
\pgfmathsetmacro{\cubey}{1.8}
\pgfmathsetmacro{\cubez}{-3}
\pgfmathsetmacro{\cubx}{4.2}
\pgfmathsetmacro{\cuby}{4.2}
\pgfmathsetmacro{\cubz}{-4.2}

	\draw[dashed] (\cubex,\cubey,\cubez) -- (\cubx,\cubey,\cubez) -- (\cubx,\cubey,\cubz) -- (\cubex,\cubey,\cubz) -- cycle;
	\draw[dashed] (\cubex,\cubey,\cubez) -- (\cubex,\cuby,\cubez) -- (\cubex,\cuby,\cubz) -- (\cubex,\cubey,\cubz) -- cycle;
	\draw[dashed] (\cubex,\cubey,\cubez) -- (\cubx,\cubey,\cubez) -- (\cubx,\cuby,\cubez) -- (\cubex,\cuby,\cubez) -- cycle;
	\draw[dashed] (\cubex,\cubey,\cubz) -- (\cubx,\cubey,\cubz) -- (\cubx,\cuby,\cubz) -- (\cubex,\cuby,\cubz) -- cycle;
	\draw[dashed] (\cubx,\cubey,\cubez) -- (\cubx,\cuby,\cubez) -- (\cubx,\cuby,\cubz) -- (\cubx,\cubey,\cubz) -- cycle;

\draw (2.15,4,-3) node {\footnotesize $\alpha$};
\draw (3.25,4,-3) node {\footnotesize $\beta$};
\draw (2.15,4,-4) node {\footnotesize $\gamma$};
\draw (3.25,4,-4) node {\footnotesize $\delta$};

\end{tikzpicture}
\caption{Definition of $\alpha$, $\beta$, $\gamma$ and $\delta$}
\end{figure}

And we have 
\begin{equation}
(iv)_1^1 = \int\int_{\alpha} i(x_1,y) + \int\int_{\beta} i(x_1,y) +\int\int_{\gamma} i(x_1,y) + \int\int_{\delta} i(x_1,y).
\end{equation}

We also make the change of variable $(x_1,y) = (x_1 - \nu_\varepsilon^1,y-\nu_\varepsilon^1)$ without relabeling the set nor the variables.

We will regroup the integral corresponding to $\delta$ and $\beta$. The symmetric case ($\gamma$ and $\alpha$) is identical. Because it is very similar to the computation we did in the previous chapters, we will skip a few details. We consider the term

\begin{multline}\label{Jdefinition22}
J(x_2)=\int\int_{\delta \cup \beta}  \frac{C_\varepsilon}{|\phi(t,x_1) - \phi(t,y) |^{1/2-2\lambda}} \frac{ x_1}{\phi_{y}(t,x_1)^2} \frac{ y}{\phi_{y}(t,y)^2} \\
= \int_{x_1=0}^{\kappa_\varepsilon} \int_{y=0}^{\kappa_\varepsilon} \Bigg[ \frac{y \cdot x_1}{|\phi(t,\nu_\varepsilon+x_1) - \phi(t,\nu_\varepsilon+y) |^{1/2-2\lambda}\phi_{y}(t,\nu_\varepsilon+x_1)^2\phi_{y}(t,\nu_\varepsilon+y)^2}  \\
-\frac{ y \cdot x_1 }{|\phi(t,\nu_\varepsilon+x_1) - \phi(t,\nu_\varepsilon-y) |^{1/2-2\lambda}\phi_{y}(t,\nu_\varepsilon-y)^2\phi_{y}(t,\nu_\varepsilon+x_1)^2}
 \Bigg]
\end{multline}

We similarly decompose the integrand as follows.

\begin{multline}
\frac{1}{|\phi(t,\nu_\varepsilon+x_1) - \phi(t,\nu_\varepsilon+y) |^{1/2-2\lambda}} \frac{\chi'(\nu_\varepsilon+x_1) \cdot x_1}{\phi_{y}(t,\nu_\varepsilon+x_1)^2} \frac{\chi'(\nu_\varepsilon+y) \cdot y}{\phi_{y}(t,\nu_\varepsilon+y)^2} \\
-\frac{1}{|\phi(t,\nu_\varepsilon+x_1) - \phi(t,\nu_\varepsilon-y) |^{1/2-2\lambda}} \frac{\chi'(\nu_\varepsilon+x_1) \cdot x_1}{\phi_{y}(t,\nu_\varepsilon+x_1)^2} \frac{\chi'(\nu_\varepsilon-y) \cdot y}{\phi_{y}(t,\nu_\varepsilon-y)^2} \\
= \Bigg[ \frac{1}{|\phi(t,\nu_\varepsilon+x_1) - \phi(t,\nu_\varepsilon+y) |^{1/2-2\lambda}} \frac{\chi'(\nu_\varepsilon+x_1) \cdot x_1}{\phi_{y}(t,\nu_\varepsilon+x_1)^2} \frac{\chi'(\nu_\varepsilon+y) \cdot y}{\phi_{y}(t,\nu_\varepsilon+y)^2} \\
- \frac{1}{|\phi(t,\nu_\varepsilon+x_1) - \phi(t,\nu_\varepsilon-y) |^{1/2-2\lambda}} \frac{\chi'(\nu_\varepsilon+x_1) \cdot x_1}{\phi_{y}(t,\nu_\varepsilon+x_1)^2} \frac{\chi'(\nu_\varepsilon+y) \cdot y}{\phi_{y}(t,\nu_\varepsilon+y)^2} \Bigg] \\
+ \Bigg[ \frac{1}{|\phi(t,\nu_\varepsilon+x_1) - \phi(t,\nu_\varepsilon-y) |^{1/2-2\lambda}} \frac{\chi'(\nu_\varepsilon+x_1) \cdot x_1}{\phi_{y}(t,\nu_\varepsilon+x_1)^2} \frac{\chi'(\nu_\varepsilon+y) \cdot y}{\phi_{y}(t,\nu_\varepsilon+y)^2} \\
- \frac{1}{|\phi(t,\nu_\varepsilon+x_1) - \phi(t,\nu_\varepsilon-y) |^{1/2-2\lambda}} \frac{\chi'(\nu_\varepsilon+x_1) \cdot x_1}{\phi_{y}(t,\nu_\varepsilon+x_1)^2} \frac{\chi'(\nu_\varepsilon+y) \cdot y}{\phi_{y}(t,\nu_\varepsilon-y)^2} \Bigg] \\
=D_1 + D_2.
\end{multline}

To shorten a bit the notations, we will denote 

\begin{equation}
\alpha_{\pm} = \left| \phi_\varepsilon(t,\nu_\varepsilon+x_1) - \phi_\varepsilon(t,\nu_\varepsilon\pm y) \right|.
\end{equation}

For $D_1$, we write

\begin{multline}
\frac{1}{|\phi_\varepsilon(t,\nu_\varepsilon+x_1) - \phi_\varepsilon(t,\nu_\varepsilon+y) |^{1/2-2\lambda}} - \frac{1}{|\phi_\varepsilon(t,\nu_\varepsilon+x_1) - \phi_\varepsilon(t,\nu_\varepsilon-y) |^{1/2-2\lambda}} \\
= \frac{\alpha_- - \alpha_+}{\alpha_+^{1/2-2\lambda} \alpha_-^{1/2-2\lambda}\left(\alpha_+^{1/2+2\lambda} + \alpha_-^{1/2+2\lambda}\right)} + \frac{\alpha_+^{4\lambda} - \alpha_-^{4\lambda}}{\left(\alpha_+^{1/2+2\lambda} + \alpha_-^{1/2+2\lambda}\right)}
\end{multline}

It is clear that $D_1$ is non-negative when $x_1\geq y$. We now provide a lower bound for $D_{1,x_1 \leq y}$. We will show that it is positive up to a smaller order term, and provide a lower bound for the positive term. We hence now consider $x_1 \leq y$.

We write, 

\begin{multline}\label{lowerlol9}
\alpha_- - \alpha_+ = \left| \phi_\varepsilon(t,\nu_\varepsilon+x_1)-\phi_\varepsilon(t,\nu_\varepsilon-y) \right| - \left| \phi_\varepsilon(t,\nu_\varepsilon+x_1) - \phi_\varepsilon(t,\nu_\varepsilon+y) \right| \\
= 2 \phi_\varepsilon(t,\nu_\varepsilon+x_1) - \phi_\varepsilon(t,\nu_\varepsilon+y) - \phi_\varepsilon(t,\nu_\varepsilon-y) = \int_{s=-y}^{x_1} \phi_{\varepsilon,y}(t,\nu_\varepsilon+s) - \int_{s=x_1}^y \phi_{\varepsilon,y}(t,\nu_\varepsilon+s) \\
= \int_{s=-x_1}^{x_1} \phi_{\varepsilon,y}(t,\nu_\varepsilon+s) + \int_{s=x_1}^y \left( \phi_{\varepsilon,y} (t,\nu_\varepsilon-s) - \phi_{\varepsilon,y}(t,\nu_\varepsilon+s) \right).
\end{multline}

We will use this idea to provide a lower bound for $D_1$. 

\begin{multline}\label{unlabellabel2}
\frac{1}{\alpha_+^{1/2-2\lambda}} - \frac{1}{\alpha_-^{1/2-2\lambda}} = \int_{s=\alpha_-}^{\alpha_+} -(\frac{1}{2}-2\lambda) \frac{1}{s^{3/2-2\lambda}} = C \int_{s=\alpha_+}^{\alpha_-} s^{-3/2+2\lambda}\\
= C \int_{s=\alpha_+}^{\alpha_+ + \int_{z=-x_1}^{x_1} \phi_{\varepsilon,y}(t,\nu_\varepsilon+z) } s^{-3/2+2\lambda} \\
+ C\int_{s=\alpha_+ + \int_{z=-x_1}^{x_1}\phi_{\varepsilon,y}(t,\nu_\varepsilon+z) }^{\alpha_+ + \int_{z=-x_1}^{x_1} \phi_{\varepsilon,y}(t,\nu_\varepsilon+z) +\int_{z=x_1}^y \left( \phi_{\varepsilon,y} (t,\nu_\varepsilon-z) - \phi_{\varepsilon,y}(t,\nu_\varepsilon+z) \right)} s^{-3/2+2\lambda} = (i) + (ii),
\end{multline}

where $(i)$ is non-negative and $(ii)$ is small. First, we make the following upper bound for $|(ii)|$. If $\int_{x_1}^y  \left( \phi_{\varepsilon,y}(t,\nu_\varepsilon-s) - \phi_{\varepsilon,y}(t,\nu_\varepsilon+s) \right)>0$, then $(ii)$ is non-negative. Otherwise, we have $\alpha_{+}+\int_{s=-x_1}^{x_1} \phi_{\varepsilon,y}(t,\nu_\varepsilon+s)+\int_{s=x_1}^y  \left( \phi_{\varepsilon,y}(t,\nu_\varepsilon-s) - \phi_{\varepsilon,y}(t,\nu_\varepsilon+s) \right)\leq \alpha_{+}+\int_{s=-x_1}^{x_1} \phi_{\varepsilon,y}(t,\nu_\varepsilon+s)$. We then proceed as follows 

\begin{multline}
\left| \int_{s=\alpha_{+}+\int_{-x_1}^{x_1} \phi_{\varepsilon,y}(t,\nu_\varepsilon+s)}^{\alpha_{+}+\int_{z=-x_1}^{x_1} \phi_{\varepsilon,y}(t,\nu_\varepsilon+z)+\int_{z=x_1}^y  \left( \phi_{\varepsilon,y}(t,\nu_\varepsilon-z) - \phi_{\varepsilon,y}(t,\nu_\varepsilon+z) \right)}  s^{-3/2-2\lambda} \right| \\
= \left| \int_{s=\alpha_-}^{\alpha_- - \int_{z=x_1}^y  \left( \phi_{\varepsilon,y}(t,\nu_\varepsilon-z) - \phi_{\varepsilon,y}(t,\nu_\varepsilon+z) \right)} s^{-3/2+2\lambda} \right| \\
\leq \left| \int_{z=x_1}^y  \left( \phi_{\varepsilon,y}(t,\nu_\varepsilon-z) - \phi_{\varepsilon,y}(t,\nu_\varepsilon+z) \right) \right| \cdot \alpha_-^{-3/2+2\lambda}. 
\end{multline}

From the Taylor expansion, we obtain for $x,y$ small enough (only depending on $\varepsilon$),

\begin{equation}
\left| \int_{z=x_1}^y  \left( \phi_{\varepsilon,y}(t,\nu_\varepsilon-s) - \phi_{\varepsilon,y}(t,\nu_\varepsilon+s) \right) \right| \leq C_1 \left( y^4 -x_1^4 + (t_\varepsilon-t)^2 + (t-t_\varepsilon)(y^2-x_1^2) \right). 
\end{equation}

This means that we obtain (up to a non-negative contribution)

\begin{equation}
|(ii)| \leq C \left(y^4 + x_1^4 + (t_\varepsilon-t)\left( x_1^2 + y^2 \right) + (t_\varepsilon-t)^2 \right) \cdot \frac{1}{\alpha_-^{3/2-2\lambda}}.
\end{equation}

We now go on with the lower bound for $(i)$. Using the expression of $(i)$ provided by (\ref{unlabellabel2}) as well as the mean value theorem, we obtain

\begin{multline}\label{voiciunlabelnote5}
(i) = \int_{s=\alpha_+}^{\alpha_+ + \int_{z=-x_1}^{x_1} \phi_{\varepsilon,y}(t,\nu_\varepsilon+z) } s^{-3/2+2\lambda} \\
 \geq \left( \int_{z=-x_1}^{x_1} \phi_{\varepsilon,y}(t,\nu_\varepsilon+z) \right) \cdot \frac{1}{ \left(\phi_\varepsilon(t,\nu_\varepsilon+y) - \phi_{\varepsilon} (t,\nu_\varepsilon-x_1) \right)^{3/2-2\lambda}}. \\
\geq C x_1 \phi_{\varepsilon,y}(t,c_1) \cdot \frac{1}{(y+x_1)^{3/2-2\lambda} \cdot \phi_{\varepsilon,y}(t,c_2)^{3/2-2\lambda}}.
\end{multline}

Now, since we have $x_1\leq y$, we obtain

\begin{equation}\label{voiciunlabelnote6}
\phi_{\varepsilon,y}(t,c_1) \geq C (t_\varepsilon-t),
\end{equation}

and

\begin{equation}\label{voiciunlabelnote7}
\phi_{\varepsilon,y}(t,c_2) \leq C \left( (t_\varepsilon-t) + y^2\right).
\end{equation}

Using (\ref{voiciunlabelnote6}) and (\ref{voiciunlabelnote7}) inside of (\ref{voiciunlabelnote5}), we obtain

\begin{equation}
(i) \geq C \frac{x_1 (t_\varepsilon-t)}{(y+x_1)^{3/2-2\lambda} \cdot ((t_\varepsilon-t) + y^2 + (x_2-\nu_\varepsilon^2)^2 )^{3/2-2\lambda} }.
\end{equation}

Now, we obtain for $D_1$,

\begin{multline}\label{integrand22}
D_1 = D_{x_1\geq y} + D_{x_1 \leq y} \geq D_{x_1\leq y} \\
\geq C \int\int_{x_1\leq y} \frac{x_1 (t_\varepsilon-t)}{(y+x_1)^{3/2-2\lambda} \cdot ((t_\varepsilon-t) + y^2)^{3/2-2\lambda} } \frac{\chi'(\nu_\varepsilon+x_1) \cdot x_1}{\phi_{\varepsilon,y}(t,\nu_\varepsilon+x_1)^2} \frac{\chi'(\nu_\varepsilon+y) \cdot y}{\phi_{\varepsilon,y}(t,\nu_\varepsilon+y)^2} \\
- C \int_{x_1=0}^{\kappa_\varepsilon}\int_{y=x_1}^{\kappa_\varepsilon} \frac{x^4}{\alpha_{-}^{3/2-2\lambda}}\frac{\chi'(\nu_\varepsilon+x_1) \cdot x_1}{\phi_{\varepsilon,y}(t,\nu_\varepsilon+x_1)^2} \frac{\chi'(\nu_\varepsilon+y) \cdot y}{\phi_{\varepsilon,y}(t,\nu_\varepsilon+y)^2} \\
- C \int_{x_1=0}^{\kappa_\varepsilon}\int_{y=x_1}^{\kappa_\varepsilon} \frac{y^4}{\alpha_{-}^{3/2-2\lambda}}\frac{\chi'(\nu_\varepsilon+x_1) \cdot x_1}{\phi_{\varepsilon,y}(t,\nu_\varepsilon+x_1)^2} \frac{\chi'(\nu_\varepsilon+y) \cdot y}{\phi_{\varepsilon,y}(t,\nu_\varepsilon+y)^2} \\
- C \int_{x_1=0}^{\kappa_\varepsilon} \int_{y=x_1}^{\kappa_\varepsilon} \frac{(t_\varepsilon-t)\cdot x_1^2}{\alpha_{-}^{3/2-2\lambda}}\frac{\chi'(\nu_\varepsilon+x_1) \cdot x_1}{\phi_{\varepsilon,y}(t,\nu_\varepsilon+x_1)^2} \frac{\chi'(\nu_\varepsilon+y) \cdot y}{\phi_{\varepsilon,y}(t,\nu_\varepsilon+y)^2} \\
- C \int_{x_1=0}^{\kappa_\varepsilon} \int_{y=x_1}^{\kappa_\varepsilon} \frac{(t_\varepsilon-t)\cdot y^2}{\alpha_{-}^{3/2-2\lambda}}\frac{\chi'(\nu_\varepsilon+x_1) \cdot x_1}{\phi_{\varepsilon,y}(t,\nu_\varepsilon+x_1)^2} \frac{\chi'(\nu_\varepsilon+y) \cdot y}{\phi_{\varepsilon,y}(t,\nu_\varepsilon+y)^2} \\
- C \int_{x_1=0}^{\kappa_\varepsilon} \int_{y=x_1}^{\kappa_\varepsilon} \frac{(t_\varepsilon-t)^2}{\alpha_{-}^{3/2-2\lambda}}\frac{\chi'(\nu_\varepsilon+x_1) \cdot x_1}{\phi_{\varepsilon,y}(t,\nu_\varepsilon+x_1)^2} \frac{\chi'(\nu_\varepsilon+y) \cdot y}{\phi_{\varepsilon,y}(t,\nu_\varepsilon+y)^2} \\
= A_1 - B_1 - B_2 - B_3 - B_4 - B_5.
\end{multline}

Note that the proof also works for $\lambda=0$. We will now show that $A_1 \rightarrow \infty$, and that $J_2$, $J_3$, $J_4$, $D_2$, $B_1$, $B_2$, $B_3$, $B_4$ and $B_5$ are of a smaller order. We first consider $A_1$.
Here, we will consider $A_1$ because it gives the order of the main contribution, and showing that the other terms are of a smaller order is similar to what has been previously done. We will also consider the term $B_1$ to show how we deal with the smaller order terms. We will not consider $B_i$ for $i>1$ because it is similar to what have been previously done.

Using $|\phi_{y}(t,c(x_1,y))| \leq M_\varepsilon \left((x_1-\nu_\varepsilon^1)^2 + (x_2 - \nu_\varepsilon^2)^2 + (t_\varepsilon-t)\right)$, on $x_1 > y$, we have from (\ref{integrand22}) and the new change of variables $(x,y) = (r+z,r-z)$,

\begin{multline}\label{integrand7bis}
A_1 \geq  C_\varepsilon \int_{r=0}^{\kappa_\varepsilon/2} \int_{z=0}^r \frac{(r-z)(t_\varepsilon-t)}{r^{3/2-2\lambda}} \frac{r+z}{\left( \left( r+z \right)^2 + (x_2-\nu_\varepsilon^2)^2+ (t_\varepsilon-t) \right)^{7/2-2\lambda}} \\
\cdot \frac{r-z}{\left( \left( r-z \right)^2 + (x_2-\nu_\varepsilon^2)^2+ (t_\varepsilon-t) \right)^{2}} \\
\geq C_\varepsilon \int_{r=0}^{\kappa_\varepsilon/2} \frac{1}{r^{1/2-2\lambda}} \frac{(t_\varepsilon-t)}{\left( \left( 2r \right)^2 + (x_2-\nu_\varepsilon^2)^2 + (t_\varepsilon-t) \right)^{11/2-2\lambda}} \int_{z=0}^r (r-z)^2 \\
\geq C_\varepsilon \int_{r=0}^{\kappa_\varepsilon/2} \frac{(t_\varepsilon-t)\cdot r^{5/2+2\lambda}}{\left( \left( 2r \right)^2 + (x_2-\nu_\varepsilon^2)^2 + (t_\varepsilon-t) \right)^{11/2-2\lambda}} \geq \frac{C_\varepsilon \cdot (t_\varepsilon-t)}{(t_\varepsilon-t)^{15/4-3\lambda}}\\
\geq \frac{C_\varepsilon}{\left( (x_2-\nu_\varepsilon^2)^2 + (t_\varepsilon-t) \right)^{11/4-3\lambda}}.
\end{multline}

Hence, integrating (\ref{integrand7bis}) with respect to $x_2$ and using the properties of $\psi_\varepsilon^2$ given in (\ref{proppsi2}) yield

\begin{multline}\label{alphatroisieme7}
\int_{x_2=\nu_\varepsilon^2 - 2\delta_\varepsilon}^{\nu_\varepsilon^2 + 2\delta_\varepsilon} \int\int_\alpha \psi_\varepsilon^2(x_2) i(x_1,y,x_2) \geq \int_{x_2=\nu_\varepsilon^2 - \delta_\varepsilon}^{\nu_\varepsilon^2 + \delta_\varepsilon} \int\int_\alpha i(x_1,y,x_2) \\
\geq \int_{x_2=\nu_\varepsilon^2-\delta_\varepsilon}^{\nu_\varepsilon^2+\delta_\varepsilon} \frac{M_\varepsilon}{\left( (x_2-\nu_\varepsilon^2)^2 + (t_\varepsilon-t) \right)^{11/4-3\lambda}} \geq \frac{M_\varepsilon}{(t_\varepsilon-t)^{9/4-3\lambda}}.
\end{multline}

We now proceed with $B_1$. The case of $B_2$ is similar. We have since $\phi_{\varepsilon,y}(t,x) \geq (t_\varepsilon-t)$, with the mean value theorem

\begin{multline}\label{lacestunb1lol7}
|B_1| \leq \frac{C_\varepsilon}{(t_\varepsilon-t)^{3/2-2\lambda}} \int_{r=0}^{\kappa_\varepsilon/2} \int_{z=0}^r  \frac{(r-z)^4}{r^{3/2-2\lambda} } \\
\cdot  \frac{r+z}{\left( \left( r+z \right)^2+ (x_2-\nu_\varepsilon^2)^2 + (t_\varepsilon-t) \right)^{2}} \cdot \frac{r-z}{\left( \left( r-z \right)^2 + (x_2-\nu_\varepsilon^2)^2+ (t_\varepsilon-t) \right)^{2}} \\
\leq \frac{C_\varepsilon}{(t_\varepsilon-t)^{3/2-2\lambda}} \int_{r=0}^{\kappa_\varepsilon/2}  \frac{r^4}{r^{1/2-2\lambda} \left(r^2+(x_2-\nu_\varepsilon^2)^2 + (t_\varepsilon-t)\right)^2} \\
\cdot \int_{z=0}^r \frac{r-z}{ \left( (r-z)^2 + (x_2-\nu_\varepsilon^2)^2 + (t_\varepsilon-t) \right)^2}.
\end{multline}

Now, because

\begin{equation}
\int_{s=0}^\infty \frac{s}{\left(s^2 + (t_\varepsilon-t) \right)^2} \leq \frac{C}{(t_\varepsilon-t)},
\end{equation}

(\ref{lacestunb1lol7}) yields

\begin{multline}
|B_1| \leq \frac{C_\varepsilon}{(t_\varepsilon-t)^{3/2-2\lambda}} \cdot \frac{1}{\left( (x_2-\nu_\varepsilon^2)^2 + (t_\varepsilon-t) \right)} \int_{r=0}^{\kappa_\varepsilon/2}  \frac{r^{7/2+2\lambda}}{ \left(r^2+(t_\varepsilon-t)\right)^2} \\
 \leq \frac{C_\varepsilon}{(t_\varepsilon-t)^{3/2-2\lambda}} \cdot \frac{1}{\left( (t_\varepsilon-t) + (x_2-\nu_\varepsilon^2)^2 \right)}.
\end{multline}

Overall, we obtain 

\begin{equation}
B_1(x_2) \leq \frac{C_\varepsilon}{\left((t_\varepsilon-t)+(x_2-\nu_\varepsilon^2)^2\right)^{4/4}} \cdot \frac{1}{(t_\varepsilon-t)^{6/4-2\lambda}}.
\end{equation}

Hence, integrating with respect to $x_2$ and using the properties of $\psi_\varepsilon^2$ given in (\ref{proppsi2}) yield

\begin{equation}\label{alphatroisieme}
\int_{x_2=\nu_\varepsilon^2 - 2\delta_\varepsilon}^{\nu_\varepsilon^2 + 2\delta_\varepsilon} B_1(x_2) \leq \frac{C_\varepsilon}{(t_\varepsilon-t)^2}.
\end{equation}
In the end, we obtain

\begin{equation}
\int_{x_2} \int\int_{\beta\cup \delta} i(x_1,x_2,y) \geq \frac{C_\varepsilon}{(t_\varepsilon-t)^{9/4-3\lambda}}.
\end{equation}

We only studied $(iv)_1^1$. We now need to do an estimation for all the other terms that arise when we expand $(iv)_1$ whose expression is given by (\ref{touslesi}). 

First, we look at 

\begin{equation}
(iv)_1^2 = \int\int \frac{y(x_2-\nu_\varepsilon^2)}{|\phi(t,x_1,x_2) - \phi(t,y,x_2)|^{1/2-2\lambda}} \frac{1}{\phi_{x_1}^2(x_1) \phi_{x_1}^2(y)}.
\end{equation}

We will use the fact that the integrand is "almost" odd in $(x-\nu_\varepsilon^1,y-\nu_\varepsilon^1)$. We will now consider four domains $\iota^{+,+}$, $\iota^{+,-}$, $\iota^{-,-}$ and $\iota^{-,+}$ (for the new variables $y = y - \nu_\varepsilon^1$, $x= x - \nu_\varepsilon^1$.) They are defined as

\begin{equation}
\begin{aligned}
&\iota_{+,+} = \{x_1,y>0\}, \hspace{0.2cm} \iota_{+,-} = \{x_1>0,y<0\},\\
&\iota_{-,-} = \{x_1,y<0\}, \hspace{0.2cm} \iota_{-,+} = \{x_1<0,y>0\}.
\end{aligned}
\end{equation}

We will consider the two integrals 

\begin{equation}
\begin{aligned}
&I_1 = \int\int_{\iota^{+,+} \bigcup \iota^{-,-}} \frac{y(x_2-\nu_\varepsilon^2)}{ \phi_{x_1}(x_1)^2 \phi_{x_1}(y)^2 |\phi(x_1) - \phi(y)|^{1/2-2\lambda}}\\
&I_2 = \int\int_{\iota^{+,-} \bigcup \iota^{-,+}} \frac{y(x_2-\nu_\varepsilon^2)}{ \phi_{x_1}(x_1)^2 \phi_{x_1}(y)^2 |\phi(x_1) - \phi(y)|^{1/2-2\lambda}}.
\end{aligned}
\end{equation}

It will be easier to deal with $I_2$. For $I_1$ however, we will need some preliminaries. We first perform the change of variable for the part $x_1<0,y<0$ of the form $x=-x,y=-y$. We hence obtain for $I_1$

\begin{multline}
I_1 = \int\int_{\iota^{+,+}} y(x_2-\nu_\varepsilon^2) \\ \cdot \left[ \frac{1}{\phi_{x_1}(x_1)^2 \phi_{x_1}(y_1)^2 |\phi(x_1)-\phi(y)|^{1/2-2\lambda}} - \frac{1}{\phi_{x_1}(-x_1)^2 \phi_{x_1}(-y_1)^2 |\phi(-x_1)-\phi(-y)|^{1/2-2\lambda}} \right].
\end{multline}

We simplify (we add and subtract $\frac{1}{\phi_{x_1}(x_1)^2 \phi_{x_1}(y)^2 |\phi(-x)-\phi(-y)|^{1/2-2\lambda}}$) and obtain

\begin{multline}\label{iciilyadesjs}
I_1 = \int\int_{\iota^{+,+}} \frac{y(x_2-\nu_\varepsilon)}{\phi_{x_1}(x_1)^2 \phi_{x_1}(y)^2 }  \left[ \frac{1}{|\phi(x_1) - \phi(y)|^{1/2-2\lambda}} - \frac{1}{|\phi(-x)-\phi(-y)|^{1/2-2\lambda}} \right] \\
+ \int\int_{\iota^{+,+}} \frac{y(x_2-\nu_\varepsilon^2)}{|\phi(-x) - \phi(-y)|^{1/2-2\lambda}} \left[ \frac{1}{\phi_{x_1}(x_1)^2 \phi_{x_1}(y)^2} - \frac{1}{\phi_{x_1}(-x_1)^2 \phi_{x_1}(-y)^2} \right]\\
=J_1 + J_2.
\end{multline}

We denote for simplicity 

\begin{equation}
\beta_{\pm} = |\phi(\pm x) - \phi(\pm y)|.
\end{equation}

We consider the following simplification inside $J_1$. Also, we work on the set where $y>x$ by symmetry. Call this set $\iota_{y}^{+,+}$.

\begin{multline}\label{encoreunlabeln}
\frac{1}{\beta_{+}^{1/2-2\lambda}} - \frac{1}{\beta_{-}^{1/2-2\lambda}} 
= \frac{\beta_- - \beta_+ }{\beta_-^{1/2-2\lambda} \beta_+^{1/2-2\lambda}\left( \beta_-^{1/2+2\lambda} + -\beta_+^{1/2+2\lambda} \right)}+ \frac{\beta_+^{4\lambda} -\beta_-^{4\lambda} }{\beta_+^{1/2+2\lambda} + \beta_-^{1/2+2\lambda}}.
\end{multline}

Now, because of the symmetry of the problem, $\beta_- -\beta_+$ is of a smaller order. Indeed, considering that the Taylor expansion of $\phi_{x_1}$ around $\nu_\varepsilon^1$ is of the form (for the old variable $x_1$)

\begin{multline}\label{taylorok}
\phi_{x_1}(t,\nu_\varepsilon^1+y,x_2) = C_1 y^2 + C_2 (x_2-\nu_\varepsilon^2)^2 + C_3 (t_\varepsilon-t)\\
 + f_1(t,y,x_2) (x_2-\nu_\varepsilon^2)^3 + f_2(t,y,x_2) y^3 + f_3(t,y,x_2) (t_\varepsilon-t) y \\
 + f_4(t,y,x_2) (t_\varepsilon-t) (x_2-\nu_\varepsilon^2) + f_5(t,y,x_2) y (\nu_\varepsilon-x_2)^2 + f_6(t,y,x_2) (t_\varepsilon-t)^2 \\
 + f_7 (t,y,x_2) (\nu_\varepsilon-x_2) y^2.
\end{multline}

We now focus exclusively on the first term of (\ref{encoreunlabeln}). Call $J_1^1$ the corresponding integral.
We obtain that the term of order $2$ in the space variable and in order $1$ in the time variable cancel out. Indeed, for the first part, we have

\begin{multline}\label{encoredesbetas}
\frac{\beta_- - \beta_+ }{\beta_-^{1/2-2\lambda} \beta_+^{1/2-2\lambda}\left( \beta_-^{1/2+2\lambda} + -\beta_+^{1/2+2\lambda} \right)} = \frac{ \int_{s=\min(x,y)}^{\max(x,y)}  \left( \phi_{x_1}(s) - \phi_{x_1}(-s) \right) }{\beta_-^{1/2-2\lambda} \beta_+^{1/2-2\lambda}\left( \beta_-^{1/2+2\lambda} + -\beta_+^{1/2+2\lambda} \right)}
\end{multline}

and since we have, (we assume $x\leq y$ by symmetry), by (\ref{taylorok}),

\begin{multline}\label{unpetitordre}
\left| \int_{s=x}^{y}  \left( \phi_{x_1}(s) - \phi_{x_1}(-s) \right) \right| \leq C (y-x) |x_2-\nu_\varepsilon^2|^3 + C(y-x) y^3 + C(y-x) (t_\varepsilon-t)^2 \\
+ C(y-x) y(t_\varepsilon-t) + C(y-x) |x_2-\nu_\varepsilon^2|(t_\varepsilon-t).
\end{multline}

Plugging (\ref{encoredesbetas}) and (\ref{unpetitordre}) into the definition of $J_1^1$ gives

\begin{multline}
J_1^1 \simeq \int\int_{\iota^{+,+}} \frac{C y(x_2-\nu_\varepsilon^2) (x-y)^4 \phi_{x_1}(x_1)^{-2} \phi_{x_1}(y)^{-2}}{|\phi(-y)-\phi(-x)|^{1/2-2\lambda}|\phi(-x)-\phi(-y)|^{1/2-2\lambda}}\\
 \frac{1}{ \left(|\phi(y)-\phi(x)|^{1/2+2\lambda} + |\phi(-x)-\phi(-y)|^{1/2+2\lambda} \right)}.
\end{multline}

By symmetry, we again assume that $y\geq 0$, and perform the change of variable $x = r-z$, $y=r+z$. ($z\in [0,r]$).

Now, we have 

\begin{multline}
|J_{1}^1| \leq C |x_2-\nu_\varepsilon^2| \int\int \frac{(r+z) (2z)^4 }{(2z)^{1.5-2\lambda} \phi_{x_1}(y)^{3.5-2\lambda} \phi_{x_1}(x_1)^2 } \\
 \leq  \int_{r} \int_{z=0}^r  \frac{C |x_2-\nu_\varepsilon^2| (r+z) (2z)^4 }{(2z)^{1.5-2\lambda} \left( (r+z)^2 + (x_2-\nu_\varepsilon^2)^2 + (t_\varepsilon-t) \right)^{3.5-2\lambda} \left( (r-z)^2 + (x_2-\nu_\varepsilon^2)^2 + (t_\varepsilon-t) \right)^{2}} \\
\leq \int_{r} \int_{z=0}^{r/2}  \frac{C |x_2-\nu_\varepsilon^2| r (2z)^4 }{(2z)^{1.5-2\lambda} \left( (r)^2 + (x_2-\nu_\varepsilon^2)^2 + (t_\varepsilon-t) \right)^{5.5-2\lambda}} \\
+ \int_{r} \int_{z=r/2}^r  \frac{C |x_2-\nu_\varepsilon^2| (r) (z)^4 }{(r)^{1.5-2\lambda} \left( (r)^2 + (x_2-\nu_\varepsilon^2)^2 + (t_\varepsilon-t) \right)^{3.5-2\lambda} \left( (r-z)^2 + (x_2-\nu_\varepsilon^2)^2 + (t_\varepsilon-t) \right)^{2}} \\
= (i) + (ii).
\end{multline}

We obtain for $(i)$, 

\begin{multline}
(i) \leq \int_{r} C |x_2-\nu_\varepsilon^2| r \int_{z=0}^{r/2}  \frac{C  (2z)^{2.5+2\lambda }}{\left( (r)^2 + (x_2-\nu_\varepsilon^2)^2 + (t_\varepsilon-t) \right)^{5.5-2\lambda}} \\
\leq \int_{r} |x_2-\nu_\varepsilon^2| r^{4.5+2\lambda}   \frac{C}{\left( (r)^2 + (x_2-\nu_\varepsilon^2)^2 + (t_\varepsilon-t) \right)^{5.5-2\lambda}}\\
\leq |x_2-\nu_\varepsilon^2|  \frac{C}{\left( (x_2-\nu_\varepsilon^2)^2 + (t_\varepsilon-t) \right)^{11/4-\lambda}}.
\end{multline}

Now, integrating with respect to $x_2$ yields

\begin{equation}
\int_{x_2} \psi_\varepsilon^2(x_2) |(i)| \leq \frac{C}{(t_\varepsilon-t)^{7/4}},
\end{equation}

which is indeed of a smaller order.

Going on with $(ii)$, we get

\begin{multline}
(ii) \leq \int_{r} \frac{C |x_2-\nu_\varepsilon^2| (r)^5}{(r)^{1.5-2\lambda} \left( (r)^2 + (x_2-\nu_\varepsilon^2)^2 + (t_\varepsilon-t) \right)^{3.5-2\lambda}}  \\
\cdot \int_{z=0}^{r/2} \frac{1}{ \left( (z)^2 + (x_2-\nu_\varepsilon^2)^2 + (t_\varepsilon-t) \right)^{2}}\\
\leq \int_{r} \frac{C |x_2-\nu_\varepsilon^2| (r)^5}{(r)^{1.5-2\lambda} \left( (r)^2 + (x_2-\nu_\varepsilon^2)^2 + (t_\varepsilon-t) \right)^{3.5-2\lambda}} \cdot \frac{C}{\left( (x_2-\nu_\varepsilon^2)^2 + (t_\varepsilon-t) \right)}\\
\leq \frac{C}{\left( (x_2-\nu_\varepsilon^2)^2 + (t_\varepsilon-t) \right)^{11/4-\lambda}}.
\end{multline}

Now, integrating with respect to $x_2$ yields

\begin{equation}
\int_{x_2} \psi_\varepsilon^2(x_2) |(ii)| \leq \frac{C}{(t_\varepsilon-t)^{7/4}},
\end{equation}

which is indeed of a smaller order.

We go on with $J_1^2$, which corresponds to the second term in (\ref{encoreunlabeln}). Again, we will use the fact that

\begin{equation}
\beta_+^{4\lambda} - \beta_-^{4\lambda} = \int_{\beta_-}^{\beta_+} 4 \lambda s^{4\lambda-1}, 
\end{equation}

and that $|\beta_+-\beta_-|$ is of a smaller order, as we have previously shown. Indeed, we get

\begin{equation}
\left| \beta_-^{4\lambda} - \beta_+^{4\lambda} \right| \leq 4\lambda \left| \beta_- - \beta_+ \right|\cdot \left( |\beta_-|^{4\lambda-1} + |\beta_+|^{4\lambda-1} \right)
\end{equation}

The proof of this is very similar to what has been done to find the lower bound of $I_2$ so we will not do it explicitly here. In the end, the term is of a smaller order when $\lambda>0$.

Now, we go on with the term $J_2$ defined in (\ref{iciilyadesjs}). We will introduce the two difference terms

\begin{multline}
\frac{1}{\phi_{x_1}(x_1)^2 \phi_{x_1}(y)^2} - \frac{1}{\phi_{x_1}(-x_1)^2 \phi_{x_1}(-y)^2} \\
= \left( \frac{1}{\phi_{x_1}(x_1)^2 \phi_{x_1}(y)^2} - \frac{1}{\phi_{x_1}(x_1)^2 \phi_{x_1}(-y)^2} \right) \\ +  \left( \frac{1}{\phi_{x_1}(x_1)^2 \phi_{x_1}(-y)^2} - \frac{1}{\phi_{x_1}(-x_1)^2 \phi_{x_1}(-y)^2}  \right)
= j_2^1 + j_2^2,
\end{multline}

and define $J_2^1$ and $J_2^2$ as the two corresponding integrals.

Now for $J_2^1$, we write the simplification

\begin{equation}\label{destrucsencore}
\frac{1}{\phi_{x_1}(x_1)^2 \phi_{x_1}(y)^2} - \frac{1}{\phi_{x_1}(x_1)^2 \phi_{x_1}(-y)^2} = \frac{\left( \phi_{x_1}(-y) - \phi_{x_1} (y) \right) \left( \phi_{x_1}(-y) + \phi_{x_1}(y) \right)}{\phi_{x_1}(-y)^2 \phi_{x_1}(y)^2 \phi_{x_1}(x_1)^2}.
\end{equation}

Now, using (\ref{taylorok}) inside (\ref{destrucsencore}), we obtain that

\begin{multline}\label{mmm}
\left| J_2^1 \right| \leq  \int\int_{\iota^{+,+}} \frac{y(x_2-\nu_\varepsilon^2)}{|\phi(-x) - \phi(-y)|^{1/2-2\lambda}} \frac{\left| \phi_{x_1}(-y) - \phi_{x_1} (y) \right|}{\phi_{x_1}(-y)^2 \phi_{x_1}(y) \phi_{x_1}(x_1)^2}\\
+\int\int_{\iota^{+,+}} \frac{y(x_2-\nu_\varepsilon^2)}{|\phi(-x) - \phi(-y)|^{1/2-2\lambda}} \frac{\left| \phi_{x_1}(-y) - \phi_{x_1} (y) \right|}{\phi_{x_1}(-y) \phi_{x_1}(y) \phi_{x_1}(x_1)^2}. \\
\end{multline}

Because the two terms are similar, we only consider the first term of (\ref{mmm}). We get

\begin{multline}
\int\int_{\iota^{+,+}} \frac{y(x_2-\nu_\varepsilon^2)}{|\phi(-x) - \phi(-y)|^{1/2-2\lambda}} \frac{\left| \phi_{x_1}(-y) - \phi_{x_1} (y) \right|}{\phi_{x_1}(-y)^2 \phi_{x_1}(y) \phi_{x_1}(x_1)^2} \\
\leq \int\int_{\iota^{+,+}} \frac{y(x_2-\nu_\varepsilon^2)}{|\phi(-x) - \phi(-y)|^{1/2-2\lambda}} \frac{C_1 y^3}{\phi_{x_1}(-y)^2 \phi_{x_1}(y) \phi_{x_1}(x_1)^2} \\
+ \int\int_{\iota^{+,+}} \frac{y(x_2-\nu_\varepsilon^2)}{|\phi(-x) - \phi(-y)|^{1/2-2\lambda}} \frac{C_2 y^2 |x_2-\nu_\varepsilon^2|}{\phi_{x_1}(-y)^2 \phi_{x_1}(y) \phi_{x_1}(x_1)^2} \\
+ \int\int_{\iota^{+,+}} \frac{y(x_2-\nu_\varepsilon^2)}{|\phi(-x) - \phi(-y)|^{1/2-2\lambda}} \frac{C_3 y |x_2-\nu_\varepsilon|^2}{\phi_{x_1}(-y)^2 \phi_{x_1}(y) \phi_{x_1}(x_1)^2} \\
+ \int\int_{\iota^{+,+}} \frac{y(x_2-\nu_\varepsilon^2)}{|\phi(-x) - \phi(-y)|^{1/2-2\lambda}} \frac{C_4 |x_2-\nu_\varepsilon^2|^3}{\phi_{x_1}(-y)^2 \phi_{x_1}(y) \phi_{x_1}(x_1)^2} \\
+ \int\int_{\iota^{+,+}} \frac{y(x_2-\nu_\varepsilon^2)}{|\phi(-x) - \phi(-y)|^{1/2-2\lambda}} \frac{C_5 (t_\varepsilon-t)y}{\phi_{x_1}(-y)^2 \phi_{x_1}(y) \phi_{x_1}(x_1)^2} \\
+ \int\int_{\iota^{+,+}} \frac{y(x_2-\nu_\varepsilon^2)}{|\phi(-x) - \phi(-y)|^{1/2-2\lambda}} \frac{C_6 (t_\varepsilon-t)|x_2-\nu_\varepsilon^2|}{\phi_{x_1}(-y)^2 \phi_{x_1}(y) \phi_{x_1}(x_1)^2} \\
+ \int\int_{\iota^{+,+}} \frac{y(x_2-\nu_\varepsilon^2)}{|\phi(-x) - \phi(-y)|^{1/2-2\lambda}} \frac{C_7 (t_\varepsilon-t)^2}{\phi_{x_1}(-y)^2 \phi_{x_1}(y) \phi_{x_1}(x_1)^2} \\
=(i) + (ii) + (iii) + (iv) + (v) + (vi) + (vii).
\end{multline}

Now, the case of $(i)$, $(ii)$, $(iii)$ and $(iv)$ are similar, also $(v)$ and $(vi)$ are similar. This means that we will only consider $(i)$, $(v)$ and $(vii)$.

For $(i)$, we consider the same change of variable as previously. We obtain

\begin{multline}
(i) \leq \frac{C}{(t_\varepsilon-t)^{1/2-2\lambda}} \\
\cdot  \int_{r} \int_{z=0}^r \frac{(r+z)^4 |x_2-\nu_\varepsilon^2|}{z^{1/2-2\lambda} \left( (r+z)^2 + (x_2-\nu_\varepsilon^2)^2 + (t_\varepsilon-t) \right)^3 \left( (r-z)^2 + (x_2-\nu_\varepsilon^2)^2 + (t_\varepsilon-t) \right)^2} \\
\leq \frac{C}{(t_\varepsilon-t)^{1/2-2\lambda}} \int_r \int_{z=0}^{r/2} \frac{r^4 |x_2-\nu_\varepsilon^2| }{z^{1/2-2\lambda} \left( r^2 +(x_2-\nu_\varepsilon^2)^2 + (t_\varepsilon-t) \right)^5} \\
+ \frac{C}{(t_\varepsilon-t)^{1/2-2\lambda}} \int_{r}\int_{z=r/2}^r \frac{r^4 |x_2-\nu_\varepsilon|}{r^{1/2-2\lambda} \left( r^2 + (x_2-\nu_\varepsilon^2)^2 + (t_\varepsilon-t) \right)^3 \left( (r-z)^2 + (x_2-\nu_\varepsilon^2)^2 + (t_\varepsilon-t) \right)^2 }\\
= (i)_1 + (i)_2.
\end{multline}

For $(i)_1$, we get

\begin{multline}\label{mmmm}
(i)_1 \leq \frac{C}{(t_\varepsilon-t)^{1/2-2\lambda}} \int_r \frac{r^{4.5+2\lambda} |x_2-\nu_\varepsilon|}{(r^2 + (\nu_\varepsilon^2-x_2)^2 + (t_\varepsilon-t))^5} \\
\leq \frac{C}{(t_\varepsilon-t)^{1/2-2\lambda}} \frac{1}{\left( (x_2-\nu_\varepsilon^2)^2 + (t_\varepsilon-t) \right)^{9/4-3\lambda}}.
\end{multline}

Now, integrating (\ref{mmmm}) with respect to $x_2$ yields

\begin{equation}
\int_{x_2} |\psi_\varepsilon^2(x_2)| |(i)_1| \leq \frac{C}{(t_\varepsilon-t)^{7/4-3\lambda}},
\end{equation}

which is of a smaller order.

For $(i)_2$, we obtain

\begin{multline}\label{mmmmm}
(i)_2 \leq \frac{C}{(t_\varepsilon-t)^{1/2-2\lambda}} \int_r \frac{r^{3.5+2\lambda} |x_2-\nu_\varepsilon^2|}{\left( r^2 + (x_2 - \nu_\varepsilon^2)^2 + (t_\varepsilon-t) \right)^3} \\
\cdot \int_{z=0} \frac{1}{\left( z^2 + (x_2-\nu_\varepsilon^2)^2 + (t_\varepsilon-t) \right)^2}\\
\leq \frac{C}{(t_\varepsilon-t)^{1/2-2\lambda}} \int_r \frac{r^{3.5+2\lambda} |x_2-\nu_\varepsilon^2|}{\left( r^2 + (x_2 - \nu_\varepsilon^2)^2 + (t_\varepsilon-t) \right)^3} 
\cdot \frac{C}{((x_2-\nu_\varepsilon^2)^2 + (t_\varepsilon-t))^{3/2}}\\
\leq \frac{C}{(t_\varepsilon-t)^{1/2-2\lambda}} \frac{1}{\left( (x_2-\nu_\varepsilon^2)^2 + (t_\varepsilon-t) \right)^{9/4-3\lambda}}.
\end{multline}

Now, integrating (\ref{mmmmm}) with respect to $x_2$ yields

\begin{equation}
\int_{x_2} |\psi_\varepsilon^2(x_2)| |(i)_2| \leq \frac{C}{(t_\varepsilon-t)^{7/4-3\lambda}},
\end{equation}

which is of a smaller order. The same result holds for $(ii)$, $(iii)$ and $(iv)$. We go on and consider now $(v)$. We consider the same change of variable as previously. We obtain

\begin{multline}
(v) \leq \frac{C}{(t_\varepsilon-t)^{1/2-2\lambda}} \\
\cdot  \int_{r} \int_{z=0}^r \frac{(r+z)^2 |x_2-\nu_\varepsilon^2| (t_\varepsilon-t)}{z^{1/2-2\lambda} \left( (r+z)^2 + (x_2-\nu_\varepsilon^2)^2 + (t_\varepsilon-t) \right)^3 \left( (r-z)^2 + (x_2-\nu_\varepsilon^2)^2 + (t_\varepsilon-t) \right)^2} \\
\leq C (t_\varepsilon-t)^{1/2+2\lambda} \int_r \int_{z=0}^{r/2} \frac{r^2 |x_2-\nu_\varepsilon^2| }{z^{1/2-2\lambda} \left( r^2 +(x_2-\nu_\varepsilon^2)^2 + (t_\varepsilon-t) \right)^5} \\
+ C (t_\varepsilon-t)^{1/2+2\lambda} \int_{r}\int_{z=r/2}^r \frac{r^2 |x_2-\nu_\varepsilon|}{r^{1/2-2\lambda} \left( r^2 + (x_2-\nu_\varepsilon^2)^2 + (t_\varepsilon-t) \right)^3 \left( (r-z)^2 + (x_2-\nu_\varepsilon^2)^2 + (t_\varepsilon-t) \right)^2 }\\
= (v)_1 + (v)_2.
\end{multline}

For $(v)_1$, we get

\begin{multline}\label{mmmmmm}
(v)_1 \leq C (t_\varepsilon-t)^{1/2+2\lambda} \int_r \frac{r^{2.5+2\lambda} |x_2-\nu_\varepsilon|}{(r^2 + (\nu_\varepsilon^2-x_2)^2 + (t_\varepsilon-t))^5} \\
\leq C (t_\varepsilon-t)^{1/2+2\lambda} \frac{1}{\left( (x_2-\nu_\varepsilon^2)^2 + (t_\varepsilon-t) \right)^{13/4-\lambda}}.
\end{multline}

Now, integrating (\ref{mmmmmm}) with respect to $x_2$ yields

\begin{equation}
\int_{x_2} |\psi_\varepsilon^2(x_2)| |(v)_1| \leq \frac{C}{(t_\varepsilon-t)^{7/4-3\lambda}},
\end{equation}

which is of a smaller order.

For $(v)_2$, we obtain

\begin{multline}\label{mmmmmmm}
(v)_2 \leq C (t_\varepsilon-t)^{1/2+2\lambda} \int_r \frac{r^{1.5+2\lambda} |x_2-\nu_\varepsilon^2|}{\left( r^2 + (x_2 - \nu_\varepsilon^2)^2 + (t_\varepsilon-t) \right)^3} \\
\cdot \int_{z=0} \frac{1}{\left( z^2 + (x_2-\nu_\varepsilon^2)^2 + (t_\varepsilon-t) \right)^2}\\
\leq C (t_\varepsilon-t)^{1/2+2\lambda} \int_r \frac{r^{1.5+2\lambda} |x_2-\nu_\varepsilon^2|}{\left( r^2 + (x_2 - \nu_\varepsilon^2)^2 + (t_\varepsilon-t) \right)^3} 
\cdot \frac{C}{((x_2-\nu_\varepsilon^2)^2 + (t_\varepsilon-t))^{3/2}}\\
\leq C(t_\varepsilon-t)^{1/2+2\lambda} \frac{1}{\left( (x_2-\nu_\varepsilon^2)^2 + (t_\varepsilon-t) \right)^{13/4}}.
\end{multline}

Now, integrating (\ref{mmmmmmm}) th respect to $x_2$ yields

\begin{equation}
\int_{x_2} |\psi_\varepsilon^2(x_2)| |(v)_2| \leq \frac{C}{(t_\varepsilon-t)^{7/4-3\lambda}},
\end{equation}

which is of a smaller order. Lastly, we go on with $(vii)$.

\begin{multline}
(vii) \leq \frac{C}{(t_\varepsilon-t)^{1/2-2\lambda}} \\
\cdot  \int_{r} \int_{z=0}^r \frac{ |x_2-\nu_\varepsilon^2| (t_\varepsilon-t)^2}{z^{1/2-2\lambda} \left( (x_2-\nu_\varepsilon^2)^2 + (t_\varepsilon-t) \right)^3 \left( (r-z)^2 + (x_2-\nu_\varepsilon^2)^2 + (t_\varepsilon-t) \right)^2} \\
\leq C (t_\varepsilon-t)^{3/2+2\lambda} \int_r \int_{z=0}^{r/2} \frac{ |x_2-\nu_\varepsilon^2| }{z^{1/2-2\lambda} \left( r^2 +(x_2-\nu_\varepsilon^2)^2 + (t_\varepsilon-t) \right)^5} \\
+ C (t_\varepsilon-t)^{3/2+2\lambda} \int_{r}\int_{z=r/2}^r \frac{ |x_2-\nu_\varepsilon|}{r^{1/2-2\lambda} \left( r^2 + (x_2-\nu_\varepsilon^2)^2 + (t_\varepsilon-t) \right)^3 \left( (r-z)^2 + (x_2-\nu_\varepsilon^2)^2 + (t_\varepsilon-t) \right)^2 }\\
= (vii)_1 + (vii)_2.
\end{multline}

For $(v)_1$, we get

\begin{multline}\label{mmmmmmmm}
(vii)_1 \leq C (t_\varepsilon-t)^{3/2+2\lambda} \int_r \frac{r^{0.5+2\lambda} |x_2-\nu_\varepsilon|}{(r^2 + (\nu_\varepsilon^2-x_2)^2 + (t_\varepsilon-t))^5} \\
\leq C (t_\varepsilon-t)^{3/2+2\lambda} \frac{1}{\left( (x_2-\nu_\varepsilon^2)^2 + (t_\varepsilon-t) \right)^{17/4-\lambda}}.
\end{multline}

Now, integrating (\ref{mmmmmmmm}) with respect to $x_2$ yields

\begin{equation}
\int_{x_2} |\psi_\varepsilon^2(x_2)| |(vii)_1| \leq \frac{C}{(t_\varepsilon-t)^{13/4-3\lambda -3/2}} \leq \frac{C}{(t_\varepsilon-t)^{7/4-3\lambda}},
\end{equation}

which is of a smaller order.

For $(vii)_2$, we obtain

\begin{multline}\label{mmmmmmmmm}
(vii)_2 \leq C (t_\varepsilon-t)^{3/2+2\lambda} \int_r \frac{ |x_2-\nu_\varepsilon^2|}{r^{1/2-2\lambda}\left( r^2 + (x_2 - \nu_\varepsilon^2)^2 + (t_\varepsilon-t) \right)^3} \\
\cdot \int_{z=0} \frac{1}{\left( z^2 + (x_2-\nu_\varepsilon^2)^2 + (t_\varepsilon-t) \right)^2}\\
\leq C (t_\varepsilon-t)^{3/2+2\lambda} \int_r \frac{r^{-1/2+2\lambda} |x_2-\nu_\varepsilon^2|}{\left( r^2 + (x_2 - \nu_\varepsilon^2)^2 + (t_\varepsilon-t) \right)^3} 
\cdot \frac{C}{((x_2-\nu_\varepsilon^2)^2 + (t_\varepsilon-t))^{3/2}}\\
\leq C(t_\varepsilon-t)^{3/2+2\lambda} \frac{1}{\left( (x_2-\nu_\varepsilon^2)^2 + (t_\varepsilon-t) \right)^{17/4}}.
\end{multline}

Now, integrating (\ref{mmmmmmmmm}) with respect to $x_2$ yields

\begin{equation}
\int_{x_2} |\psi_\varepsilon^2(x_2)| |(vii)_2| \leq \frac{C}{(t_\varepsilon-t)^{7/4-3\lambda}},
\end{equation}

which is again of a smaller order.

\newpage

Now, we have established the lower bound for $(iv)_1$ in (\ref{touslesi}). The last step is to establish an upper bound for $(iv)_2$, $(iv)_3$ and $(iv)_4$ to conclude the proof. The case of $(iv)_2$ and $(iv)_3$ are similar. We start with $(iv)_2$.

We decompose as follows

\begin{multline}\label{desalphaslol}
(iv)_2 \leq C_\varepsilon \sum_{i+j+k=2} \int\int_\iota \frac{(x_1)^i |x_2-\nu_\varepsilon^2|^j (t_\varepsilon-t)^k}{\left| \phi(t,x_1,x_2)-\phi(t,y,x_2) \right|^{1/2-2\lambda}} \cdot \frac{|y| + |x_2-\nu_\varepsilon^2| + (t_\varepsilon-t) }{\phi_{x_1}(t,x_1,x_2)^2 \phi_{x_1}(t,y,x_2)^2} \\
= \alpha_1 + \alpha_2 + \alpha_3.
\end{multline}

We obtain

\begin{multline}\label{desalphaslol1}
\alpha_1 \leq C_\varepsilon \sum_{i+j+k=2} \int\int_\iota \frac{(x_1)^i |x_2-\nu_\varepsilon^2|^j (t_\varepsilon-t)^k}{\left| \phi(t,x_1,x_2)-\phi(t,y,x_2) \right|^{1/2-2\lambda}} \cdot \frac{|y|}{\phi_{x_1}(t,x_1,x_2)^2 \phi_{x_1}(t,y,x_2)^2} \\
\leq (t_\varepsilon-t)^{k-1/2+2\lambda} |x_2-\nu_\varepsilon^2|^j \\
\cdot \int\int_\iota \frac{(x_1)^i\cdot |y|}{\left| x_1 - y \right|^{1/2-2\lambda}\cdot \phi_{x_1}(t,x_1,x_2)^2 \phi_{x_1}(t,y,x_2)^2}
\end{multline}

Now, we use H\"older in (\ref{desalphaslol1}), 

\begin{multline}\label{encoredesalphaslol}
\alpha_1 \leq \sum_{i+j+k=2} C_\varepsilon (t_\varepsilon-t)^{k-1/2+2\lambda} |x_2-\nu_\varepsilon^2|^j \left( \int\int_\iota \frac{1}{|x_1-y|^{3/4-3\lambda}} \right)^{2/3}\\
 \cdot \left( \int\int_\iota \frac{(x_1)^{3i}}{ \left( x_1^2 + (x_2-\nu_\varepsilon^2)^2 + (t_\varepsilon-t) \right)^6 } \right)^{1/3}
 \cdot \left( \int\int_\iota \frac{(y)^{3}}{ \left( x_1^2 + (x_2-\nu_\varepsilon^2)^2 + (t_\varepsilon-t) \right)^6 } \right)^{1/3} \\
 \leq \sum_{i+j+k=2} C_\varepsilon \frac{(t_\varepsilon-t)^{k-1/2+2\lambda} |x_2-\nu_\varepsilon|^j}{ \left( (x_2-\nu_\varepsilon)^2 + (t_\varepsilon-t) \right)^{19/6-i/2}}.
\end{multline}

To do this, we studied the two last integrals. In particular, we give the following estimates.

\begin{equation}\label{qqintegrales}
\begin{aligned}
&\left( \int_0^\infty \frac{1}{(x^2+a)^6}\right)^{1/3} \leq \frac{C_\varepsilon}{a^{11/6}}, \\
&\left( \int_0^\infty \frac{x^3}{(x^2+a)^6}\right)^{1/3} \leq \frac{C_\varepsilon}{a^{4/3}}, \\
&\left( \int_0^\infty \frac{x^6}{(x^2+a)^6}\right)^{1/3} \leq \frac{C_\varepsilon}{a^{5/6}}. \\
\end{aligned}
\end{equation}

Integrating (\ref{encoredesalphaslol}) with respect to $x_2$ yields

\begin{equation}
\int_{x_2} \alpha_1 \leq \sum_{i+j+k=2} \frac{1}{(t_\varepsilon-t)^{32/12 - i/2 - j/2-k +1/2 -2 \lambda}}.
\end{equation}

Given the fact that $i+j+k = 2$, the highest order possible is $26/12$, which is indeed smaller than $9/4=27/12$. We will now go through $\alpha_2$ and $\alpha_3$. Because they are somehow similar, we will skip a few details.

For $\alpha_2$, we obtain

\begin{multline}\label{encoredesalphaslol2}
\alpha_2 \leq \sum_{i+j+k=2} C_\varepsilon (t_\varepsilon-t)^{k-1/2+2\lambda} |x_2-\nu_\varepsilon^2|^{j+1} \left( \int\int_\iota \frac{1}{|x_1-y|^{3/4-3\lambda}} \right)^{2/3}\\
 \cdot \left( \int\int_\iota \frac{(x_1)^{3i}}{ \left( x_1^2 + (x_2-\nu_\varepsilon^2)^2 + (t_\varepsilon-t) \right)^6 } \right)^{1/3}
 \cdot \left( \int\int_\iota \frac{1}{ \left( x_1^2 + (x_2-\nu_\varepsilon^2)^2 + (t_\varepsilon-t) \right)^6 } \right)^{1/3} \\
 \leq \sum_{i+j+k=2} C_\varepsilon \frac{(t_\varepsilon-t)^{k-1/2+2\lambda} |x_2-\nu_\varepsilon|^{j+1}}{ \left( (x_2-\nu_\varepsilon)^2 + (t_\varepsilon-t) \right)^{22/6-i/2}}.
\end{multline}

Integrating (\ref{encoredesalphaslol2}) with respect to $x_2$ yields

\begin{equation}
\int_{x_2} \alpha_2 \leq \sum_{i+j+k=2} \frac{1}{(t_\varepsilon-t)^{19/6 - i/2 - (j+1)/2-k +1/2 -2 \lambda}}.
\end{equation}

Given the fact that $i+j+k = 2$, the highest order possible is $26/12$, which is indeed smaller than $9/4=27/12$. We will now go through $\alpha_2$ and $\alpha_3$. Because they are somehow similar, we will skip a few details.

For $\alpha_3$, we obtain

\begin{multline}\label{encoredesalphaslol3}
\alpha_2 \leq \sum_{i+j+k=2} C_\varepsilon (t_\varepsilon-t)^{k+1-1/2+2\lambda} |x_2-\nu_\varepsilon^2|^{j} \left( \int\int_\iota \frac{1}{|x_1-y|^{3/4-3\lambda}} \right)^{2/3}\\
 \cdot \left( \int\int_\iota \frac{(x_1)^{3i}}{ \left( x_1^2 + (x_2-\nu_\varepsilon^2)^2 + (t_\varepsilon-t) \right)^6 } \right)^{1/3}
 \cdot \left( \int\int_\iota \frac{1}{ \left( x_1^2 + (x_2-\nu_\varepsilon^2)^2 + (t_\varepsilon-t) \right)^6 } \right)^{1/3} \\
 \leq \sum_{i+j+k=2} C_\varepsilon \frac{(t_\varepsilon-t)^{k+1-1/2+2\lambda} |x_2-\nu_\varepsilon|^{j}}{ \left( (x_2-\nu_\varepsilon)^2 + (t_\varepsilon-t) \right)^{22/6-i/2}}.
\end{multline}

Integrating (\ref{encoredesalphaslol3}) with respect to $x_2$ yields

\begin{equation}
\int_{x_2} \alpha_2 \leq \sum_{i+j+k=2} \frac{1}{(t_\varepsilon-t)^{19/6 - i/2 - j/2-k-1 +1/2 -2 \lambda}}.
\end{equation}

Given the fact that $i+j+k = 2$, the highest order possible is $20/12$, which is indeed smaller than $9/4=27/12$. We will now go through $\alpha_2$ and $\alpha_3$. Because they are somehow similar, we will skip a few details.

We now look at $(iv)_4$, the expression being given by (\ref{touslesi}). We establish an upper bound.

\begin{multline}\label{leli4}
(iv)_4 \leq C_\varepsilon \sum\limits_{\substack{i_1+j_1+k_1=2 \\ i_2+j_2+k_2=2}} (t_\varepsilon-t)^{k_1+k_2} \int\int_\iota \frac{|x_1|^{i_1} |y|^{i_2} |x_2-\nu_\varepsilon^2|^{j_1+j_2}}{|\phi(t,\nu_\varepsilon^1+x_1,x_2) - \phi(t,\nu_\varepsilon^1+y,x_2)|^{1/2-2\lambda} \phi_{x_1}^2(x_1) \phi_{x_1}^2(y)} \\
\leq  C_\varepsilon \sum\limits_{\substack{i_1+j_1+k_1=2 \\ i_2+j_2+k_2=2}} \frac{(t_\varepsilon-t)^{k_1+k_2}}{(t_\varepsilon-t)^{1/2-2\lambda}} \int\int_\iota \frac{|x_1|^{i_1} |y|^{i_2} |x_2-\nu_\varepsilon^2|^{j_1+j_2}}{|x_1-y|^{1/2-2\lambda}\phi_{x_1}^2(x_1) \phi_{x_1}^2(y)}
\end{multline}

Now, we use H\"older's inequality in (\ref{leli4}) and simplify, 

\begin{multline}\label{lelp}
(iv)_4 \leq C_\varepsilon  \sum\limits_{\substack{i_1+j_1+k_1=2 \\ i_2+j_2+k_2=2}} (t_\varepsilon-t)^{k_1+k_1-1/2+2\lambda} |x_2-\nu_\varepsilon^2|^{j_1+j_2} \int\int_\iota \frac{|x_1|^{i_1} |y|^{i_2}}{|x_1-y|^{1/2-2\lambda}\phi_{x_1}^2(x_1) \phi_{x_1}^2(y)}\\
\leq \sum\limits_{\substack{i_1+j_1+k_1=2 \\ i_2+j_2+k_2=2}}  C_\varepsilon (t_\varepsilon-t)^{k_1+k_1-1/2+2\lambda} |x_2-\nu_\varepsilon^2|^{j_1+j_2} \left( \int\int_\iota \frac{1}{|x_1-y|^{3/4-3\lambda}} \right)^{2/3} \\
\cdot \left( \int_0^\infty \frac{x_1^{3i_1}}{\left( x_1^2 + (x_2-\nu_\varepsilon^2)^2 + (t_\varepsilon-t) \right)^6} \right)^{1/3} \cdot \left( \int_0^\infty \frac{y^{3i_2}}{\left( y^2 + (x_2-\nu_\varepsilon^2)^2 + (t_\varepsilon-t) \right)^6} \right)^{1/3}
\end{multline}

From (\ref{lelp}) and (\ref{qqintegrales}), we have 

\begin{multline}\label{lelp2}
(iv)_4 \leq \sum\limits_{\substack{i_1+j_1+k_1=2 \\ i_2+j_2+k_2=2}}  \frac{ C_\varepsilon (t_\varepsilon-t)^{k_1+k_1-1/2+2\lambda} |x_2-\nu_\varepsilon^2|^{j_1+j_2}}{\left( (t_\varepsilon-t) + (\nu_\varepsilon^2-x_2)^2\right)^{\frac{11}{6}-\frac{3i_1}{6}} \left( (t_\varepsilon-t) + (\nu_\varepsilon^2-x_2)^2\right)^{\frac{11}{6}-\frac{3i_2}{6}}}\\
\leq \sum\limits_{\substack{i_1+j_1+k_1=2 \\ i_2+j_2+k_2=2}}  \frac{ C_\varepsilon (t_\varepsilon-t)^{k_1+k_1-1/2+2\lambda} |x_2-\nu_\varepsilon^2|^{j_1+j_2}}{\left( (t_\varepsilon-t) + (\nu_\varepsilon^2-x_2)^2\right)^{\frac{22}{6}-3 \frac{i_1+i_2}{6}} } \\
\leq \sum\limits_{\substack{i_1+j_1+k_1=2 \\ i_2+j_2+k_2=2}} \frac{C_\varepsilon (t_\varepsilon-t)^{k_1+k_2-1/2+2\lambda}}{\left( (t_\varepsilon-t) + (\nu_\varepsilon^2-x_2)^2 \right)^{\frac{22}{6}- \frac{i_1+j_1+i_2+j_2}{2}}}
\end{multline}

We will distinguish the different cases.
First, if $k_1+k_2=0$, then $i_1+i_2+j_1+j_2=4$ and we have

\begin{equation}\label{encoreune}
\frac{C_\varepsilon (t_\varepsilon-t)^{k_1+k_2-1/2+2\lambda}}{\left( (t_\varepsilon-t) + (\nu_\varepsilon^2-x_2)^2 \right)^{\frac{22}{6}- \frac{i_1+j_1+i_2+j_2}{2}}} = \frac{C_\varepsilon}{(t_\varepsilon-t)^{1/2-2\lambda}} \frac{1}{\left( (t_\varepsilon-t) + (\nu_\varepsilon^2-x_2)^2 \right)^{10/6}},
\end{equation}

hence, integrating (\ref{encoreune}) yields

\begin{multline}\label{encoreune1}
\left| \int_{x_2=\nu_\varepsilon^2-2\delta_\varepsilon}^{\nu_\varepsilon^2+2\delta_\varepsilon} \psi_\varepsilon^2(x_2) \frac{C_\varepsilon}{(t_\varepsilon-t)^{1/2-2\lambda}} \frac{1}{\left( (t_\varepsilon-t) + (\nu_\varepsilon^2-x_2)^2 \right)^{10/6}} \right| \\
\leq \frac{C_\varepsilon}{(t_\varepsilon-t)^{7/6+1/2-2\lambda}} << \frac{1}{(t_\varepsilon-t)^{9/4-3\lambda}}.
\end{multline}

If $k_1+k_2=1$, then $i_1+i_2+j_1+j_2=3$ and we have

\begin{equation}\label{encoreune2}
\frac{C_\varepsilon (t_\varepsilon-t)^{k_1+k_2-1/2+2\lambda}}{\left( (t_\varepsilon-t) + (\nu_\varepsilon^2-x_2)^2 \right)^{\frac{22}{6}- \frac{i_1+j_1+i_2+j_2}{2}}} =  \frac{C_\varepsilon (t_\varepsilon-t)^{1/2+2\lambda}}{\left( (t_\varepsilon-t) + (\nu_\varepsilon^2-x_2)^2 \right)^{13/6}},
\end{equation}

hence, integrating (\ref{encoreune2}) yields

\begin{multline}\label{encoreune3}
\left| \int_{x_2=\nu_\varepsilon^2-2\delta_\varepsilon}^{\nu_\varepsilon^2+2\delta_\varepsilon} \psi_\varepsilon^2(x_2) \frac{C_\varepsilon (t_\varepsilon-t)^{1/2+2\lambda}}{\left( (t_\varepsilon-t) + (\nu_\varepsilon^2-x_2)^2 \right)^{13/6}} \right| \\
\leq \frac{C_\varepsilon}{(t_\varepsilon-t)^{7/6-2\lambda}} << \frac{1}{(t_\varepsilon-t)^{9/4-3\lambda}}.
\end{multline}

If $k_1+k_2=2$, then $i_1+i_2+j_1+j_2=2$ and we have

\begin{equation}\label{encoreune4}
\frac{C_\varepsilon (t_\varepsilon-t)^{k_1+k_2-1/2+2\lambda}}{\left( (t_\varepsilon-t) + (\nu_\varepsilon^2-x_2)^2 \right)^{\frac{22}{6}- \frac{i_1+j_1+i_2+j_2}{2}}} =  \frac{C_\varepsilon (t_\varepsilon-t)^{3/2+2\lambda}}{\left( (t_\varepsilon-t) + (\nu_\varepsilon^2-x_2)^2 \right)^{16/6}},
\end{equation}

hence, integrating (\ref{encoreune4}) yields

\begin{multline}\label{encoreune5}
\left| \int_{x_2=\nu_\varepsilon^2-2\delta_\varepsilon}^{\nu_\varepsilon^2+2\delta_\varepsilon} \psi_\varepsilon^2(x_2) \frac{C_\varepsilon (t_\varepsilon-t)^{3/2+2\lambda}}{\left( (t_\varepsilon-t) + (\nu_\varepsilon^2-x_2)^2 \right)^{16/6}} \right| \\
\leq \frac{C_\varepsilon}{(t_\varepsilon-t)^{4/6-2\lambda}} << \frac{1}{(t_\varepsilon-t)^{9/4-3\lambda}}.
\end{multline}

Now, we deduce from the same computations that the term corresponding to $k_1+k_2=3$ is smaller than $\frac{C_\varepsilon}{(t_\varepsilon-t)^{1/6}}$ and that the term corresponding to $k_1+k_2=4$ is bounded.

Now, we use (\ref{lelp2}) to conclude for the term corresponding to $(iv)_4$ in (\ref{desix2}).

\begin{equation}
\left| \int_{x_2=\nu_\varepsilon^2-2\delta_\varepsilon}^{\nu_\varepsilon^2+2\delta_\varepsilon} \psi_\varepsilon^2(x_2) (iv)_4\right| << \frac{C_\varepsilon}{(t_\varepsilon-t)^{9/4-3\lambda}}.
\end{equation}

This concludes this proof.

\end{proof}

\subsection{Introducing a perturbation of the initial data}

In this chapter, we will study the stability of the ill-posedness of our model equation with respect to a perturbation of the initial conditions. We will not go into the full proof, but rather try to convince the reader that the one can write a similar proof as to the one we did in the previous section, after having established some preliminary results. Note that the case we are studying here is not the full stability of the blow-up with respect to a perturbation of the initial condition of the problem (\ref{modelequation}), as we did not introduce the term $\partial_{x_2}^2$ here, which means it is not of the form (\ref{QLW}).
More precisely, we consider the Cauchy problem

\begin{equation}\label{cellecimodifiee1}
\left\{
\begin{aligned}
&\Box u = Du D^2 u + f(t,x_1,x_2,u) , \\
&u_{|t=0} = \tilde{u}_0, ~~ \frac{\partial u}{\partial t}_{|t=0} =-\chi(x_1) + \tilde{u}_1\\
\end{aligned},
\right. ~~~(MQLW2)
\end{equation}

and the corresponding regularized Cauchy problem

\begin{equation}\label{cellecimodifiee2}
\left\{
\begin{aligned}
&\Box u = Du D^2 u + f(t,x_1,x_2,u) , \\
&u_{|t=0} = \tilde{u}_0, ~~ \frac{\partial u}{\partial t}_{|t=0} =-\chi_\varepsilon(x_1) + \tilde{u}_1,\\
\end{aligned}
\right. ~~~(MQLW2)_\varepsilon
\end{equation}

where $\chi_\varepsilon$ is defined as in (\ref{chilol}). We will denote $D\tilde{u}(x_1,x_2) = \tilde{v}(x_1,x_2) = \partial_{x_1} \tilde{u}_0 - \tilde{u_1}.$ We also assume the two following conditions

\begin{equation}\label{cellecimodifiee3}
\begin{aligned}
&(i)~ \forall \alpha, ~ \left| \frac{\partial^\alpha f}{\partial_x^\alpha} \right| \leq C,\\
&(ii)~ (\tilde{u}_0,\tilde{u}_1) \in H^{4}(\mathbb{R}^2) \times H^{3}(\mathbb{R}^2).
\end{aligned}
\end{equation}

In this chapter, we do not go through all the previous computations as they will still hold. Instead, we are going to prove an equivalent of lemma \ref{laconditionimportante} and the result will follow.

Now, rewriting (\ref{cellecimodifiee2}) with $Du=v$ yields

\begin{equation} \label{cellecimodifiee4}
\left\{
\begin{aligned}
&\left( \partial_t + \frac{1+v}{1-v} \partial_{x_1} \right) v  = \frac{f(t,x_1,x_2,v)}{1-v} = g(t,x_1,x_2,v). \\
&\frac{\partial u}{\partial t}_{|t=0} = - \chi + \tilde{u}_1, \hspace{0.2cm} u_{|t=0} = \tilde{u}_0,
\end{aligned}
\right.
\end{equation}

where $g$ also satisfies $(i)$ of (\ref{cellecimodifiee3}) as $v<0$. We now state the main theorem of this chapter.

\begin{theorem}\label{quatriemecas}

Let $u_\varepsilon$ be the solution to problem (\ref{cellecimodifiee2}).
There exists a time $t_\varepsilon$ such that 
\begin{equation}
\left\{
\begin{aligned}
&||u_\varepsilon(0,\cdot)||_{H^{11/4}(\mathbb{R}^2)} < \infty \\
&||u_\varepsilon(t,\cdot)||_{H^{11/4}(\mathbb{R}^2)} \rightarrow \infty \hspace{0.2cm} \text{ as } t\rightarrow t_\varepsilon.
\end{aligned}
\right.
\end{equation}

Also, we have that 

\begin{equation}
t_\varepsilon \rightarrow 0 \hspace{0.2cm} \text{ as } \varepsilon \rightarrow 0.
\end{equation}

\end{theorem}

We first do some preliminary work, and later state and prove an equivalent of lemma \ref{laconditionimportante} for this case.

\textbf{Preliminary work}

As in \cite{ohlmann2021illposedness}, we define $\phi$ (implicitly depending on $\varepsilon$) the same way as for the unperturbed problem,

\begin{equation}\label{cellecimodifiee5}
\left\{
\begin{aligned}
&\phi(0,x_1,x_2) = x_1 \\
& \partial_t \phi(t,x_1,x_2) = \frac{1+v(t,\phi(t,x_1,x_2),x_2)}{1-v(t,\phi(t,x_1,x_2),x_2)}.
\end{aligned}
\right.
\end{equation}

From (\ref{cellecimodifiee4}) and (\ref{defphiproblem3}), we get

\begin{equation}\label{cellecimodifiee6}
\frac{\partial}{\partial_t} \left(v(t,\phi(t,x_1,x_2),x_2) \right) = g(t,\phi(t,x_1,x_2),x_2,v(t,\phi(t,x_1,x_2),x_2)).
\end{equation}

From now on, we will not specify every variable, but only if the function is applied to $x_1$ of $\phi(t,x_1,x_2)$. Integrating (\ref{cellecimodifiee6}) yields

\begin{equation}\label{cellecimodifiee7}
v(\phi)= \chi(x_1) + \tilde{v}(x_1,x_2) + \int_{\tau=0}^t g(\tau,\phi) d\tau. 
\end{equation}

Differentiating (\ref{cellecimodifiee7}) leads to the following expressions for the derivatives of $v$.

\begin{multline}\label{modifiee7}
\partial_{x_1} \phi(x_1) \partial_{x_1} v(\phi) = \chi'(x_1) + \partial_{x_1} \tilde{v} + \int_\tau \partial_{x_1} \phi(\tau,x_1) \partial_1 g(\tau,x_1) + \int_\tau \partial_{x_1} \phi(\tau,x_1) \partial_{x_1} v(\tau,x_1) \partial_3 g(\tau,x_1) \\
\Rightarrow \partial_{x_1} v(\phi)= \frac{1}{\partial_{x_1} \phi(\tau,x_1)} \Big[ \chi'(x_1) + \partial_{x_1} \tilde{v} \\ + \int_\tau \partial_{x_1} \phi(\tau,x_1) \partial_1 g(\tau,\phi) + \int_\tau \partial_{x_1} \phi(x_1) \partial_{x_1} v(\tau,x_1) \partial_3 g(\tau,x_1)  \Big].
\end{multline}

\begin{multline}\label{modifiee8}
\partial_{x_1}^2 v(\phi) = \frac{1}{\left( \partial_{x_1} \phi(x_1) \right)^2} \Big[ \chi''(x_1) +\partial_{x_1}^2 \tilde{v} + \int_{\tau} \left( \partial_{x_1} \phi(\tau,x_1) \right)^2\partial_{1}^2 g(\tau,\phi) + \int_\tau \partial_{x_1}^2 \phi(\tau,x_1) \partial_{1} g(\tau,\phi) \\
+ \int_\tau \left( \partial_{x_1} \phi(\tau,x_1) \right)^2 \partial_{x_1} v(\tau,\phi) \partial_3 \partial_1 g(\tau,\phi) + \int_\tau \partial_{x_1}^2 \phi(\tau,x_1) \partial_3 g(\tau,\phi) \partial_{x_1} v(\tau,\phi) \\
+ \int_\tau \left( \partial_{x_1} \phi(\tau,x_1) \right)^2 \partial_1 \partial_3 g(\tau,\phi) \partial_{x_1} v(\tau,\phi)  
+ \int_\tau \left( \partial_{x_1} \phi(\tau,x_1) \right)^2 \left( \partial_{x_1} v(\tau,\phi) \right)^2 \partial_3^2 g(\tau,\phi)  \\+ \int_\tau \left( \partial_{x_1} \phi(\tau,x_1) \right)^2  \partial_{x_1}^2 v(\tau,\phi) \partial_3 g(\tau,\phi)  \Big] \\
- \frac{ \partial_{x_1}^2 \phi(x_1)}{\left( \partial_{x_1} \phi(x_1) \right)^3 } \Big[ \chi'(x_1) + \partial_{x_1} \tilde{v}+  \int_\tau \partial_{x_1} \phi(\tau,x_1) \partial_1 g(\tau,\phi) + \int_\tau \partial_{x_1} \phi(\tau,x_1) \partial_3 g(\tau,\phi) \partial_{x_1} v(\tau,\phi) \Big]\\
= A - B.
\end{multline}

Again, we have

\begin{equation}\label{modifiee9}
\phi_{tx_1}(x_1) = \frac{2 v_{x_1}(\phi)}{(1-v(\phi))^2}.
\end{equation}

\begin{equation}\label{modifiee10}
\phi_{tx_1x_1}(x_1) = \frac{2 v_{x_1x_1}(\phi) (1-v(\phi)) + 2 v_{x_1}^2(\phi)}{(1-v(\phi))^3}.
\end{equation}

We now state the lemma. It is in fact identical to lemma \ref{laconditionimportante}

\begin{lemma}\label{modifiee11}
There exists a time $t_\varepsilon$ such that the following properties are verified.

\begin{equation}\label{teps1lol}
t_\varepsilon \rightarrow 0 \hspace{0.2cm} \text{ as } \varepsilon \rightarrow 0. 
\end{equation}

\begin{equation}\label{teps2lol}
\begin{aligned}
&\phi_{x_1} > 0 \hspace{0.2cm} \forall t<t_\varepsilon\\
&\exists (\nu_\varepsilon^1, \nu_\varepsilon^2) \hspace{0.2cm} s.t. \hspace{0.2cm} \phi_{x_1}(t_\varepsilon,\nu_\varepsilon^1,\nu_\varepsilon^2) = 0
\end{aligned}
\end{equation}

\begin{equation}
\begin{aligned}
&0 < \phi_{x_1} < 1, \hspace{0.1cm} \phi_{t,x_1} <0 \hspace{0.2cm}, \forall t<t_\varepsilon, \\
&\exists C_\varepsilon, \hspace{0.2cm} |\phi_{t,x_1,x_1}| \leq C_\varepsilon |\ln(t_\varepsilon-t)|. \\
&v_{x_1} <0, \hspace{0.2cm} v<0\\
&\exists C_\varepsilon, \hspace{0.2cm} |v_{x_1}(\phi)| \leq \frac{C_\varepsilon(\chi'(x_1) + C_\varepsilon)}{|\phi_{x_1}|}.
\end{aligned}
\end{equation}

If $x_1$, $x_2$ and $t$ are sufficiently close to $(\nu_\varepsilon^1,\nu_\varepsilon^2,t_\varepsilon)$, we have the following estimates.

\begin{multline}\label{modifiee12}
\exists C_\varepsilon^1,C_\varepsilon^2 >0, \hspace{0.2cm} \frac{C_\varepsilon^1}{(x_1-\nu_\varepsilon^1)^2 + (x_2-\nu_\varepsilon^2)^2 + (t_\varepsilon-t)} \leq \frac{1}{\phi_{x_1}(x_1)} \\ \leq 
\frac{C_\varepsilon^2}{(x_1-\nu_\varepsilon^1)^2 + (x_2-\nu_\varepsilon^2)^2 + (t_\varepsilon-t)}\\
\exists C_\varepsilon^1, C_\varepsilon^2, C_\varepsilon^3, \hspace{0.2cm} \phi_{x_1x_1}(x_1) = C_\varepsilon^1 (x_1-\nu_\varepsilon^1) + C_\varepsilon^2 (x_2-\nu_\varepsilon^2) + C_\varepsilon^3 (t_\varepsilon-t) \\
+ \sum_{i+j+k=2} (x_1-\nu_\varepsilon^1)^i (x_2-\nu_\varepsilon^2)^j (t_\varepsilon-t)^k f_{i,j,k}(t,x_1,x_2),
\end{multline}

where all the involved $f$ functions are bounded near $(t,\nu_\varepsilon^1,\nu_\varepsilon^2)$. We also have

\begin{equation}\label{superestA12}
|A| \leq \frac{C_\varepsilon |\ln(t_\varepsilon-t)|}{|\phi_{x_1}|^2},
\end{equation}

where $A$ is the $A$ involved in (\ref{modifiee8}).

Lastly, for $\varepsilon$ small enough, we have 

\begin{equation}\label{superestB13}
\begin{aligned}
&\frac{C_\varepsilon \phi_{x_1x_1}(x_1)}{\left( \phi_{x_1}(x_1) \right)^3} \frac{\chi'}{2} \leq B \leq \frac{C_\varepsilon \phi_{x_1x_1}(x_1)}{\left( \phi_{x_1}(x_1) \right)^3} 2\chi', \\
&\frac{C_\varepsilon \phi_{x_1x_1}(x_1)}{\left( \phi_{x_1}(x_1) \right)^3} 2\chi' \leq B \leq \frac{C_\varepsilon \phi_{x_1x_1}(x_1)}{\left( \phi_{x_1}(x_1) \right)^3} \frac{\chi'}{2}.
\end{aligned}
\end{equation}

The first inequality being verified when $\phi_{x_1x_1} \leq 0$, and the second being verified when $\phi_{x_1x_1} \geq 0$.

\end{lemma}

We stop here for this case. The rest of the is identical to what have been done when we added the source term $f(t,x_1,x_2,\nabla u)$. 

\newpage

\begin{center}
	\textbf{Acknowledgements}
\end{center}

\emph{We thank Professor Joachim Krieger at EPFL and Professor Enno Lenzmann for their help and proofreading. The research was funded by EPFL and University of Basel. Since 01/10/2021, the research is supported by the Swiss National Science Fundation (SNSF) through Grant number 200020$\_$204121.}

\bibliographystyle{alpha}
\bibliography{bibliography}

\end{document}